\documentclass[11pt, reqno]{amsart}    

\usepackage{amsmath,amssymb,latexsym}
\usepackage{color}
\usepackage{xcolor}
\usepackage{graphicx}
\usepackage{cite}
\usepackage[colorlinks=true]{hyperref}
\hypersetup{urlcolor=blue, citecolor=red, linkcolor=blue}

\newcommand{\R}{\mathbb{R}}

\newtheorem{theorem}{Theorem}[section]
\newtheorem{lemma}[theorem]{Lemma}
\newtheorem{proposition}[theorem]{Proposition}
\newtheorem{remark}[theorem]{Remark}
\newtheorem{conjecture}[theorem]{Conjecture}

\newtheorem*{main-theorem}{Main Theorem}
\newtheorem*{remark*}{Remark}
\newtheorem*{lemma*}{Lemma A.1}

\numberwithin{equation}{section}

\begin{document}

\title[Modified fractional Korteweg-de Vries equation]{On the   modified fractional Korteweg-de Vries and related equations}

\author{Christian Klein}

\author{Jean-Claude Saut}

\author{Yuexun Wang}

\address{Institut de Math\'ematiques de Bourgogne,
                Universit\'e de Bourgogne, 9 avenue Alain Savary, 21078 Dijon
                Cedex, France}
    \email{Christian.Klein@u-bourgogne.fr}

\address{ Universit\' e Paris-Saclay, CNRS, Laboratoire de Math\'  ematiques d'Orsay, 91405 Orsay, France.}
\email{jean-claude.saut@universite-paris-saclay.fr}

\address{	
	School of Mathematics and Statistics, 
	Lanzhou University, 370000 Lanzhou, China.}
\address{Universit\' e Paris-Saclay, CNRS, Laboratoire de Math\'  ematiques d'Orsay, 91405 Orsay, France.}

\email{yuexun.wang@universite-paris-saclay.fr}

\thanks{}


\begin{abstract}  We consider in this paper modified fractional Korteweg-de Vries and related equations (modified Burgers-Hilbert and Whitham). They have the advantage with respect to the usual fractional KdV equation to have a defocusing case with a different dynamics. 
We will  distinguish the weakly dispersive case where the phase 
velocity is unbounded for low frequencies and tends to zero at 
infinity  and the strongly dispersive case where the phase velocity 
vanishes at the origin and goes to infinity at infinity. In the 
former case, the nonlinear hyperbolic effects dominate for large 
data, leading to the possibility of shock formation though the 
dispersive effects manifest for small initial data where scattering 
is possible. In the latter case, finite time blow-up is possible in the focusing case but not the shock formation. In the defocusing case global existence and scattering is expected in the energy subcritical case, while finite time blow-up is expected in the energy supercritical case.

We establish rigorously the existence of shocks with blow-up time and location being explicitly computed in the weakly dispersive case, while most of the results on the strongly dispersive case are derived via numerical simulations, for large solutions. 
Moreover, the shock formation result can be extended to the weakly dispersive equation with some generalized nonlinearity.

We will also comment briefly on the BBM versions of those equations.

\end{abstract}
\maketitle

\section{Introduction}
We consider the modified fractional Korteweg-de Vries (modified fKdV) equation 
\begin{equation}\label{eq:main}
u_t\pm u^2u_x-|D|^\alpha\partial_xu=0,
\end{equation}
where $D=-\mathrm{i}\partial_x$ and hence \(|D|^\alpha\) has Fourier multiplier \(|\xi|^\alpha\). We will distinguish the "weakly dispersive" case $-1\leq \alpha<0$ and the "strongly dispersive case $0<\alpha<2$. (We exclude the case $\alpha=2$ that corresponds to the well-known modified KdV equation.) Note that $\alpha=-1$ corresponds to the modified Burgers-Hilbert equation and $\alpha=1$ to the modified Benjamin-Ono equation.

The $+$ sign will be referred to as the focusing case and the $-$ sign as the defocusing case (actually this distinction is irrelevant when $\alpha <0$).  It is well-known in the weakly dispersive case that 
the symbol \(|D|^\alpha\partial_x\) has the expression 
\begin{equation*}
\begin{aligned}
|D|^\alpha\partial_xf(x)
=c_\alpha\int_\R\frac{\mathrm{sgn}(y)}{|y|^{2+\alpha}}[f(x)-f(x-y)]\, d y,
\end{aligned}
\end{equation*}	
where \(c_\alpha\) is a positive constant only depending on \(\alpha\), which will be regarded as \(1\) for simplicity.

The motivation of the present paper is, as in previous works 
concerning perturbations of the Burgers equation  \cite{EW1, EW2, KLPS, KS, 
LPS2, MPV, SW1} to study the influence of a relatively weak 
dispersive perturbation on the dynamics of  a quasilinear hyperbolic 
equation. Do the hyperbolic properties (appearance of shocks, global entropy
weak  solutions) persist or on the contrary do dispersive effects dominate? Of course both those properties might persist in the same equation and conversely depending on the size of the initial data.

The situation is quite different for $-1<\alpha<0$ and $0<\alpha\leq 1$ as can be seen by looking at the phase velocity $c(\xi)=|\xi|^\alpha$ which is unbounded for low frequencies and goes to zero for large frequencies in the former case. One thus expects that the nonlinear hyperbolic aspects dominate (for large solutions) in the case $\alpha<0$. 

On the other hand dispersive effects appear for small solutions and actually in the case $-1<\alpha<0$ the global existence and modified scattering for \eqref{eq:main} with small initial data was studied  in \cite{SW1} leading to the following result in the focusing case: \footnote{A similar Proposition holds in the defocusing case with a slightly different formulation.}

\begin{proposition}[\cite{SW1}]\label{th:previous} Let \(\alpha\in(-1,0)\).  Define the
	  profile 
	\[f(t,x)=e^{-t|D|^{\alpha} \partial_x}u(t,x),\]
and	the \(Z\)-norm 
	\[\|g\|_Z=\|(1+|\xi|)^{10}\widehat{g}(\xi)\|_{L^\infty_\xi}.\]	 
	Assume that \(N_0=100,\ p_0\in (0,\frac{1}{1000}]\cap (0,-\frac{\alpha}{100}] \) are fixed, and \(\phi\in H^{N_0}(\mathbb{R})\) satisfies
	\begin{align*}
	\|\phi\|_{H^{N_0}}+\|\phi\|_{H^{1,1}}+\|\phi\|_Z=\varepsilon_0\leq \bar{\varepsilon},
	\end{align*}
	for some constant \(\bar{\varepsilon}\) sufficiently small (depending only on \(\alpha\) and \(p_0\)). Then the Cauchy problem of the equation \eqref{eq:main} with the initial data \(u(0,x)=\phi(x)\) admits a unique
	global solution \(u\in C(\mathbb{R}: H^{N_0}(\mathbb{R}))\) satisfying the following uniform bounds for \(t\geq 1\)
	\begin{align*}
	t^{-p_0}\|u\|_{H^{N_0}}+t^{-p_0}\|f\|_{H^{1,1}}+\|f\|_Z\lesssim \varepsilon_0.
	\end{align*}
	Moreover,  there exists \(w_\infty\in L^\infty(\mathbb{R})\) such that  for \(t\geq 1\)
	\begin{align*}
	t^{p_0}\left\|\exp\left(\frac{3\mathrm{i}\xi|\xi|^{1-\alpha}}{\alpha(\alpha+1)}\int_1^t|\widehat{f}(s,\xi)|^2\,\frac{d s}{s}\right)(1+|\xi|)^{10}\widehat{f}(\xi)-w_\infty(\xi)\right\|_{L^\infty_\xi}\lesssim \varepsilon_0.
	\end{align*}
\end{proposition}

One aim of the present paper is to prove that    the solutions of the equation \eqref{eq:main} can form shocks for large initial data. 
We recall that the fKdV equation 
\begin{equation}\label{eq:fKdV}
u_t+uu_x-|D|^\alpha u_x=0
\end{equation}
can form shocks for large solutions in the range \(\alpha\in(-1,-1/3)\) \cite{HT, Hur, SW2}. \footnote{It is very likely that this result holds in the case $-1/3\leq\alpha<0$ but 
this is still unproven.} We 
will show a similar shock formation result holds true for the equation \eqref{eq:main}, but for the whole range  \(\alpha\in(-1,0)\) \footnote{For \(\alpha=-1\), the equation \eqref{eq:main} can be regarded as a modified Burgers-Hilbert equation which is not dispersive although one can extend the shock formation result (Theorem \ref{th:main}) to this case (Theorem \ref{th:mBH}).  The Burgers-Hilbert equation was introduced in \cite{BH} as a model for waves with constant nonzero linearized frequency.} .

On the other hand, when $\alpha>0$ the dispersive effects play a more important role and although finite time blow-up is expected, it should not be shock formation (see for the quadratic case  the numerics in \cite{KS} and the results on the local Cauchy problem in \cite{LPS2, MPV, Men} where the dispersive properties are used to enlarge the space of resolution). 

We go back to the cubic case and first  comment  on the case $\alpha>0.$ In addition to the conservation of the $L^2$ norm (mass), one has the (Hamiltonian) formally conserved quantity
$$H_\alpha(u)=\frac{1}{2}\int_\R\left(|D^{\alpha/2} u|^2\mp \frac{1}{6}u^4\right)\, dx,$$
which by the Sobolev embedding $H^{1/4}(\R)\subset L^4(\R)$ implies that $\alpha=1/2$ is the energy critical exponent. On the other hand,  \eqref{eq:main} is invariant under the scaling transformation $u_\lambda(x)=\lambda^{\alpha/2} u(\lambda x, \lambda^{\alpha +1}t)$ which implies that $\alpha=1$ is the $L^2$ critical exponent.  

In the defocusing  case and when $\alpha>0$ one has  a formal conservation of the energy space  $H^{\alpha/2}(\R)$.

Thus one expects in the focusing case global well-posedness in the energy space when $\alpha>1$ and finite time blow-up when $0<\alpha\leq 1.$ Note that the case $\alpha=1$ corresponds to the modified Benjamin-Ono equation and the finite time blow-up in the focusing case has been proved by Martel and Pilod\cite{MP}. The blow-up should not be a shock (the sup-norm of the solution and of the derivative should blow up at the same time) but its structure should be different in the energy super critical case $0<\alpha<1/2$ and in the $L^2-$ critical case $1/2\leq \alpha <1$. Again we refer to \cite {KS} for numerical simulations in the quadratic case. 

Concerning the local Cauchy problem in "large" Sobolev spaces (that is larger than the "hyperbolic space" $H^{3/2 +}(\R)$) we are not aware of results similar to those  in \cite{LPS2, MPV} corresponding to the quadratic space except when $\alpha =1$ (the modified Benjamin-Ono equation considered in \cite{KK, KT}). In particular it is proven in \cite{KT} that the Cauchy problem for the focusing and defocusing modified Benjamin-Ono equation is locally well-posed in $H^s(\R), s\geq 1/2$ and thus globally well-posed in the same range in the defocusing case.
(The local well-posedness in $H^1(\R)$ was proven in \cite{KK}).

In the defocusing case, while one expects global well-posedness (and scattering) in the energy subcritical case $\alpha\geq 1/2$, things are unclear in the energy supercritical case and one aim of this paper is to present relevant conjectures.

This is in contrast  to the case $\alpha<0$ where the hyperbolic effects dominate for large solutions and the distinction between focusing and defocusing becomes irrelevant.

The paper is organized as follows: The first  Section is devoted to 
the statement and the proof of the main result Theorem \ref{th:main} in the case 
$-1<\alpha<0$, that is the possibility of shocks. In the next two Sections we extend the shock formation result to the modified Burgers-Hilbert and Whitham equations, and the fractional Korteweg-de Vries equation  with some generalized nonlinearity. Then we focus on the case $\alpha >0$, 
and consider successively the solitary wave solutions and the Cauchy problem in the focusing and defocusing case. Most issues will lead to conjectures illustrated by numerical simulations.
 
 We conclude the paper by some remarks on the "BBM" version of the modified fKdV equation when $\alpha>0.$
  
\section{The case $-1<\alpha<0.$}
\subsection{Main result}
 We say that the solution of \eqref{eq:main}  exhibits shock formation
 if there exists some \(T>0\) such that
 \begin{equation*}
 \begin{aligned}
 |u(x,t)|<\infty,\quad x\in\R,\ t\in[0,T),
 \end{aligned}
 \end{equation*}	
 but	
 \begin{equation*}
 \begin{aligned}
 \sup_{x\in\R}|\partial_xu(x,t)|\longrightarrow +\infty,\quad \text{as}\ t\longrightarrow T-.
 \end{aligned}
 \end{equation*}	
 
 The main result of this Section is stated as follows in the focusing case \footnote{A similar shock formation result holds with slight modifications in the defocusing case.}: 
 
 \begin{theorem}\label{th:main} \textup{(Rough version)} Let  \(\alpha\in(-1,0)\). 
 	There exists a wide class of functions \(\phi\in H^3(\R)\) with appropriate large positive amplitude \(\phi\) and negative slope \(\inf_{x\in\R}\phi^\prime(x)\) such that the Cauchy problem for the equation \eqref{eq:main} with initial data \(u(0,x)=\phi(x)\) exhibits shock formation. 
 	
 	\textup{(Precise version)} Let \(\alpha\in(-1,0)\) and \(\delta\) be a sufficiently small positive number. Assume \(\bar{x}_1\) and \(\bar{x}_2\) are the largest and smallest numbers such that
 	\(\overline{\{x: \phi^\prime(x)<0\}}\subset [\bar{x}_1,\bar{x}_2] \).
 	If \(\phi\in H^3(\R)\) satisfies the \textbf{slope condition}
 	\begin{align}
 	&-\inf_{x\in\R}\phi^\prime(x)>\delta^{-1}f_1,\label{c1}\\
 	&-\inf_{x\in\R}\phi^\prime(x)>(1-\delta)^{-2}f_2,\label{c2}\\
 	&-\inf_{x\in\R}\phi^\prime(x)>(1-\delta)^{-3}f_3,\label{c3}
 	\end{align}
 	and  the \textbf{local amplitude condition}
 	\begin{align}
 	&\phi(x)< B-(1-\delta)^{-2}f_4,\label{c4}\\
 	&\phi(x)> A+(1-\delta)^{-2}f_4, \label{c5}
 	\end{align}
 	for all \(x\in [\bar{x}_1,\bar{x}_2]\).
 	Here the functions \((f_1,f_2,f_3,f_4)\) are homogeneous in each argument of order \((1/2,0,0,0)\), and have the following explicit formulae	
 	\begin{equation*}
 	\begin{aligned}
 	&\quad f_1=:f_1(\|\phi\|_{H^2},C_1,\|\phi^{\prime\prime\prime}\|_{L^2})\\
 	&=\big[C_s\|\phi\|_{H^2}+4C_1(1+\alpha)^{-1}
 	+2C_s(-\alpha)^{-1}(2^{-1}AB^{-1})^{-7BA^{-2}/4}
 	\|\phi^{\prime\prime\prime}\|_{L^2}\big]^{1/2},
 	\end{aligned}
 	\end{equation*}
 	\begin{equation*}
 	\begin{aligned}
 	f_2=:f_2(C_0^{-1}C_1)=8(-\alpha(1+\alpha))^{-1}+4\alpha^{-2}C_0^{-1}C_1,
 	\end{aligned}
 	\end{equation*}
 	\begin{equation*}
 	\begin{aligned}
 	&\quad f_3=:f_3(C_1^{-1}\|\phi^{\prime\prime\prime}\|_{L^2})\\
 	&=8(1+\alpha)^{-1}+4C_s(-\alpha)^{-1}(2^{-1}AB^{-1})^{-7BA^{-2}/4}
 	C_1^{-1}\|\phi^{\prime\prime\prime}\|_{L^2},
 	\end{aligned}
 	\end{equation*}
 	\begin{equation*}
 	\begin{aligned}
 	&\quad f_4=:f_4\big(C_0\big(-\inf_{x\in\R}\phi^\prime(x)\big)^{-1},C_1\big(-\inf_{x\in\R}\phi^\prime(x)\big)^{-1}\big)\\
 	&=\big(-\inf_{x\in\R}\phi^\prime(x)\big)^{-1}\big[4C_0\big(-\alpha(1+\alpha)\big)^{-1}+2C_1\alpha^{-2}\big],
 	\end{aligned}
 	\end{equation*}
 	where \(C_0>0\) and \(C_1>0\) satisfying	
 	\begin{equation}\label{c6}
 	\begin{aligned}
 	\|\phi\|_{L^\infty}\leq \frac{C_0}{2},\quad 
 	\|\phi^\prime\|_{L^\infty}\leq \frac{C_1}{2},
 	\end{aligned}
 	\end{equation}	
 	and \(A>0\) and \(B>0\) satisfying
 	\begin{equation}\label{c7}
 	\begin{aligned}
 	A^2>8\delta B^2,\quad 8A^{2}>7B,\quad B> A+2(1-\delta)^{-2}f_4,
 	\end{aligned}
 	\end{equation}
 	and \(C_s\) is the best embedding constant of Sobolev inequality
 	\begin{equation*}
 	\begin{aligned}
 	\|f\|_{L^\infty(\R)}\leq C_s\|f\|_{H^1(\R)}.
 	\end{aligned}
 	\end{equation*}   
 	Then the solution for the equation \eqref{eq:main} with initial data \(u(0,x)=\phi(x)\) exhibits shock formation at some time \(T>0\) satisfying
 	\begin{equation}\label{shock time}
 	\begin{aligned}
 	&(2B+\delta)^{-1}\big(-\inf_{x\in\R}\phi^\prime(x)\big)^{-1}
 	<T\\
 	&<(AB^{-1}-\delta)^{-1}(2A-\delta)^{-1}\big(-\inf_{x\in\R}\phi^\prime(x)\big)^{-1},
 	\end{aligned}
 	\end{equation}
 	and at some location \(x_*\) satisfying
 	\begin{align}\label{shock location}
 	\bar{x}_1-C_0^2T
 	\leq x_*\leq \bar{x}_2+C_0^2T.
 	\end{align}
 	Moreover, we have the blow-up rate estimate
 	\begin{align}\label{shock blow-up rate}
 	(2B+\delta)^{-1}(T-t)^{-1}\leq\|\partial_xu(\cdot,t)\|_{L^\infty}\leq (AB^{-1}-\delta)^{-1}(2A-\delta)^{-1}(T-t)^{-1}, 
 	\end{align}
 	as \(t\rightarrow T^{-}\).	
 \end{theorem}
 
 Compared to the fKdV equation, the additional price for 
 \eqref{eq:main} that one shall pay is not only to assume the negative 
 slope \(\phi^\prime\) is appropriately large but also \(\phi\) itself. The reason lies in that the nonlinear term \(2v_0v_1^2\) (which is expected to be the leading term) of the equation \eqref{eq:main} in particle path form (see \eqref{3.2}) has no fixed sign generally. To make sure that \(v_0(t,x)=u(t,X(t,x))\) (see \eqref{0}) has a positive lower bound, we need to impose some kind of conditions such as  \eqref{c4}-\eqref{c5} on \(\phi\), and then verify that \(v_0(t,x)\) is bounded below by a positive constant before the shock comes.

 \begin{remark}
 	There exists a wide class of functions \(\phi\in H^3(\R)\) 
	satisfying  \eqref{c1}-\eqref{c5} in Theorem \ref{th:main}. This 
	can be seen easily from the orders of \(\phi\) in  both sides of 
	each inequality by homogeneities of the functions 
	\((f_1,f_2,f_3,f_4)\) on each of its arguments. Indeed, the left hand side of \eqref{c1}-\eqref{c5} has  one more order on \(\phi\) than that of the right hand side correspondingly, if the original \(\phi\) does not work, then one can replace it by \(\lambda \phi\) with a sufficiently large \(\lambda>0\). 
 \end{remark}
 
 \begin{remark}
 	There exists an open neighborhood in the \(H^3\)-topology of the set of initial data satisfying the hypotheses in Theorem \ref{th:main} such that the conclusions of Theorem \ref{th:main} hold. This is an obvious fact since the inequalities \eqref{c1}-\eqref{c5} are stable for small perturbations. So Theorem \ref{th:main} contains the following precise information on the shock:\\
 	(a1) shock time; \\
 	(a2) shock location;\\
 	(a3) shock blow-up rate;\\
 	(a4) openness for initial data of producing shock.
 \end{remark}

 \begin{remark}
 	We emphasize that in the local amplitude condition 
	\eqref{c4}-\eqref{c5}, the interval \( [\bar{x}_1,\bar{x}_2]\) 
	can be replaced by a larger but finite interval. Otherwise, the local amplitude condition will become a global one, which means \(\phi\notin L^2(\R)\) and thus contradicts \(\phi\in H^3(\R)\). More importantly, if the latter case happened, then \(\phi\) is positive on the entire line which is physically irrelevant since \(\phi\) stands for the initial elevation. 
 \end{remark}

 \begin{remark}
 	We mention here previous works on the finite time blow-up 
	phenomena for related weakly dispersive nonlinear equations. 
	Naumkin and Shishmarev \cite{NS}, and Constantin and Escher 
	\cite{CE} have proven shock formation for a Whitham type equation  which, however, does not include the 
 	Whitham equation considered in this paper. 
 	We now consider the fKdV equation \eqref{eq:fKdV} for \(-1<\alpha<0\).
 	The finite time blow-up of a $C^{1+\delta}(\R)$ norm was established 
 	in \cite{CCG} but the shock formation was not proven there. The 
 	possibility of appearance of shocks was proven in \cite{HT, Hur} when $-1<\alpha<-1/3$ and for the Whitham equation. A simpler proof, applying to the case $-1<\alpha <-2/5$ and the Whitham equation was given in \cite{SW2}.  
 \end{remark}

\subsection{Proof of Theorem \ref{th:main}}

	It is standard (see e.g., \cite{JCS}) to show that the Cauchy problem for the equation \eqref{eq:main} with initial data \(u(0,x)=\phi(x)\) is well-posed in the class \(C([0,T): H^3(\R))\) for some \(T>0\) which in what follows will denote the maximal time of existence.  
	
	We define the particle path
	\begin{equation}\label{eq:particle}
	\begin{aligned}
	&\frac{d }{d t}X(t,x)=u^2(t,X(t,x)),\\ 
	&X(0,x)=x.
	\end{aligned}
	\end{equation}
	Since \(u(x,t)\in C([0,T):H^3(\R))\), the ODE theory shows that \(X(\cdot;x)\) exists throughout the interval \(t\in[0,T)\) for all \(x\in\R\).
	We define
	\begin{equation}\label{0}
	\begin{aligned}
	&v_0(t,x)=u(t,X(t,x)), \\
	&v_1(t,x)=\partial_xu(t,X(t,x)),
	\end{aligned}
	\end{equation}
	and 
	\begin{equation}\label{1}
	\begin{aligned}
	m(t)=\inf_{x\in\R}v_1(t;x)=\inf_{x\in\R}\partial_xu(x,t)=:m(0)q^{-1}(t).
	\end{aligned}
	\end{equation}
	It is easy to see that
	\begin{align}
	&m(t)<0,\quad t\in [0,T),\label{2}\\
	&q(0)=1,\quad 
	q(t)>0,\quad t\in [0,T).\label{3}
	\end{align}
	It then follows from \eqref{eq:main} that 
	\begin{align}
	&\frac{d v_0}{d t}+ K_0(t,x)=0,\label{3.1}\\
	&\frac{d v_1}{d t}+2v_0v_1^2+K_1(t,x)=0,\label{3.2}
	\end{align}
	where
	\begin{align}
	&K_0(t,x)
	=\int_\R\frac{\mathrm{sgn}(y)}{|y|^{2+\alpha}}[u(t,X(t,x))-u(t,X(t,x)-y)]\, d y,\label{3.3}\\
	&K_1(t,x)
	=\int_\R \frac{\mathrm{sgn}(y)}{|y|^{2+\alpha}}[\partial_xu(t,X(t,x))-\partial_xu(t,X(t,x)-y)]\, d y \label{3.4}.
	\end{align}

	The main ingredient is to show
	\begin{equation}\label{4}
	\begin{aligned}
	|K_1(t,x)|<\delta^2 m^2(t),\quad \mathrm{for\ all}\ t\in[0,T)\ \mathrm{and} \ x\in\R.
	\end{aligned}
	\end{equation}
	In view of \eqref{c1}, one easily checks that \eqref{4} holds at \(t=0\). 
	We will prove \eqref{4} by contradiction. Suppose that \(|K_1(T_1,x_0)|=\delta^2 m^2(T_1)\) for some \(T_1\in(0,T)\) and some \(x_0\in\R\). 
	By continuity, without loss of generality, we may assume that
	\begin{equation}\label{5}
	\begin{aligned}
	|K_1(t,x)|\leq \delta^2 m^2(t),\quad \text{for\ all}\ t\in[0,T_1]\ \mathrm{and} \ x\in\R.
	\end{aligned}
	\end{equation}

	The following technical lemmas are a variant of those in \cite{NS,HT,SW2} which dealt with the quadratic nonlinearity. Differently, here we need to work near the shock to deal with the cubic nonlinearity.

	\subsubsection{Bounds on \(q(t)\)}
	
	Let \(t\in[0,T_1]\) and \(x\in [\bar{x}_1,\bar{x}_2]\).  We a priori assume 
	\begin{equation}\label{bound on v_0}
	\begin{aligned}
	A\leq v_0(t,x)\leq B,  
	\end{aligned}
	\end{equation}
	in which \(A\) and \(B\) are given in Theorem \ref{th:main} satisfying \eqref{c7}.

	We define 
	\begin{equation*}
	\begin{aligned}
	\Sigma_{\delta}(t)=\{x\in[\bar{x}_1,\bar{x}_2]:\ v_1(t,x)\leq (AB^{-1}-\delta)m(t)\},
	\end{aligned}
	\end{equation*}
	and 
	\begin{equation*}
	\begin{aligned}
	v_1(t,x)=:m(0)r^{-1}(t,x).
	\end{aligned}
	\end{equation*}

	\begin{lemma}\label{le:a1} We have \(\Sigma_{\delta}(t_2)\subset \Sigma_{\delta}(t_1) \) whenever \(0\leq t_1\leq t_2\leq T_1\).
		
	\end{lemma}
	
	\begin{proof} 
		Suppose that there exists some \(x_1\in[\bar{x}_1,\bar{x}_2]\) such that
		\(x_1\notin\Sigma_{\delta}(t_1)\) but \(x_1\in\Sigma_{\delta}(t_2)\) for some \(0\leq t_1\leq t_2\leq T_1\), that is
		\begin{equation}\label{a1}
		\begin{aligned}
		&v_1(t_1,x_1)> (AB^{-1}-\delta)m(t_1),\\ 
		&v_1(t_2,x_1)\leq (AB^{-1}-\delta)m(t_2)<\frac{1}{2}AB^{-1}m(t_2).
		\end{aligned}
		\end{equation}
		Since \(v_1(\cdot,x_1)\) and \(m\) are uniformly continuous on \([0,T_1]\), due to the second inequality of \eqref{a1},	one can choose \(t_1\) and \(t_2\) sufficiently  close so that
		\begin{equation}\label{ad1}
		\begin{aligned}
		v_1(t,x_1)\leq \frac{1}{2}AB^{-1}m(t),\quad \text{for\ all}\ t\in [t_1,t_2].
		\end{aligned}
		\end{equation}
		Let 
		\begin{equation}\label{a2}
		\begin{aligned}
		v_1(t_1,x_2)=m(t_1)\big(\leq \frac{1}{2}AB^{-1}m(t_1)\big),
		\end{aligned}
		\end{equation}
		again one may choose \(t_2\) close to \(t_1\) so that
		\begin{equation}\label{ad2}
		\begin{aligned}
		v_1(t,x_2)\leq \frac{1}{2}AB^{-1}m(t),\quad \text{for\ all}\ t\in [t_1,t_2].
		\end{aligned}
		\end{equation}
		In the following we fix \(t_1\) and \(t_2\) such that all the inequalities \eqref{a1}, \eqref{ad1} and \eqref{ad2} hold true.

		According to \eqref{5} and \(A^2>8\delta B^2\) in \eqref{c7}, one has
		\begin{equation*}
		\begin{aligned}
		|K_1(t,x_j)|\leq \delta^2 m^2(t)
		\leq 4A^{-2}B^{2}\delta^2v_1^2(t,x_j)<\frac{\delta}{2} v_1^2(t,x_j),
		\end{aligned}
		\end{equation*}
		for all \(t\in[t_1,t_2],\ j=1,2\).
		This together with \eqref{3.2} yields
		\begin{equation*}
		\begin{aligned}
		\frac{d v_1}{d t}(t,x_1)&=-2v_0v_1^2(t,x_1)- K_1(t,x_1)\\
		&\geq -\big(2v_0(t,x_1)+\frac{\delta}{2}\big)v_1^2(t,x_1),\\
		&\geq -\big(2B+\frac{\delta}{2}\big)v_1^2(t,x_1),
		\end{aligned}
		\end{equation*}
		and
		\begin{equation*}
		\begin{aligned}
		\frac{d v_1}{d t}(t,x_2)&=-2v_0v_1^2(t,x_2)- K_1(t,x_2)\\
		&\leq -\big(2v_0(t,x_2)-\frac{\delta}{2}\big)v_1^2(t,x_2),\\
		&\leq -\big(2A-\frac{\delta}{2}\big)v_1^2(t,x_2),
		\end{aligned}
		\end{equation*}
		for \(t\in [t_1,t_2]\). Solving the resulting two inequalities above gives
		\begin{equation}\label{a3}
		\begin{aligned}
		v_1(t_2,x_1)\geq\frac{v_1(t_1,x_1)}{1+\big(2B+\frac{\delta}{2}\big)v_1(t_1,x_1)(t_2-t_1)},
		\end{aligned}
		\end{equation}
		and
		\begin{equation}\label{a4}
		\begin{aligned}
		v_1(t_2,x_2)\leq\frac{v_1(t_1,x_2)}{1+\big(2A-\frac{\delta}{2}\big)v_1(t_1,x_2)(t_2-t_1)}.
		\end{aligned}
		\end{equation}
		
		Applying \eqref{a2} to \eqref{a4}, one obtains
		\begin{equation}\label{a5}
		\begin{aligned}
		m(t_2)\leq\frac{m(t_1)}{1+\big(2A-\frac{\delta}{2}\big)m(t_1)(t_2-t_1)}.
		\end{aligned}
		\end{equation}
		In view of \eqref{a1} and \eqref{a3}, one estimates
		\begin{equation*}
		\begin{aligned}
		v_1(t_2,x_1)&>\frac{(AB^{-1}-\delta)m(t_1)}{1+\big(2B+\frac{\delta}{2}\big)(AB^{-1}-\delta)m(t_1)(t_2-t_1)}\\
		&>\frac{(AB^{-1}-\delta)m(t_1)}{1+\big(2A-\frac{\delta}{2}\big)m(t_1)(t_2-t_1)}\\
		&>(AB^{-1}-\delta)m(t_2),
		\end{aligned}
		\end{equation*}
		where we have used \eqref{a5} in the last inequality. We get a contradiction!	
		
	\end{proof}

	\begin{lemma}\label{le:a2} \(q(t)\) is decreasing and satisfies
		\begin{equation}\label{a6}
		\begin{aligned}
		0<q(t)\leq 1,\quad \text{for}\ t\in [0,T_1].
		\end{aligned}
		\end{equation}
		We also have the integral estimates 
		\begin{equation}\label{a7}
		\begin{aligned}
		\int_0^tq^{-s}(\tau)\, d \tau
		&\leq (1-s)^{-1}m^{-1}(0)(2A-\delta)^{-1}
		(AB^{-1}-\delta)^{-s}\\
		&\quad\times\big[q^{1-s}(t)-(AB^{-1}-\delta)^{s-1}\big],\quad \text{for}\ t\in [0,T_1].
		\end{aligned}
		\end{equation}	
		where \(s>0, s\neq 1\), and	
		\begin{equation}\label{a8}
		\begin{aligned}
		\int_0^tq^{-1}(\tau)\, d \tau
		&\leq m^{-1}(0)(2A-\delta)^{-1}(AB^{-1}-\delta)^{-1}\\
		&\quad\times\big[\log (AB^{-1}-\delta)+\log q(t)\big],\quad \text{for}\ t\in [0,T_1].
		\end{aligned}
		\end{equation}

	\end{lemma}
	
	\begin{proof} 
		Let \(x\in\Sigma_{\delta}(T_1)\), it first follows from Lemma \ref{le:a1} that
		\begin{equation}\label{a9}
		\begin{aligned}
		m(t)\leq v_1(t,x)\leq (AB^{-1}-\delta)m(t), \quad \text{for}\ t\in [0,T_1].
		\end{aligned}
		\end{equation}	
		The solution of \eqref{3.2} can be expressed 	
		\begin{equation}\label{a10}
		\begin{aligned}
		&m(0)r^{-1}(t,x)=v_1(t,x)\\
		&=\frac{v_1(0,x)}{1+ v_1(0,x)\int_0^t\big[2v_0(\tau,x)+v_1^{-2}K_1(\tau,x)\big]\, d \tau}.
		\end{aligned}
		\end{equation}	
		It follows from \eqref{5} and \eqref{a9} that 
		\begin{equation*}
		\begin{aligned}
		|v_1^{-2}K_1(\tau,x)|\leq (AB^{-1}-\delta)^{-2}\delta^2
		<\delta,\quad \text{for}\ t\in [0,T_1]\ \text{and}\ x\in\Sigma_{\delta}(T_1).
		\end{aligned}
		\end{equation*}	
		This together with \eqref{a10} implies 
		\begin{align}\label{a11}
		(2B+\delta)m(0)\leq \frac{d }{d t}r(t,x)\leq (2A-\delta)m(0),\quad \text{for}\ t\in [0,T_1]\ \text{and}\ x\in\Sigma_{\delta}(T_1).
		\end{align}	
		It is easy to see that \(r(t,x)\) is decreasing for all 
		\(t\in [0,T_1]\) from \eqref{a11}, and hence \(v_1(t,x)\), too. 
		Furthermore, by \eqref{1}, \(q(t)\) is also decreasing for all 
		\(t\in [0,T_1]\), which implies \eqref{a6} by \eqref{3}.
		
		On the other hand, we conclude from \eqref{1}, \eqref{a9} and \eqref{a10} that
		\begin{equation}\label{a12}
		\begin{aligned}
		q(t)\leq r(t,x)\leq (AB^{-1}-\delta)^{-1}q(t),\quad \text{for}\ t\in [0,T_1]\ \text{and}\ x\in\Sigma_{\delta}(T_1).
		\end{aligned}
		\end{equation}
		Let \(s>0, s\neq 1\), we use \eqref{a11} and \eqref{a12} to deduce that
		\begin{equation*}
		\begin{aligned}
		&\int_0^tq^{-s}(\tau)\, d \tau
		\leq (AB^{-1}-\delta)^{-s}\int_0^tr^{-s}(\tau,x)\, d \tau\\
		&\leq m^{-1}(0)(2A-\delta)^{-1}(AB^{-1}-\delta)^{-s}
		\int_0^tr^{-s}(\tau,x)\frac{d }{d \tau}r(\tau,x)\, d \tau\\
		&=(1-s)^{-1}m^{-1}(0)(2A-\delta)^{-1}(AB^{-1}-\delta)^{-s}
		[r^{1-s}(t,x)-r^{1-s}(0,x)],
		\end{aligned}
		\end{equation*}	
		which combines \eqref{a12} again implies \eqref{a7}. One can verify \eqref{a8} similarly.
	\end{proof}

	\bigskip
	To show \eqref{5}, we also use a contradiction argument. 
	We claim that
	\begin{align}\label{6}
	\|v_0(t)\|_{L^\infty}=\|u(t)\|_{L^\infty}<C_0,\quad   \text{for\ all}\  t\in [0,T_1], 
	\end{align}
	and
	\begin{align}\label{7}
	\|v_1(t)\|_{L^\infty}=\|\partial_xu(t)\|_{L^\infty}<C_1q^{-1}(t),\quad   \text{for\ all}\  t\in [0,T_1], 
	\end{align}
	where \(C_0,C_1\) satisfy \eqref{c6}. First observe that 
	\begin{equation*}
	\begin{aligned}
	\|v_0(0)\|_{L^\infty}=\|\phi\|_{L^\infty}<C_0,
	\end{aligned}
	\end{equation*}
	and
	\begin{equation*}
	\begin{aligned}
	\|v_1(0)\|_{L^\infty}=\|\phi^\prime\|_{L^\infty}<C_1q^{-1}(0).
	\end{aligned}
	\end{equation*}
	We then proceed by  contradiction in order  to show \eqref{6} and \eqref{7}. 
	Suppose that \eqref{6} and \eqref{7} hold for all  \(t\in [0, T_2)\), but fails for either \eqref{6} or \eqref{7} at \(t = T_2\) for some \(T_2\in (0, T_1]\). Hence, by continuity, it holds
	\begin{align}\label{8}
	\|v_0(t)\|_{L^\infty}=\|u(t)\|_{L^\infty}\leq C_0,\quad   \text{for\ all}\  t\in [0,T_2], 
	\end{align}
	and
	\begin{align}\label{9}
	\|v_1(t)\|_{L^\infty}=\|\partial_xu(t)\|_{L^\infty}\leq C_1q^{-1}(t),\quad   \text{for\ all}\  t\in [0,T_2].
	\end{align}
	
	\bigskip
	
	\subsubsection{Estimates on Nonlocal Terms}	
	\begin{lemma}\label{le:fKdV-nonlocal} For all \(t\in[0,T_2]\) and \(x\in\R\),  we have 
		\begin{equation}\label{10}
		\begin{aligned}
		|K_0(t,x)|\leq \big[4C_0(1+\alpha)^{-1}+2C_1(-\alpha)^{-1}\big]q(t)^{-(1+\alpha)},
		\end{aligned}
		\end{equation}
		and
		\begin{equation}\label{12}
		\begin{aligned}
		|K_1(t,x)|
		\leq \big[4C_1(1+\alpha)^{-1}+2C_s(-\alpha)^{-1}(2^{-1}AB^{-1})^{-7BA^{-2}/4}\|\phi^{\prime\prime\prime}\|_{L^2}\big]q^{-2}(t),
		\end{aligned}
		\end{equation}

	\end{lemma}
	
	\begin{proof} We follow the idea of \cite{SW2}, however the proof here is simpler when  dealing with \(\|\partial_x^2u\|_{L^\infty}\) in order to estimate \(K_1(t,x)\) than that of the fKdV equation  due to the stronger nonlinear effect present in  the equation \eqref{eq:main}. The proof on \(K_0(t,x)\) is the same as that of \cite{SW2}, we include it here for sake of completeness. 
		To estimate the terms \(K_0(t,x)\) and \(K_1(t,x)\), we perform  the following decompositions  
		\begin{equation*}
		\begin{aligned}
		|K_0(t,x)|\leq A_1+A_2,\\
		|K_1(t,x)|\leq B_1+B_2,
		\end{aligned}
		\end{equation*}
		where
		\begin{equation*}
		\begin{aligned}
		&A_1=\int_{|y|<\eta_1}\frac{1}{|y|^{2+\alpha}}\big|u(t,X(t,x))-u(t,X(t,x)-y)\big|\, d y,\\
		&A_2=\int_{|y|\geq\eta_1}\frac{1}{|y|^{2+\alpha}}\big|u(t,X(t,x))-u(t,X(t,x)-y)\big|\, d y,\\
		&B_1=\int_{|y|<\eta_2}\frac{1}{|y|^{2+\alpha}}\big|\partial_xu(t,X(t,x))-\partial_xu(t,X(t,x)-y)\big|\, d y,\\
		&B_2=\int_{|y|\geq\eta_2}\frac{1}{|y|^{2+\alpha}}\big|\partial_xu(t,X(t,x))-\partial_xu(t,X(t,x)-y)\big|\, d y,
		\end{aligned}
		\end{equation*}
		for some \(\eta_1>0\) and \(\eta_2>0\) to be specified later. 
		In view of \eqref{8}-\eqref{9},	one may estimate
		\begin{equation}\label{14}
		\begin{aligned}
		|A_1|
		&\leq |u|_{C^{0,1}} \int_{|y|<\eta_1}\frac{1}{|y|^{2+\alpha}}|y|\, d y
		\leq 2(-\alpha)^{-1}\eta_1^{-\alpha}\|\partial_xu\|_{L^\infty}\\
		&\leq 2 C_1(-\alpha)^{-1}\eta_1^{-\alpha}q(t)^{-1},
		\end{aligned}
		\end{equation}
		and   	
		\begin{equation}\label{15}
		\begin{aligned}
		|A_2|&\leq 2\|v_0\|_{L^\infty}\int_{|y|<\eta_1}\frac{1}{|y|^{2+\alpha}}\, d y
		\leq 4(1+\alpha)^{-1}\eta_1^{-(1+\alpha)}\|v_0\|_{L^\infty}\\
		&\leq 4 C_0(1+\alpha)^{-1}\eta_1^{-(1+\alpha)}.
		\end{aligned}
		\end{equation}	
		Choosing \(\eta_1=q(t)\), one may minimize \eqref{14} and \eqref{15} to obtain the estimate \eqref{10}.
		
		Similarly, one has
		\begin{equation}\label{16}
		\begin{aligned}
		|B_1|
		\leq 2(-\alpha)^{-1}\eta_2^{-\alpha}\|\partial_x^2u\|_{L^\infty},
		\end{aligned}
		\end{equation}
		and
		\begin{equation}\label{17}
		\begin{aligned}
		|B_2|
		\leq 4 C_1(1+\alpha)^{-1}\eta_2^{-(1+\alpha)}q^{-1}(t).
		\end{aligned}
		\end{equation}
		To control \(\|\partial_x^2u\|_{L^\infty}\), we shall estimate \(\|\partial_x^3u\|_{L^2}\).
		Using integration by parts, a straightforward calculation gives
		\begin{equation}\label{the third derivative}
		\begin{aligned}
		\frac{1}{2}\frac{d}{d t}\int_\R(\partial_x^3u)^2\, d x
		=-\frac{7}{2}\int_\R\partial_xu(\partial_x^3u)^2\, d x
		\leq -\frac{7}{2} m(0)q^{-1}(t)\|\partial_x^3u\|_{L^2}^2,
		\end{aligned}
		\end{equation}
		which together with \eqref{a8} yields
		\begin{equation}\label{18}
		\begin{aligned}
		\|\partial_x^3u\|_{L^2}
		&\leq \|\phi^{\prime\prime\prime}\|_{L^2}(AB^{-1}-\delta)^{-\frac{7}{2(2A-\delta)(AB^{-1}-\delta)}}
		q(t)^{-\frac{7}{2(2A-\delta)(AB^{-1}-\delta)}}\\
		&\leq (2^{-1}AB^{-1})^{-7BA^{-2}/4}\|\phi^{\prime\prime\prime}\|_{L^2}
		q(t)^{-7BA^{-2}/4},
		\end{aligned}
		\end{equation}
		where we have used the assumption  that \(\delta\) is sufficiently small. Inserting \eqref{18} into \eqref{16} gives
		\begin{equation}\label{19}
		\begin{aligned}
		|B_1|
		\leq 2C_s(-\alpha)^{-1}(2^{-1}AB^{-1})^{-7BA^{-2}/4}\|\phi^{\prime\prime\prime}\|_{L^2}\eta_2^{-\alpha}q(t)^{-7BA^{-2}/4}.
		\end{aligned}
		\end{equation}
		Minimizing 	\eqref{17} and \eqref{19} by taking \(\eta_2=q(t)^{-1+7BA^{-2}/4}\), one obtains
		\begin{equation*}
		\begin{aligned}
		|K_1(t,x)|
		&\leq \big[4C_1(1+\alpha)^{-1}+2C_s(-\alpha)^{-1}(2^{-1}AB^{-1})^{-7BA^{-2}/4}\|\phi^{\prime\prime\prime}\|_{L^2}\big]\\
		&\quad\times q(t)^{(1-7BA^{-2}/4)(1+\alpha)-1}\\
		&\leq \big[4C_1(1+\alpha)^{-1}+2C_s(-\alpha)^{-1}(2^{-1}AB^{-1})^{-7BA^{-2}/4}\|\phi^{\prime\prime\prime}\|_{L^2}\big]
		q(t)^{-2},
		\end{aligned}
		\end{equation*}
		in which we have used the fact
		\begin{equation*}
		\begin{aligned}
		(1-7BA^{-2}/4)(1+\alpha)\geq -1,
		\end{aligned}
		\end{equation*}
		due to \(8A^{2}>7B\) in \eqref{c7} and \(\alpha\in(-1,0)\).
		This finishes the proof of \eqref{12}.

	\end{proof}

	\subsubsection{Proof of \eqref{4}} Let \(t\in[0,T_2]\) and  \(x\in\R\).		 
	In light of \eqref{10} and applying \eqref{a7}, it follows from \eqref{3.1} that
	\begin{equation}\label{21}
	\begin{aligned}
	|v_0(t,x)|&\leq \|\phi\|_{L^\infty}+\int_0^{t}|K_0(\tau,x)|\, d \tau\\
	&\leq \frac{C_0}{2}+\big[4C_0(1+\alpha)^{-1}+2C_1(-\alpha)^{-1}\big]\int_0^tq(\tau)^{-(1+\alpha)}\, d \tau\\
	&\leq \frac{C_0}{2}-m^{-1}(0)(1-\delta)^{-2}[4C_0(-\alpha(1+\alpha))^{-1}+2C_1\alpha^{-2}]\\
	&< C_0,
	\end{aligned}
	\end{equation}
	where we have used \eqref{c2} in the last inequality.
	By \eqref{12} and applying \eqref{a7}, one uses \eqref{3.2} to deduce that
	\begin{equation}\label{22}
	\begin{aligned}
	v_1(t,x)&\leq \|\phi^\prime\|_{L^\infty}+\int_0^{t}|K_1(\tau,x)|\, d \tau\\
	&\leq\|\phi^\prime\|_{L^\infty}+\big[4C_1(1+\alpha)^{-1}\\
	&\quad+2C_s(-\alpha)^{-1}(2^{-1}AB^{-1})^{-7BA^{-2}/4}\|\phi^{\prime\prime\prime}\|_{L^2}\big]\int_0^tq^{-2}(\tau)\, d \tau\\
	&\leq \frac{C_1}{2}q^{-1}(t)- m^{-1}(0)(1-\delta)^{-3}q^{-1}(t)\\
	&\quad\times\big[4C_1(1+\alpha)^{-1}+2C_s(-\alpha)^{-1}(2^{-1}AB^{-1})^{-7BA^{-2}/4}\|\phi^{\prime\prime\prime}\|_{L^2}\big]\\
	&<C_1q^{-1}(t),
	\end{aligned}
	\end{equation}
	in which we have used \eqref{c3} in the last inequality.
	On the other hand, one has
	\begin{equation}\label{23}
	\begin{aligned}
	v_1(t,x)\geq m(t)=m(0)q^{-1}(t)
	\geq-\frac{C_1}{2}q^{-1}(t).
	\end{aligned}
	\end{equation}

	A contradiction with \eqref{8}-\eqref{9} occurs following from \eqref{21}-\eqref{23}. 
	Now we go back to \eqref{12} and use \eqref{c1} to find that
	\begin{equation*}
	\begin{aligned}
	|K_1(t,x)|
	&\leq m^{-2}(0)m^2(t)\big[4C_1(1+\alpha)^{-1}\\
	&\quad+2C_s(-\alpha)^{-1}(2^{-1}AB^{-1})^{-7BA^{-2}/4}\|\phi^{\prime\prime\prime}\|_{L^2}\big]\\
	&<\delta^2m^2(t),
	\end{aligned}
	\end{equation*}
	for all \(t\in [0, T_1 ]\) and all \(x\in\R\). We get a contradiction to \eqref{5}! This means we have shown \eqref{4} for all \(t\in [0, T)\) and all \(x\in\R\). 	
	\qed

	\subsubsection{Proof of \eqref{bound on v_0}} Let \(t\in[0,T_1]\) and  \(x\in[\bar{x}_1,\bar{x}_2]\).
	In a similar manner to \eqref{21},  one may estimate 
	\begin{equation}\label{24}
	\begin{aligned}
	v_0(t,x)
	\leq \phi(x)-m^{-1}(0)(1-\delta)^{-2}\big[4C_0\big(-\alpha(1+\alpha)\big)^{-1}+2C_1\alpha^{-2}\big]
	< B,
	\end{aligned}
	\end{equation}
	and
	\begin{equation}\label{25}
	\begin{aligned}
	v_0(t,x)
	\geq \phi(x)+m^{-1}(0)(1-\delta)^{-2}\big[4C_0\big(-\alpha(1+\alpha)\big)^{-1}+2C_1\alpha^{-2}\big]
	> A,
	\end{aligned}
	\end{equation}
	where we have used \eqref{c4} in \eqref{24} and \eqref{c5} in \eqref{25} respectively. The estimates \eqref{24}-\eqref{25} entail \eqref{bound on v_0}.
	\qed

	\bigskip
	We are now in a position to finish the Proof of Theorem \ref{th:main}.
	\begin{proof}[Proof of Theorem \ref{th:main}]
		We first mention that since we have already shown \eqref{4}, 
		the results in Lemma \ref{le:a1} and Lemma \ref{le:a2} hold true for all closed intervals \(t \in [0, T^\prime]\) with any \(T^\prime<T\) of \([0, T)\).
		Let \(t \in [0, T)\), for any  
		\begin{equation*}
		\begin{aligned}
		x\in \Sigma_{\delta}(t)=\{x\in[\bar{x}_1,\bar{x}_2]:v_1(t,x)\leq (AB^{-1}-\delta)m(t)\},
		\end{aligned}
		\end{equation*} 
		we deduce by applying  Lemma \ref{le:a1} with $t_1=0, t_2=t$ that
		\begin{equation*}
		\begin{aligned}
		m(0)\leq v_1(0,x)\leq (AB^{-1}-\delta)m(0).
		\end{aligned}
		\end{equation*}
		Combining this with \eqref{a10} and \eqref{a11} one sees that
		\begin{equation*}
		\begin{aligned}
		r(t,x)&\leq m(0)\big[v_1^{-1}(0,x)+(2A-\delta)t\big]\\
		&\leq (AB^{-1}-\delta)^{-1}+m(0)(2A-\delta)t,
		\end{aligned}
		\end{equation*}
		and
		\begin{equation*}
		\begin{aligned}
		r(t,x)&\geq m(0)\big[v_1^{-1}(0,x)+(2B+\delta)t\big]\\
		&\geq 1+m(0)(2B+\delta)t.
		\end{aligned}
		\end{equation*}
		These two inequalities together with  \eqref{a12} give
		\begin{equation*}
		\begin{aligned}
		&(AB^{-1}-\delta)+m(0)(AB^{-1}-\delta)(2B+\delta)t\leq q(t)\\
		&\leq (AB^{-1}-\delta)^{-1}+m(0)(2A-\delta)t.
		\end{aligned}
		\end{equation*}
		It is easy to see that \(q(t)\) goes to zero (which means \(m(t)\) goes to \(-\infty\)) by letting
		\(t\) go to \((2B+\delta)^{-1}\big(-\inf_{x\in\R}\phi^\prime(x)\big)^{-1}\) on the left hand side and \((AB^{-1}-\delta)^{-1}(2A-\delta)^{-1}\big(-\inf_{x\in\R}\phi^\prime(x)\big)^{-1}\) on the right hand side respectively. 
		On the other hand, since we have already shown 
		\eqref{4}, we deduce that \(v_0(t,x)\) is bounded for all \(t 
		\in [0, T^\prime]\) with any \(T^\prime<T\) in view of \eqref{6}.  This means that a shock of \eqref{eq:main} with the initial data \(u(0,x)=\phi(x)\) occurs at time \(T\) obeying \eqref{shock time}. The blow-up rate \eqref{shock blow-up rate} follows from \eqref{a10} and \eqref{shock time} immediately.   
		It remains to estimate the location of this shock. First from \eqref{a10} and \eqref{7} we see that \(v_1(t,x)\) blows up at some location  \(\bar{x}\in[\bar{x}_1,\bar{x}_2]\). Then we go back to \eqref{eq:particle} to find that 
		\begin{equation*}
		\begin{aligned}
		x_*=:X(T,\bar{x})=\bar{x}+\int_0^Tu^2(t,X(t,\bar{x}))\, d t,
		\end{aligned}
		\end{equation*}
		which immediately yields the estimate \eqref{shock location} via the bound \eqref{6}.
		This completes the proof.

	\end{proof}

We conclude this section by various numerical simulations 
illustrating the shock formation. The numerical approach is identical 
to the one outlined in \cite{KS} to which the reader is referred to for 
details. 

As an example we consider Gaussian initial data $u(x,0)=\exp(-x^{2})$ 
for $\alpha=-0.5$. The solution steepens as would be the case for the 
solution to the Burgers' equation for the same initial data. However, 
there is a considerable difference as can be seen in 
Fig.~\ref{mfKdVm05gausst} on the left. The solution develops a single 
oscillation which forms eventually a cusp as is clearly visible in 
the close-up of the solution at the final time $t=1.211$ on the right 
of the same figure. 
\begin{figure}[htb!]
  \includegraphics[width=0.49\textwidth]{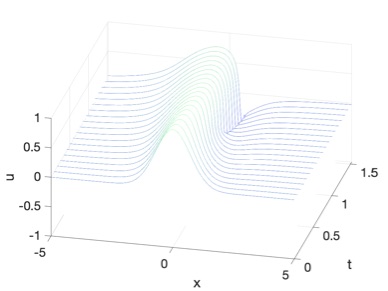}
  \includegraphics[width=0.49\textwidth]{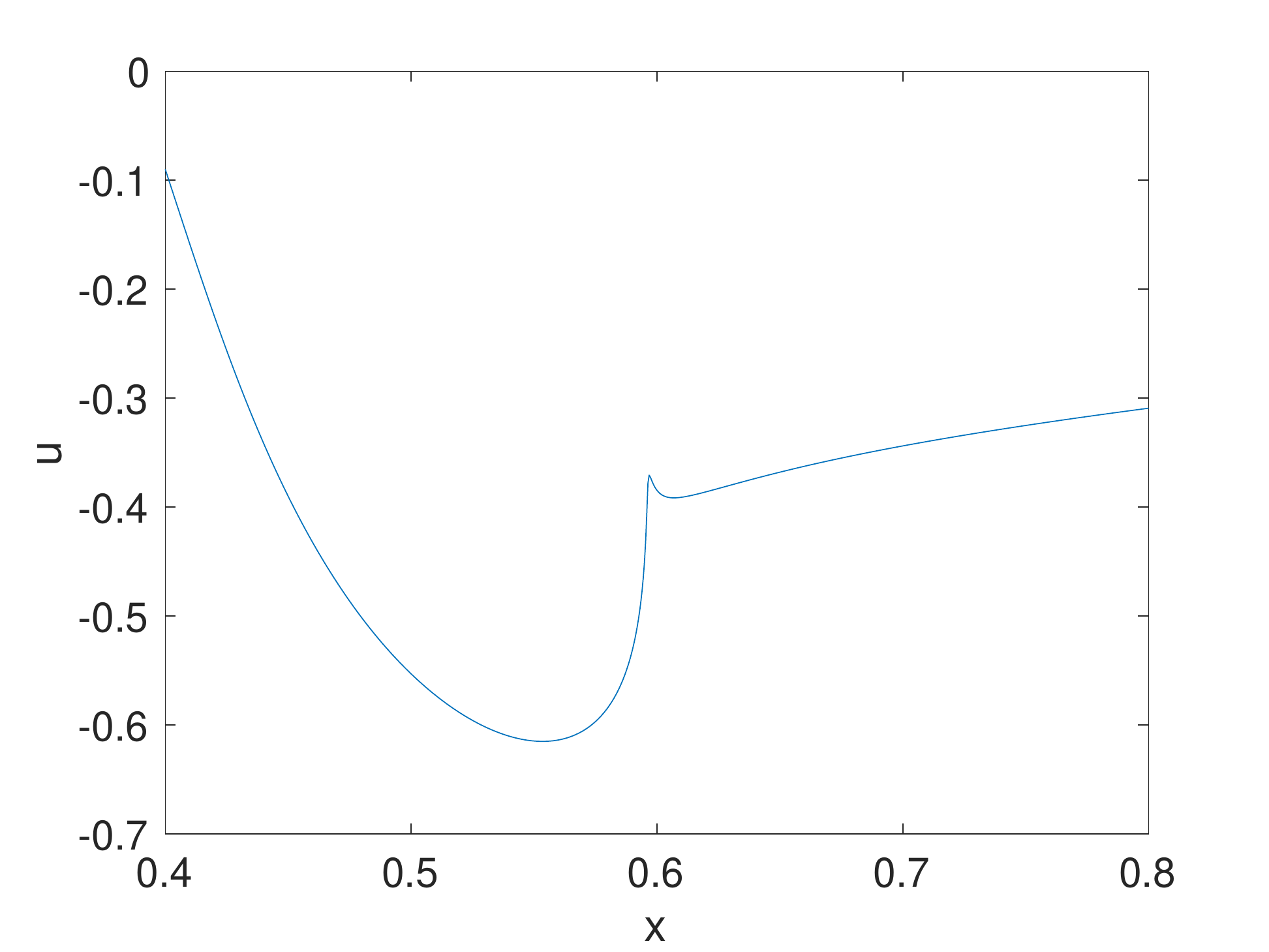}
 \caption{Solution to equation (\ref{eq:main}) with the $+$ sign for $\alpha=-0.5$ and 
 the initial data $u(x,0)=\exp(-x^{2})$ in dependence of $t$ on the 
 left, and a close-up of the solution at the final time $t=1.211$ on 
 the right.}
 \label{mfKdVm05gausst}
\end{figure}

The computation of this cusp formation is numerically challenging. We 
use $N=2^{16}$ Fourier modes for $x\in 5[-\pi,\pi]$ and 
$N_{t}=10^{4}$ time steps for $t\in[0,1.7]$. It is known that the 
Fourier coefficients of an essential singularity $u\sim 
(x-x_{s})^{\mu}$, $\mu\in\mathbb{R}/\mathbb{Z}$ in the complex plane 
for $x=x_{s}$, $x_{s}\in\mathbb{C}$ are of the form 
\begin{equation}
	|\hat{u}|\sim \frac{e^{-\delta |\xi|}}{|\xi|^{1+\mu}}
	\label{mufit}
\end{equation}
for $|\xi|\gg 1$ ($\Im x_{s}=\delta >0$). Sulem, Sulem and Frisch \cite{SSF} used this to 
characterize a singularity in solutions to hyperbolic equations via 
the coefficients of the discrete Fourier transform, see also 
\cite{KR} for a quantitative analysis. We fit the Fourier 
coefficients according to (\ref{mufit}) during the computation. The 
code is stopped once $\delta$ is slightly negative in the fitting. As 
discussed in \cite{KR}, the fitted value of $\mu$ is less reliable, 
but it is clear that it is positive ($\mu\sim0.21$). It cannot be 
excluded that this is compatible with the value $1/3$ known from 
generic shocks in solutions to the Burgers' equation, but in any 
case there is no $L^{\infty}$ blow-up here as expected. 

If we consider the same initial data for equation (\ref{eq:main}) 
with the $-$ sign (with the same numerical parameters), the situation 
changes somewhat. Now the maximum gets compressed into a cusp, 
whereas the oscillation near the maximum which developed the cusp in 
Fig.~\ref{mfKdVm05gausst} stays smooth. A close-up of the cusp at the 
time $0.8695$ can be seen on the right of the same figure. The 
computation is stopped at $t=0.8695$ since the fitting of the Fourier
coefficients according to (\ref{mufit}) produces a negative $\delta$. 
The fitted $\mu\sim0.359$ is close to $1/3$. 
\begin{figure}[htb!]
  \includegraphics[width=0.49\textwidth]{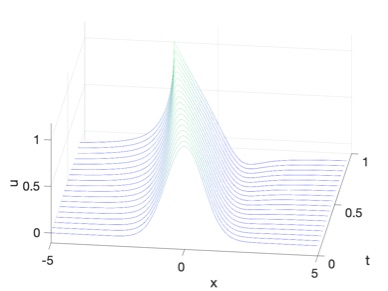}
  \includegraphics[width=0.49\textwidth]{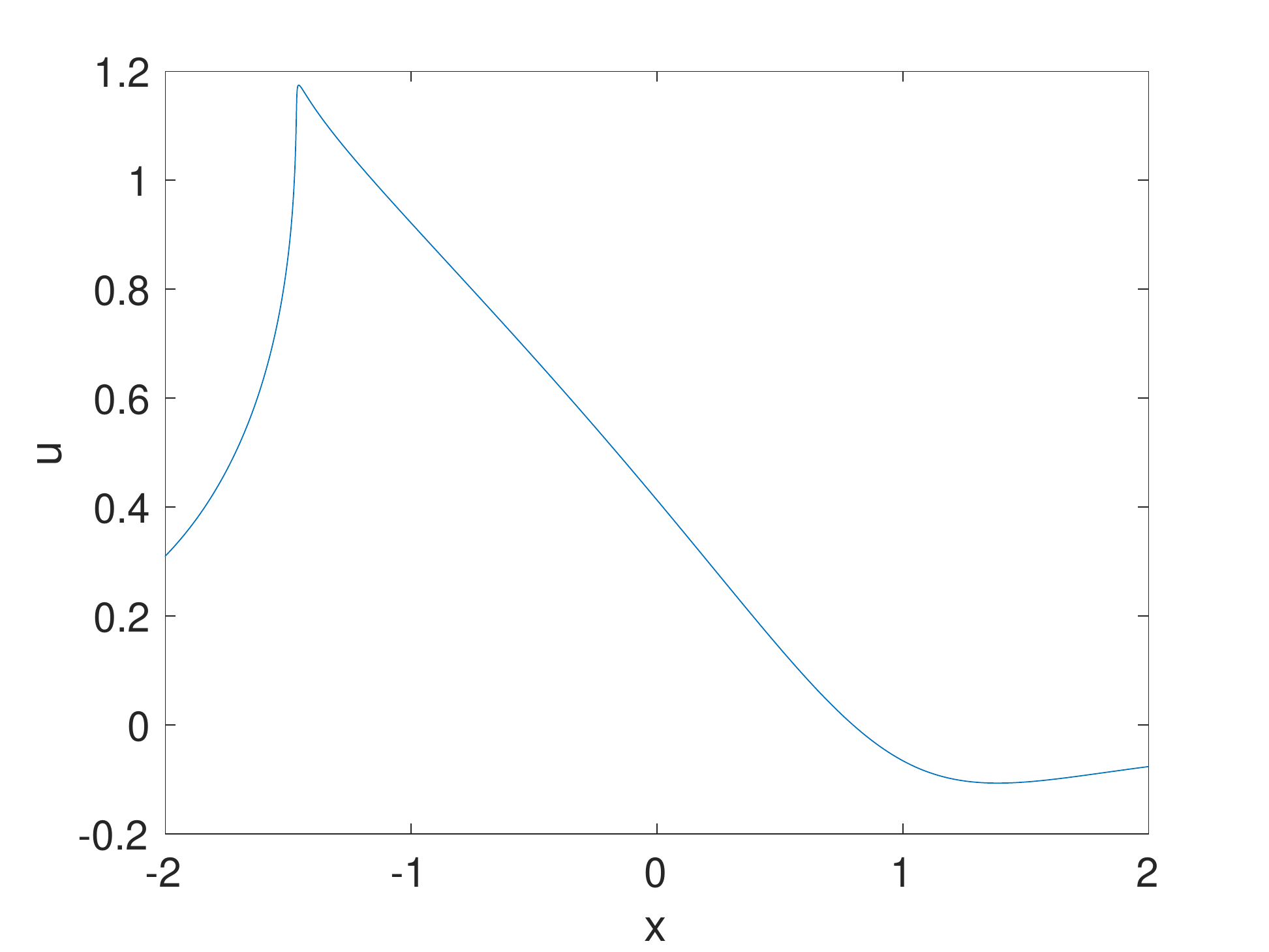}
 \caption{Solution to equation (\ref{eq:main}) with the $-$ sign for $\alpha=-0.5$ and 
 the initial data $u(x,0)=\exp(-x^{2})$ in dependence of $t$ on the 
 left, and a close-up of the solution at the final time $t=0.8695$ on 
 the right.}
 \label{mfKdVm05gausstm}
\end{figure}

If the dispersion is lowered, the situation stays qualitatively 
similar. For the $-$ sign in (\ref{eq:main}) and $\alpha=-0.8$, the code breaks for 
$t=0.8833$, and one gets  a fitted $\mu\sim 0.357$. The solution 
at the final time is shown on the left of Fig.~\ref{mfKdVm08gausst}. 
The same situation for the $+$ sign in (\ref{eq:main}) at time 
$t=2.0027$ is shown on the right of the same figure. The code breaks 
here for the shown time with a fitted $\mu\sim0.34$ (here we chose a 
slightly smaller computational domain $4[-\pi,\pi]$ in order to get 
somewhat higher resolution). 
\begin{figure}[htb!]
  \includegraphics[width=0.32\textwidth]{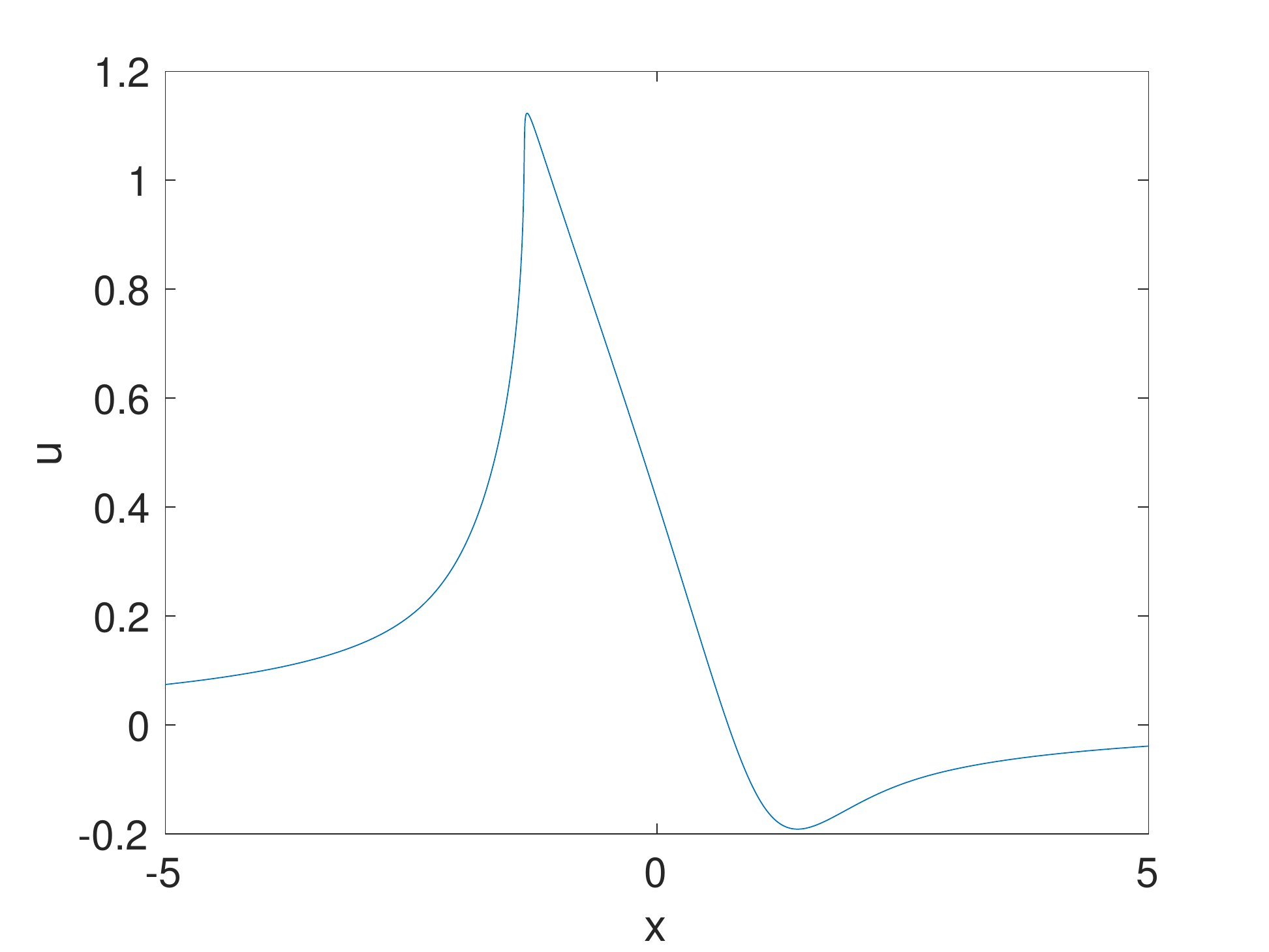}
  \includegraphics[width=0.32\textwidth]{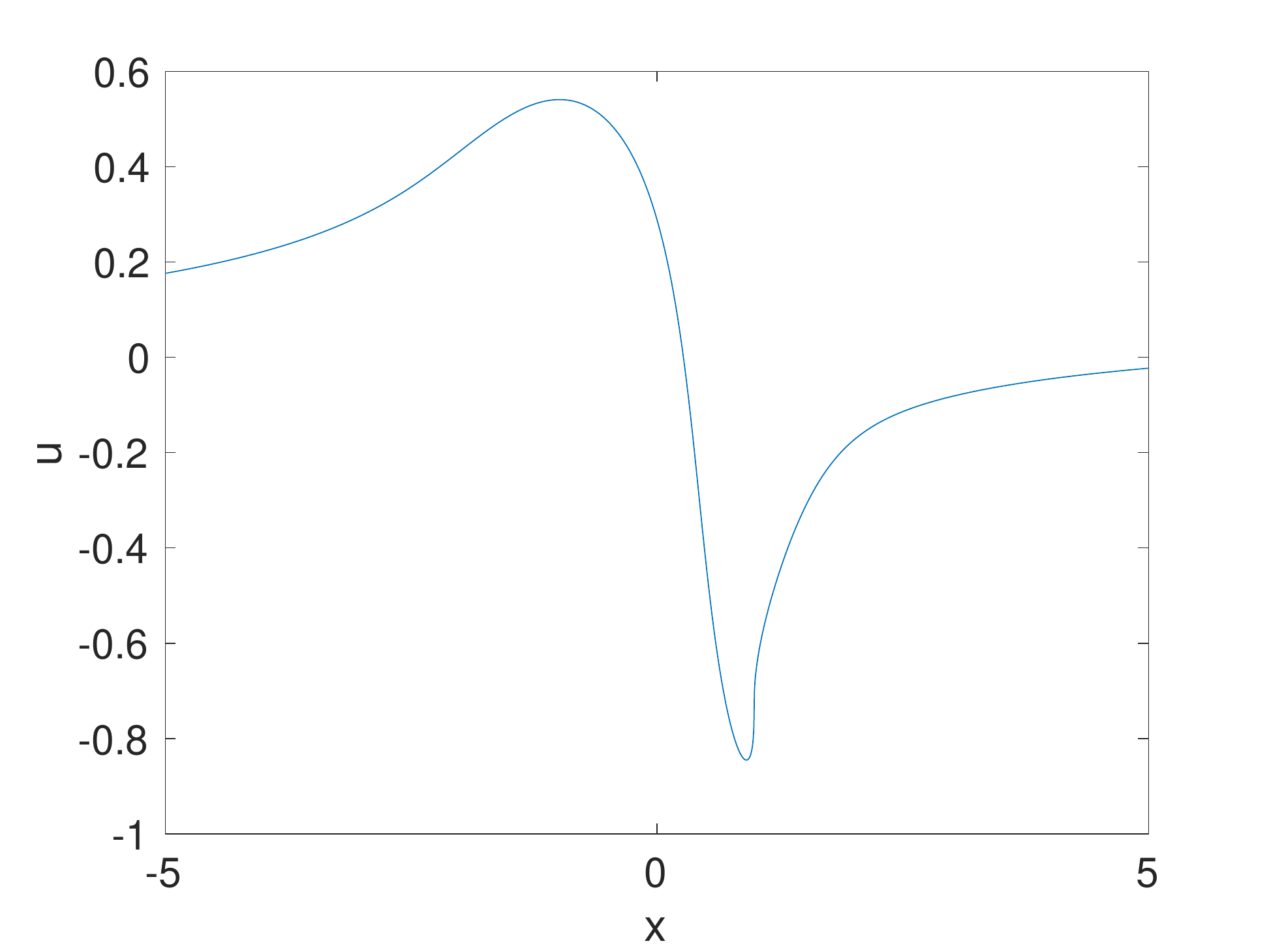}
 \includegraphics[width=0.32\textwidth]{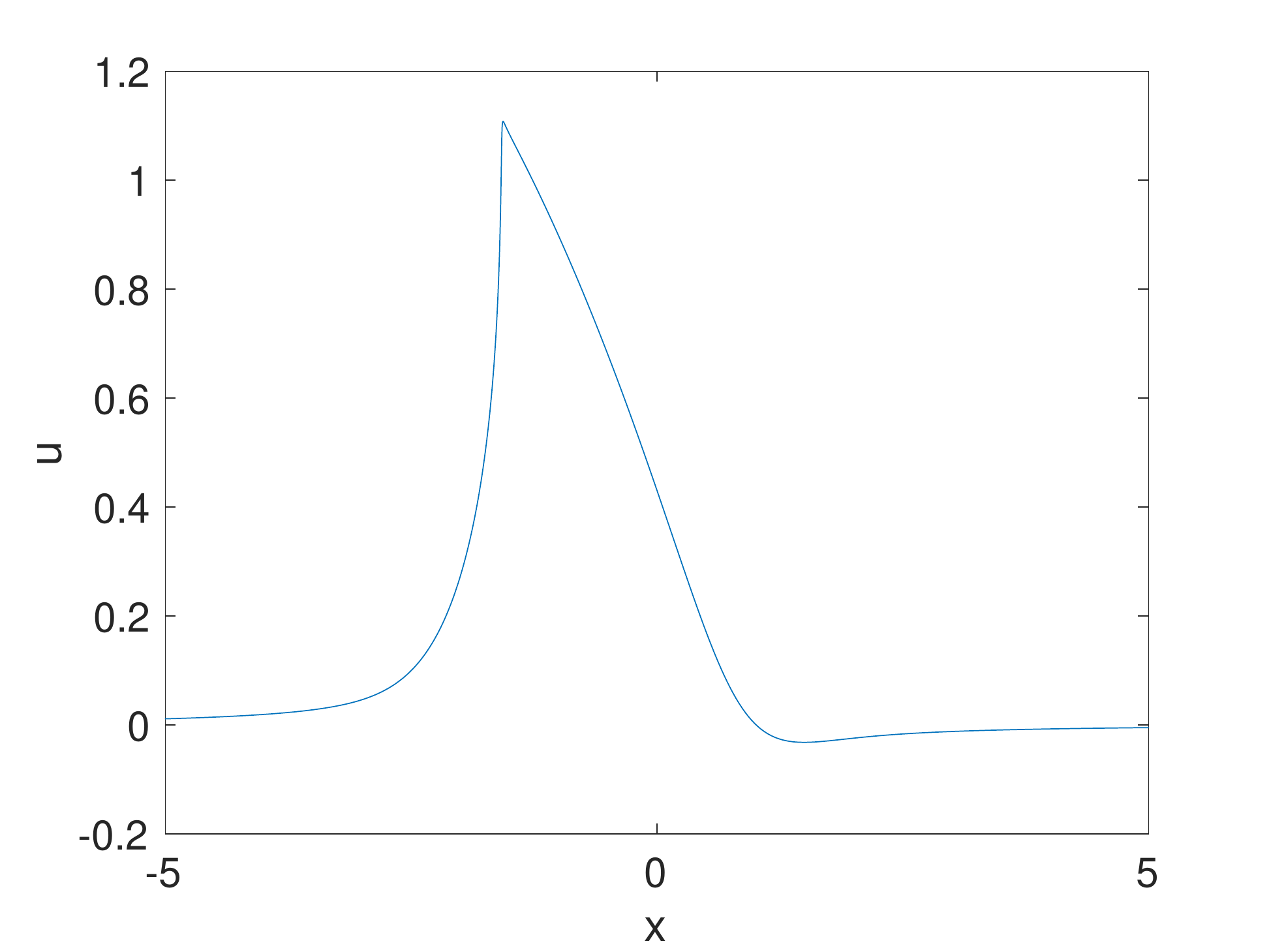}
 \caption{Solution to equation (\ref{eq:main}) for 
 $\alpha=-0.8$ and 
 the initial data $u(x,0)=\exp(-x^{2})$,  with the $-$ sign  at the final time $t=0.8833$ on 
 the left, with the $+$ sign at the final time $t=2.0027$ in the 
 middle; the solution for the equation with the $-$ sign 
 for the same initial data and $\alpha=-0.2$ for 
 (\ref{eq:main}) for $t=0.8468$ on the right.}
 \label{mfKdVm08gausst}
\end{figure}

In the case of stronger dispersion, say $\alpha=-0.2$, the situation 
is again similar for the Gaussian initial data and equation 
(\ref{eq:main}) with the $-$ sign. The code breaks for $t=0.8468$ 
with a fitted $\mu\sim0.29$. The solution at the final time can be 
seen on the right of Fig.~\ref{mfKdVm08gausst}. However, for the 
equation (\ref{eq:main}) with the $+$ sign, we do not get a shock, 
but a \emph{dispersive shock wave}, a zone of rapid modulated 
oscillations near the shock of the corresponding solution for the 
dispersionless equation for the same initial data. The solution for 
$t=1.5$ is shown in Fig.~\ref{mfKdVm02gausst}. This indicates that 
the solution stays smooth for all times in this case. 
\begin{figure}[htb!]
  \includegraphics[width=0.49\textwidth]{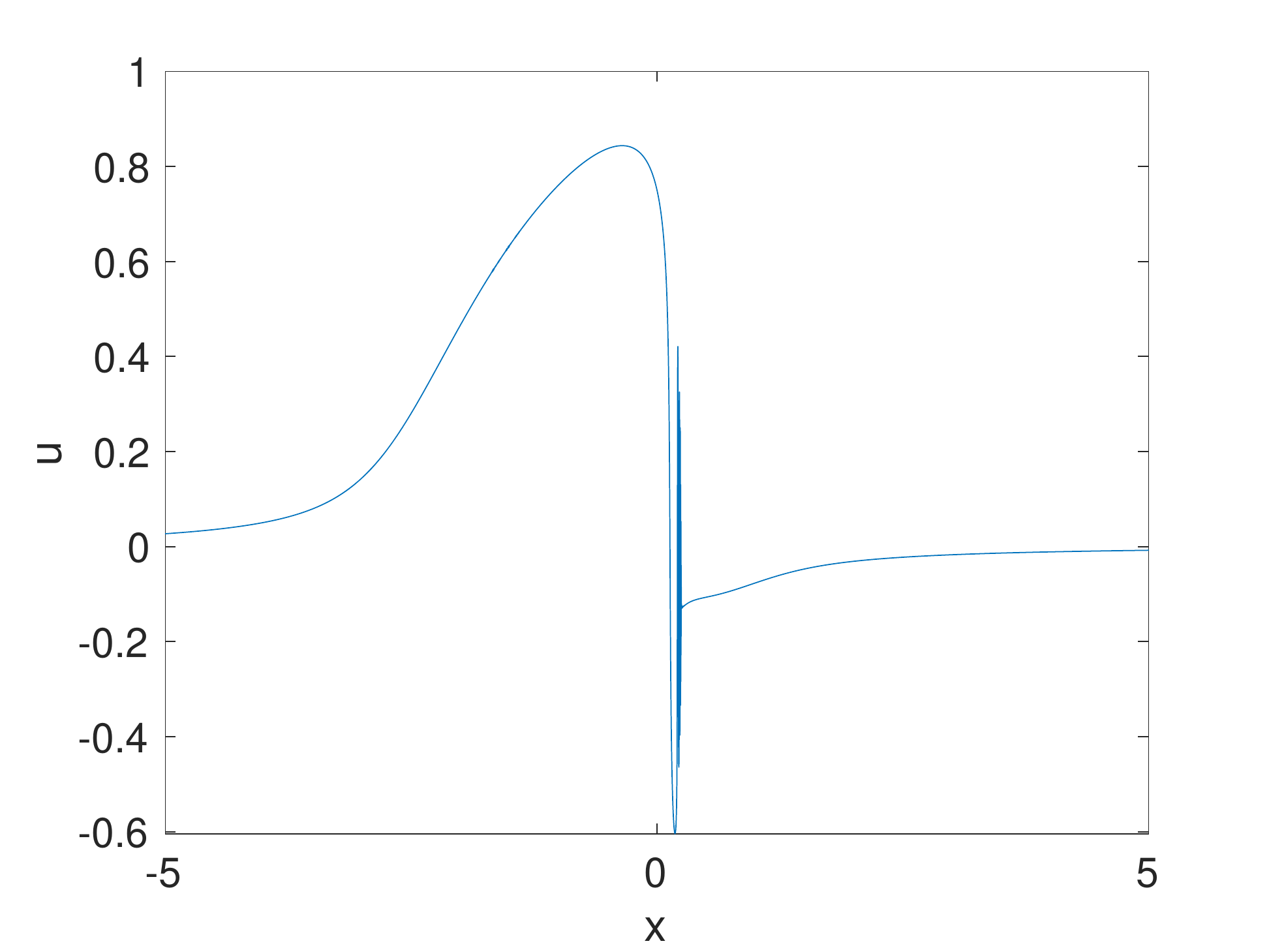}
  \includegraphics[width=0.49\textwidth]{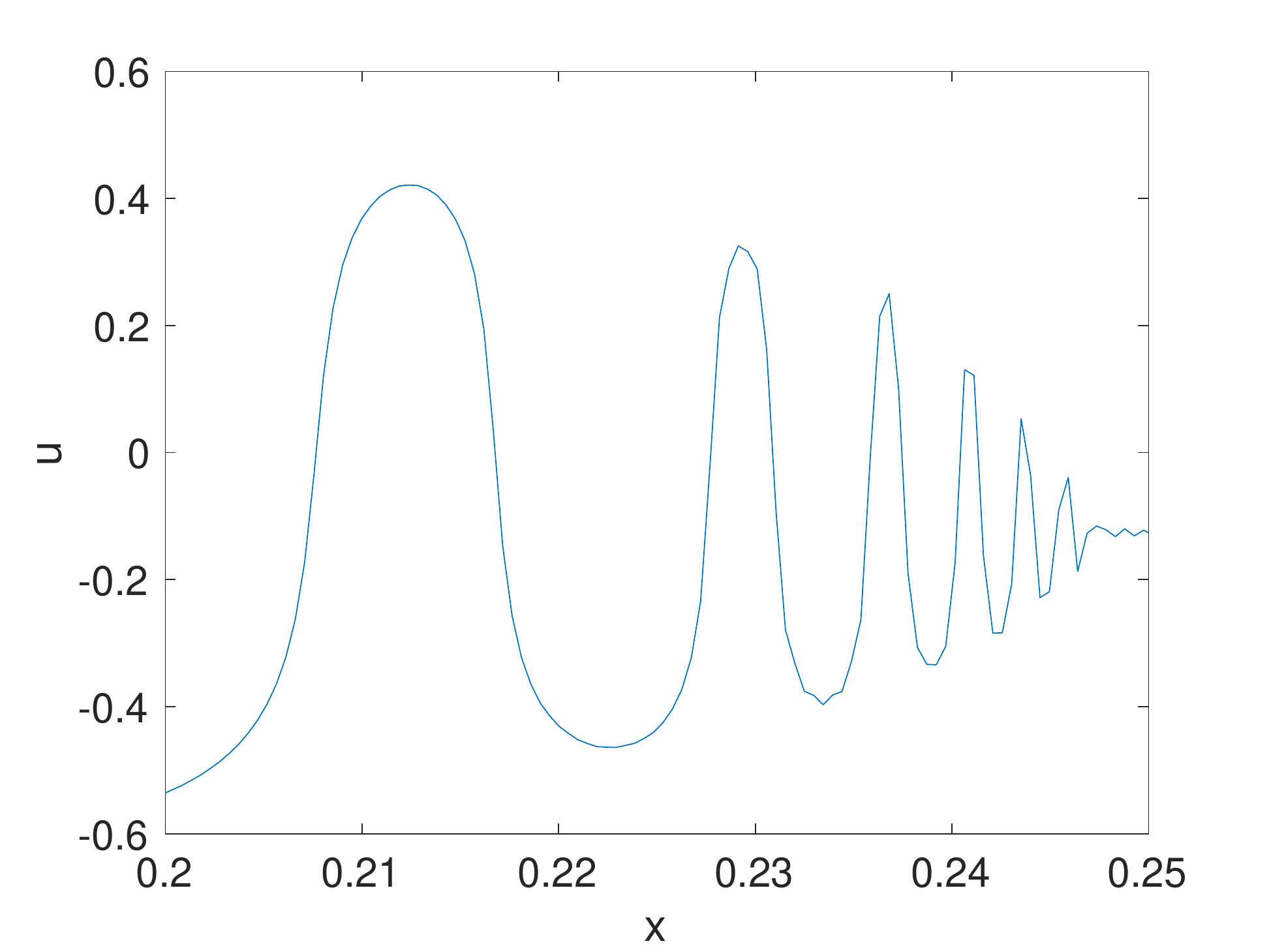}
 \caption{Solution to equation (\ref{eq:main}) with the $+$ sign for 
 $\alpha=-0.2$ and 
 the initial data $u(x,0)=\exp(-x^{2})$ for $t=1.5$ on the 
 left, and a close-up of the solution  on 
 the right.}
 \label{mfKdVm02gausst}
\end{figure}

This of course is consistent with Proposition \ref{th:previous}.

\section{Variants I: shock formation for the modified Burgers-Hilbert and Whitham equations}
%\begin{remark}

In this section we extend the shock formation result in Theorem \ref{th:main} to the modified Burgers-Hilbert and Whitham equations. 

The modified Burgers-Hilbert equation reads as (note that the distinction between focusing and defocusing is irrelevant here):
\begin{align}\label{eq:mBH}
u_t+u^2u_x-\mathcal H u=0,
\end{align}
where $\mathcal H= \text{p.v.}\; \frac{1}{x}$ is the Hilbert transform with Fourier symbol $-i\text{sgn}\;\xi.$ The equation \eqref{eq:mBH} can be formally regarded as the limit case of \eqref{eq:main} by letting \(\alpha\rightarrow -1^+\). 

We first state the result on the shock formation for \eqref{eq:mBH}:
\begin{theorem}\label{th:mBH} \textup{(Rough version)} 
	There exists a wide class of functions \(\phi\in H^2(\R)\) with appropriate large positive amplitude \(\phi\) and negative slope \(\inf_{x\in\R}\phi^\prime(x)\) such that the Cauchy problem for the equation \eqref{eq:mBH} with initial data \(u(0,x)=\phi(x)\) exhibits shock formation. 
	
	\textup{(Precise version)} Let \(\delta\) be a sufficiently small positive number. Assume \(\bar{x}_1\) and \(\bar{x}_2\) are the largest and smallest numbers such that
	\(\overline{\{x: \phi^\prime(x)<0\}}\subset [\bar{x}_1,\bar{x}_2] \).
	Let \(\phi\in H^2(\R)\) satisfy the \textbf{slope condition}
	\begin{align*}
	&-\inf_{x\in\R}\phi^\prime(x)>\delta^{-1}f_1,\\ %\label{mBH1}
	&-\inf_{x\in\R}\phi^\prime(x)>(1-\delta)^{-2}f_2,\\ %\label{mBH2}
	&-\inf_{x\in\R}\phi^\prime(x)>(1-\delta)^{-3}f_3,%\label{mBH3}
	\end{align*}
	and the \textbf{local amplitude condition}
	\begin{align}
	&\phi(x)< B-(1-\delta)^{-2}f_4,\label{mBH4}\\
	&\phi(x)> A+(1-\delta)^{-2}f_4. \label{mBH5}
	\end{align}
	for all \(x\in [\bar{x}_1,\bar{x}_2]\).
	Here the functions \((f_1,f_2,f_3,f_4)\) are homogeneous in each 
	of its arguments of order \((1/2,0,0,0)\), and have the following explicit formulae	
	\begin{equation*}
	\begin{aligned}
	f_1=:f_1(\|\phi\|_{H^2},\|\phi^\prime\|_{L^2},\|\phi^{\prime\prime}\|_{L^2})
	=\big(C_s\|\phi\|_{H^2}+4\|\phi^{\prime}\|_{L^2}+64C_m\|\phi^{\prime\prime}\|_{L^2}\big)^{1/2},
	\end{aligned}
	\end{equation*}
	\begin{equation*}
	\begin{aligned}
	f_2=:f_2(C_0^{-1}\|\phi\|_{L^2},C_0^{-1}\|\phi^\prime\|_{L^2})=6C_0^{-1}\|\phi\|_{L^2}+ 24C_mC_0^{-1}\|\phi^{\prime}\|_{L^2},
	\end{aligned}
	\end{equation*}
	\begin{equation*}
	\begin{aligned}
	f_3=:f_3(C_1^{-1}\|\phi^\prime\|_{L^2},
	C_1^{-1}\|\phi^{\prime\prime}\|_{L^2})=8C_1^{-1}\|\phi^{\prime}\|_{L^2}+ 128C_mC_1^{-1}\|\phi^{\prime\prime}\|_{L^2},
	\end{aligned}
	\end{equation*}
	\begin{equation*}
	\begin{aligned}
	f_4&=:f_4\big(\|\phi\|_{L^2}\big(-\inf_{x\in\R}\phi^\prime(x)\big)^{-1},\|\phi^{\prime}\|_{L^2}\big(-\inf_{x\in\R}\phi^\prime(x)\big)^{-1}\big)\\
	&=\big(-\inf_{x\in\R}\phi^\prime(x)\big)^{-1}(3\|\phi\|_{L^2}+12C_m\|\phi^{\prime}\|_{L^2}),
	\end{aligned}
	\end{equation*}
	where \(C_0>0\) and \(C_1>0\) satisfying	
	\begin{equation*}\label{mBH6}
	\begin{aligned}
	\|\phi\|_{L^\infty}\leq \frac{C_0}{2},\quad 
	\|\phi^\prime\|_{L^\infty}\leq \frac{C_1}{2},
	\end{aligned}
	\end{equation*}	
	and \(A>0\) and \(B>0\) satisfying
	\begin{equation*}\label{mBH7}
	\begin{aligned}
	A^2>8\delta B^2,\quad	 B> A+2(1-\delta)^{-2}f_4,
	\end{aligned}
	\end{equation*}
	and \(C_s\) and \(C_m\) are the best embedding constants of Sobolev inequality and Morerry inequality
	\begin{equation*}
	\begin{aligned}
	\|f\|_{L^\infty(\R)}\leq C_s\|f\|_{H^1(\R)},
	\end{aligned}
	\end{equation*}
	and
	\begin{equation*}
	\begin{aligned}
	|f|_{\dot C^{0,\frac{1}{2}}(\R)}\leq C_m\|f_x\|_{L^2(\R)},
	\end{aligned}
	\end{equation*} 
	where \(|\cdot|_{\dot C^{0,\frac{1}{2}}(\R)}\) the usual H\"older semi norm.  
	Then the solution for the equation  \eqref{eq:mBH} with initial data \(u(0,x)=\phi(x)\) exhibits shock formation at some time \(T>0\) with
	\begin{equation*}\label{shock time of mBH}
	\begin{aligned}
	(2B+\delta)^{-1}\big(-\inf_{x\in\R}\phi^\prime(x)\big)^{-1}
	<T<(AB^{-1}-\delta)^{-1}(2A-\delta)^{-1}\big(-\inf_{x\in\R}\phi^\prime(x)\big)^{-1},
	\end{aligned}
	\end{equation*} 
	and at some location \(x_*\) satisfying
	\begin{align*}
	\bar{x}_1-C_0^2T
	\leq x_*\leq \bar{x}_2+C_0^2T.
	\end{align*}	
	Moreover, we have the blow-up rate estimate
	\begin{align*}
	(2B+\delta)^{-1}(T-t)^{-1}\leq\|\partial_xu(\cdot,t)\|_{L^\infty}\leq (AB^{-1}-\delta)^{-1}(2A-\delta)^{-1}(T-t)^{-1}, 
	\end{align*}
	as \(t\rightarrow T^{-}\).
\end{theorem}

It is standard to show that the Cauchy problem for the equation \eqref{eq:mBH} with the initial data \(u(0,x)=\phi(x)\) is well-posed in the class \(C([0,T): H^2(\R))\) for some \(T>0\) which will be denoted the maximal time of existence in the following. Using the same notations \(X(t,x),v_0(t,x),v_1(t,x),m(t)\) and \(q(t)\) as \eqref{eq:particle}-\eqref{1}, it then 
follows from \eqref{eq:mBH} that \eqref{3.1} and \eqref{3.2} hold, in which \(K_0(t,x)\) in \eqref{3.3} and \(K_1(t,x)\) in \eqref{3.4} shall be respectively replaced by
\begin{align}
&K_0(t,x)
=\int_\R\frac{\mathrm{sgn}(y)}{|y|}[u(t,X(t,x))-u(t,X(t,x)-y)]\, d y,\label{mBH8}\\
&K_1(t,x)
=\int_\R \frac{\mathrm{sgn}(y)}{|y|}[\partial_xu(t,X(t,x))-\partial_xu(t,X(t,x)-y)]\, d y \label{mBH9}.
\end{align}

Analogously, to prove Theorem \ref{th:mBH}, we need to show \eqref{4}
via a contradiction argument by assuming \eqref{5} conversely. There are two main tasks in closing the proof of \eqref{5}. The first task is to estimate the bound of \(q(t)\), indeed one can check that Lemma \ref{le:a1} and Lemma \ref{le:a2} still hold. 
The other task is to show the following estimates on the nonlocal terms:	
\begin{lemma} For all \(t\in[0,T_2]\) and \(x\in\R\),  we have 
	\begin{equation}\label{mBH-K_0}
	\begin{aligned}
	|K_0(t,x)|\leq (2\|\phi\|_{L^2}+8C_m\|\phi^{\prime}\|_{L^2})q(t)^{-\frac{1}{3}},
	\end{aligned}
	\end{equation}
	and
	\begin{equation}\label{mBH-K_1}
	\begin{aligned}
	|K_1(t,x)|
	\leq (4\|\phi^{\prime}\|_{L^2}+64C_m\|\phi^{\prime\prime}\|_{L^2})q(t)^{-2},
	\end{aligned}
	\end{equation}
	where \(K_0(t,x)\) and \(K_1(t,x)\) are defined by \eqref{mBH8} and \eqref{mBH9} respectively.	
\end{lemma}

\begin{proof} Benefiting from the local amplitude condition \eqref{mBH4}-\eqref{mBH5}, the solution \(u\) can still live in \(L^2(\R)\). The proof of \eqref{mBH-K_0} and \eqref{mBH-K_1} is identical to that of \cite{SW2}. 
	
\end{proof}

We now turn to the modified Whitham equation which reads
\begin{equation}\label{eq:mWhitham}
\partial_tu+u^2\partial_xu+\int_\R K(x-y)\partial_y u(y,t)\,dy=0,
\end{equation}
where 
\begin{align*}
K(x)=\frac{1}{\sqrt{2\pi}}\int_\R e^{\mathrm{i}x\xi}\sqrt{\frac{\tanh \xi}{\xi}}\, d \xi.
\end{align*}

In order to deal with the modified Whitham equation  we first collect the following property of \(K(x)\) \cite{Hur} (originally appeared in \cite{EWah}):

\vspace{0.3cm}

\begin{lemma}\label{le:Hur} There exist constants \(L_0, L_\infty>0\) such that
	\begin{equation*}
	\begin{aligned}
	K(x)\leq \frac{L_0}{\sqrt{|x|}}\quad   \mathrm{and}\quad |K^\prime(x)|\leq \frac{L_0}{\sqrt{|x|^3}},\quad  \mathrm{for}\   0<|x|\leq 1,
	\end{aligned}
	\end{equation*}
	and
	\begin{equation*}
	\begin{aligned}
	\int_1^\infty |K^\prime(x)|\, d x\leq L_\infty.
	\end{aligned}
	\end{equation*}	
\end{lemma}

We now can state the result on the shock formation of \eqref{eq:mWhitham}:
\begin{theorem}\label{th:mWhitham} \textup{(Rough version)} 
	There exists a wide class of functions \(\phi\in H^3(\R)\) with appropriate large positive amplitude \(\phi\) and negative slope \(\inf_{x\in\R}\phi^\prime(x)\) such that the Cauchy problem for the equation \eqref{eq:mWhitham} with initial data \(u(0,x)=\phi(x)\) exhibits shock formation. 
	
	\textup{(Precise version)} Let \(\delta\) be a sufficiently small positive number. Assume \(\bar{x}_1\) and \(\bar{x}_2\) are the largest and smallest numbers such that
	\(\overline{\{x: \phi^\prime(x)<0\}}\subset [\bar{x}_1,\bar{x}_2] \).
	Let \(\phi\in H^3(\R)\) satisfy the \textbf{slope condition}
	\begin{align*}
	&-\inf_{x\in\R}\phi^\prime(x)>\delta^{-1}f_1,\\
	&-\inf_{x\in\R}\phi^\prime(x)>(1-\delta)^{-2}f_2,\\ 
	&-\inf_{x\in\R}\phi^\prime(x)>(1-\delta)^{-3}f_3,
	\end{align*}
	and the \textbf{local amplitude condition}
	\begin{align*}
	&\phi(x)< B-(1-\delta)^{-2}f_4,\\
	&\phi(x)> A+(1-\delta)^{-2}f_4. 
	\end{align*}
	for all \(x\in [\bar{x}_1,\bar{x}_2]\).
	Here the functions \((f_1,f_2,f_3,f_4)\) are homogeneous in its  argument of order \((1/2,0,0,0)\), and have the following explicit formulae	
	\begin{equation*}
	\begin{aligned}
	&\quad f_1=:f_1(\|\phi\|_{H^2},\|\phi^{\prime\prime\prime}\|_{L^2},C_1)\\
	&=\big[C_s\|\phi\|_{H^2}+2(3L_0+L_\infty)
	+4L_0(2^{-1}AB^{-1})^{-7BA^{-2}/4}C_1^{-1}\|\phi^{\prime\prime\prime}\|_{L^2}\big]^{1/2},
	\end{aligned}
	\end{equation*}
	\begin{equation*}
	\begin{aligned}
	f_2=:f_2(C_0^{-1}C_1)=8(3L_0+L_\infty)+16L_0C_0^{-1}C_1,
	\end{aligned}
	\end{equation*}
	\begin{equation*}
	\begin{aligned}
	&\quad f_3=:f_3(
	C_1^{-1}\|\phi^{\prime\prime\prime}\|_{L^2})\\
	&=4(3L_0+L_\infty)
	+8L_0(2^{-1}AB^{-1})^{-7BA^{-2}/4}
	C_1^{-1}\|\phi^{\prime\prime\prime}\|_{L^2},
	\end{aligned}
	\end{equation*}
	\begin{equation*}
	\begin{aligned}
	&\quad f_4=:f_4\big(C_0\big(-\inf_{x\in\R}\phi^\prime(x)\big)^{-1},C_1\big(-\inf_{x\in\R}\phi^\prime(x)\big)^{-1}\big)\\
	&=\big(-\inf_{x\in\R}\phi^\prime(x)\big)^{-1}[4(3L_0+L_\infty)C_0+8L_0C_1],
	\end{aligned}
	\end{equation*}
	where \(C_0>0\) and \(C_1>0\) satisfying	
	\begin{equation*}
	\begin{aligned}
	\|\phi\|_{L^\infty}\leq \frac{C_0}{2},\quad 
	\|\phi^\prime\|_{L^\infty}\leq \frac{C_1}{2},
	\end{aligned}
	\end{equation*}	
	and \(A>0\) and \(B>0\) satisfying
	\begin{equation*}
	\begin{aligned}
	A^2>8\delta B^2,\quad	 B> A+2(1-\delta)^{-2}f_4,
	\end{aligned}
	\end{equation*}
	and \(C_s\) and \(C_m\) are the best  constants in the  Sobolev inequality and Morrey inequality
	\begin{equation*}
	\begin{aligned}
	\|f\|_{L^\infty(\R)}\leq C_s\|f\|_{H^1(\R)},
	\end{aligned}
	\end{equation*}
	and
	\begin{equation*}
	\begin{aligned}
	|f|_{\dot C^{0,\frac{1}{2}}(\R)}\leq C_m\|f_x\|_{L^2(\R)},
	\end{aligned}
	\end{equation*} 
	where \(|\cdot|_{\dot C^{0,\frac{1}{2}}(\R)}\) is the usual H\"older semi norm.  
	Then the solution of equation \eqref{eq:mWhitham} with initial data \(u(0,x)=\phi(x)\) exhibits shock formation at some time \(T>0\) with
	\begin{align*}
	(2B+\delta)^{-1}\big(-\inf_{x\in\R}\phi^\prime(x)\big)^{-1}
	<T<(AB^{-1}-\delta)^{-1}(2A-\delta)^{-1}\big(-\inf_{x\in\R}\phi^\prime(x)\big)^{-1},
	\end{align*}
	and at some location \(x_*\) satisfying
	\begin{align*}
	\bar{x}_1-C_0^2T
	\leq x_*\leq \bar{x}_2+C_0^2T.
	\end{align*}	
	Moreover, we have the blow-up rate estimate
	\begin{align*}
	(2B+\delta)^{-1}(T-t)^{-1}\leq\|\partial_xu(\cdot,t)\|_{L^\infty}\leq (AB^{-1}-\delta)^{-1}(2A-\delta)^{-1}(T-t)^{-1}, 
	\end{align*}
	as \(t\rightarrow T^{-}\).
\end{theorem}

It is easy to show that the Cauchy problem for the equation \eqref{eq:mWhitham} with the initial data \(u(0,x)=\phi(x)\) is well-posed in the class \(C([0,T): H^3(\R))\) for some \(T>0\) which will now denote  the maximal time of existence in what follows. Using the same notations \(X(t,x),v_0(t,x),v_1(t,x),m(t)\) and \(q(t)\) as \eqref{eq:particle}-\eqref{1}, it then 
follows from \eqref{eq:mWhitham} that \eqref{3.1} and \eqref{3.2} hold, in which \(K_0(t,x)\) in \eqref{3.3} and \(K_1(t,x)\) in \eqref{3.4} shall be respectively replaced by
\begin{align}
&K_0(t,x)=\int_\R K(y)\partial_xu(t,X(t,x)-y)\, d y,\label{mWhitham1}\\
&K_1(t,x)=\int_\R K(y)\partial_x^2u(t,X(t,x)-y)\, d y \label{mWhitham2}.
\end{align}
Arguing as above for the equation \eqref{eq:mBH}, to complete the proof of 
Theorem \ref{th:mWhitham}, it suffices to show the following:

\begin{lemma} For all \(t\in[0,T_2]\) and \(x\in\R\),  we have 
	\begin{equation}\label{mWhitham-K_0}
	\begin{aligned}
	|K_0(t,x)|\leq 2\big[C_0(3L_0+L_\infty)+2L_0C_1\big]q(t)^{-\frac{1}{2}},
	\end{aligned}
	\end{equation}
	and
	\begin{equation}\label{mWhitham-K_1}
	\begin{aligned}
	|K_1(t,x)|
	\leq \big[2C_1(3L_0+L_\infty)+4L_0(2^{-1}AB^{-1})^{-7BA^{-2}/4}\|\phi^{\prime\prime\prime}\|_{L^2}\big]
	q(t)^{-2},
	\end{aligned}
	\end{equation}
	where \(K_0(t,x)\) and \(K_1(t,x)\) are defined by \eqref{mWhitham1} and \eqref{mWhitham2} respectively.	
\end{lemma}

\begin{proof} The proof of \eqref{mWhitham-K_0} can be found in \cite{SW2}. 
	Because of the stronger nonlinear term in  equation \eqref{eq:mWhitham},  
	the proof here is simpler when  dealing with \(\|\partial_x^2u\|_{L^\infty}\) to estimate \(K_1(t,x)\) than that of  \cite{SW2}. We only focus on the proof of \eqref{mWhitham-K_1}. With a parameter \(\eta\) being specified later, one splits the integration into the following form:
	\begin{equation*}
	\begin{aligned}
	K_1(t,x)&=\int_{|y|\leq\eta}K(y)\partial_x^2u(t,X(t,x)-y)\, d y\\
	&\quad+\int_{|y|>\eta}K(y)\partial_x^2u(t,X(t,x)-y)\, d y\\
	&=\colon D_1+D_2.
	\end{aligned}
	\end{equation*}
	Applying Lemma \ref{le:Hur}, it is straightforward to see that 
	\begin{equation}\label{mWhitham3}
	\begin{aligned}
	|D_1|
	\leq 2\int_{|y|\leq\eta} \frac{L_0}{\sqrt{|y|}}\, d y \cdot\|\partial_x^2u\|_{L^\infty}\leq 4L_0\eta^{\frac{1}{2}}\|\partial_x^2u\|_{L^\infty},
	\end{aligned}
	\end{equation}
	We use integration by parts to find that
	\begin{equation}\label{mWhitham4}
	\begin{aligned}
	|D_2|
	&\leq\big|K(\eta)[\partial_xu(t,X(t,x)-\eta)-\partial_xu(t,X(t,x)+\eta)]\big|\\
	&\quad+\bigg|\int_{\eta<|y|\leq 1} K^\prime(y)\partial_xu(t,X(t,x)-y)\, d y\bigg|\\
	&\quad+\bigg|\int_{|y|>1} K^\prime(y)\partial_xu(t,X(t,x)-y)\, d y\bigg|\\
	&\leq 2L_0\eta^{-\frac{1}{2}}\|v_1\|_{L^\infty}+4L_0(\eta^{-\frac{1}{2}}-1)\|v_1\|_{L^\infty}
	+2L_\infty\|v_1\|_{L^\infty}\\
	&\leq 2(3L_0\eta^{-\frac{1}{2}}+L_\infty)\|v_1\|_{L^\infty}
	\leq 2C_1(3L_0+L_\infty)\eta^{-\frac{1}{2}}q^{-1}(t),
	\end{aligned}
	\end{equation}
	where Lemma \ref{le:Hur} was used.

	Again we use \(\|\partial_x^3u\|_{L^2}\) to control \(\|\partial_x^2u\|_{L^\infty}\). Using the property that \(K(\cdot)\) is even, manipulating as \eqref{the third derivative} and \eqref{18}, one still may estimate
	\begin{equation}\label{mWhitham5}
	\begin{aligned}
	\|\partial_x^3u\|_{L^2}
	\leq (2^{-1}AB^{-1})^{-7BA^{-2}/4}\|\phi^{\prime\prime\prime}\|_{L^2}
	q(t)^{-7BA^{-2}/4}.
	\end{aligned}
	\end{equation}
	Substituting \eqref{mWhitham5} into \eqref{mWhitham3} gives
	\begin{equation}\label{mWhitham6}
	\begin{aligned}
	|D_1|
	\leq 4L_0C_s(2^{-1}AB^{-1})^{-7BA^{-2}/4}\|\phi^{\prime\prime\prime}\|_{L^2}\eta^{\frac{1}{2}}q(t)^{-7BA^{-2}/4}.
	\end{aligned}
	\end{equation}
	Minimizing 	\eqref{mWhitham4} and \eqref{mWhitham6} by taking \(\eta=q(t)^{-1+7BA^{-2}/4}\), one obtains
	\begin{equation*}
	\begin{aligned}
	|K_1(t,x)|
	&\leq \big[2C_1(3L_0+L_\infty)+4L_0(2^{-1}AB^{-1})^{-7BA^{-2}/4}\|\phi^{\prime\prime\prime}\|_{L^2}\big]\\
	&\quad\times q(t)^{7BA^{-2}/8-1/2}\\
	&\leq \big[2C_1(3L_0+L_\infty)+4L_0(2^{-1}AB^{-1})^{-7BA^{-2}/4}\|\phi^{\prime\prime\prime}\|_{L^2}\big]
	q(t)^{-2}.
	\end{aligned}
	\end{equation*}
	This finishes the proof of \eqref{mWhitham-K_1}.

\end{proof}

We illustrate the shock formation in solutions of the modified 
Whitham equation for the example of Gaussian initial data. In this 
case the code is stopped for $t=0.8511$ since the Fourier 
coefficients are no longer exponentially decreasing. The solution at 
this time can be seen in Fig.~\ref{mwhithamgauss}. A fitting of the 
Fourier coefficients
according to (\ref{mufit}) yields a $\mu\sim0.38$. Thus the shock 
formation appears to be as for the modified fKdV equation with 
negative $\alpha$. 
\begin{figure}[htb!]
  \includegraphics[width=0.49\textwidth]{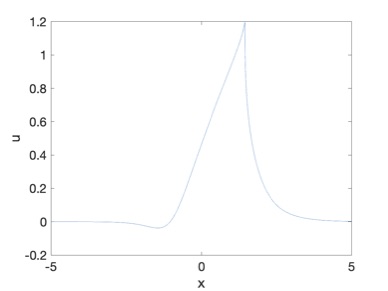}
 \caption{Solution to the modified Whitham equation for the initial 
 data $u(x,0)=\exp(-x^{2})$ for $t=0.8511$. }
 \label{mwhithamgauss}
\end{figure}

\section{Variants II: shock formation for the generalized fractional Korteweg-de Vries equation}
%\begin{remark}

The generalized fractional Korteweg-de Vries equation reads :
\begin{equation}\label{eq:gKdV}
u_t+u^pu_x-|D|^\alpha\partial_xu=0,\; p=3,4,...
\end{equation}
The result on the shock formation of \eqref{eq:gKdV} can be stated as follows:
\begin{theorem}\label{th:gKdV} \textup{(Rough version)} Let  \(p>1\) and \(\alpha\in(-1,0)\). 
	There exists a wide class of functions \(\phi\in H^3(\R)\) with appropriate large positive amplitude \(\phi\) and negative slope \(\inf_{x\in\R}\phi^\prime(x)\) such that the Cauchy problem for  equation \eqref{eq:gKdV} with initial data \(u(0,x)=\phi(x)\) exhibits shock formation. 
	
	\textup{(Precise version)} Let \(p>1\) and \(\alpha\in(-1,0)\), and \(\delta\) be a sufficiently small positive number. Assume \(\bar{x}_1\) and \(\bar{x}_2\) are the largest and smallest numbers such that
	\(\overline{\{x: \phi^\prime(x)<0\}}\subset [\bar{x}_1,\bar{x}_2] \).
	Let \(\phi\in H^3(\R)\) satisfy the \textbf{slope condition}
	\begin{align*}
	&-\inf_{x\in\R}\phi^\prime(x)>\delta^{-1}f_1,\\
	&-\inf_{x\in\R}\phi^\prime(x)>(1-\delta)^{-2}f_2,\\
	&-\inf_{x\in\R}\phi^\prime(x)>(1-\delta)^{-3}f_3,
	\end{align*}
	and the \textbf{local amplitude condition}
	\begin{align*}
	&\phi(x)< B-(1-\delta)^{-2}f_4,\\
	&\phi(x)> A+(1-\delta)^{-2}f_4, 
	\end{align*}
	for all \(x\in [\bar{x}_1,\bar{x}_2]\).
	Here the functions \((f_1,f_2,f_3,f_4)\) are homogeneous in each argument of order \((1/2,0,0,0)\), and have the following explicit formulae	
	\begin{equation*}
	\begin{aligned}
	f_1&=:f_1(\|\phi\|_{H^2},C_1,\|\phi^{\prime\prime\prime}\|_{L^2})
	=\big[C_s\|\phi\|_{H^2}+4C_1(1+\alpha)^{-1}\\
	&\quad+2C_s(-\alpha)^{-1}(2^{-1}A^{p-1}B^{1-p})^{-7B^{p-1}/(2pA^{2p-2})}
	\|\phi^{\prime\prime\prime}\|_{L^2}\big]^{1/2},
	\end{aligned}
	\end{equation*}
	\begin{equation*}
	\begin{aligned}
	f_2=:f_2(C_0^{-1}C_1)=8(-\alpha(1+\alpha))^{-1}+4\alpha^{-2}C_0^{-1}C_1,
	\end{aligned}
	\end{equation*}
	\begin{equation*}
	\begin{aligned}
	&\quad f_3=:f_3(C_1^{-1}\|\phi^{\prime\prime\prime}\|_{L^2})\\
	&=8(1+\alpha)^{-1}+4C_s(-\alpha)^{-1}(2^{-1}A^{p-1}B^{1-p})^{-7B^{p-1}/(2pA^{2p-2})}
	C_1^{-1}\|\phi^{\prime\prime\prime}\|_{L^2},
	\end{aligned}
	\end{equation*}
	\begin{equation*}
	\begin{aligned}
	&\quad f_4=:f_4\big(C_0\big(-\inf_{x\in\R}\phi^\prime(x)\big)^{-1},C_1\big(-\inf_{x\in\R}\phi^\prime(x)\big)^{-1}\big)\\
	&=\big(-\inf_{x\in\R}\phi^\prime(x)\big)^{-1}\big[4C_0\big(-\alpha(1+\alpha)\big)^{-1}+2C_1\alpha^{-2}\big],
	\end{aligned}
	\end{equation*}
	where \(C_0>0\) and \(C_1>0\) satisfying	
	\begin{equation*}
	\begin{aligned}
	\|\phi\|_{L^\infty}\leq \frac{C_0}{2},\quad 
	\|\phi^\prime\|_{L^\infty}\leq \frac{C_1}{2},
	\end{aligned}
	\end{equation*}	
	and \(A>0\) and \(B>0\) satisfying
	\begin{equation}\label{con:A,B}
	\begin{aligned}
	A^{2p-2}>8\delta B^{2p-2},\quad 4pA^{2p-2}>7B^{p-1},\quad B> A+2(1-\delta)^{-2}f_4,
	\end{aligned}
	\end{equation}
	and \(C_s\) is the best  constant of the Sobolev inequality
	\begin{equation*}
	\begin{aligned}
	\|f\|_{L^\infty(\R)}\leq C_s\|f\|_{H^1(\R)}.
	\end{aligned}
	\end{equation*}   
	Then the solution for the equation \eqref{eq:gKdV} with initial data \(u(0,x)=\phi(x)\) exhibits shock formation at some time \(T>0\) satisfying
	\begin{equation*}
	\begin{aligned}
	&(pB^{p-1}+\delta)^{-1}\big(-\inf_{x\in\R}\phi^\prime(x)\big)^{-1}
	<T\\
	&<(A^{p-1}B^{1-p}-\delta)^{-1}(pA^{p-1}-\delta)^{-1}\big(-\inf_{x\in\R}\phi^\prime(x)\big)^{-1},
	\end{aligned}
	\end{equation*}
	and at some location \(x_*\) satisfying
	\begin{align}
	\bar{x}_1-C_0^pT
	\leq x_*\leq \bar{x}_2+C_0^pT.
	\end{align}
	Moreover, we have the blow-up rate estimate
	\begin{align*}
	&(pB^{p-1}+\delta)^{-1}(T-t)^{-1}\leq\|\partial_xu(\cdot,t)\|_{L^\infty}\\
	&\leq (A^{p-1}B^{1-p}-\delta)^{-1}(pA^{p-1}-\delta)^{-1}(T-t)^{-1}, 
	\end{align*}
	as \(t\rightarrow T^{-}\).
	
\end{theorem}

It is standard to show that the Cauchy problem for the equation \eqref{eq:gKdV} with the initial data \(u(0,x)=\phi(x)\) is well-posed in the class \(C([0,T): H^3(\R))\) for some \(T>0\) which will now  denote the maximal time of existence. Using the same notations \(X(t,x),v_0(t,x),v_1(t,x),m(t)\) and \(q(t)\) as \eqref{eq:particle}-\eqref{1}, it then 
follows from \eqref{eq:gKdV} that \eqref{3.1} and \eqref{3.2} hold, in which \(K_0(t,x)\) and \(K_1(t,x)\) are the same as in \eqref{3.3} and \eqref{3.4}.

Analogously, to prove Theorem \ref{th:gKdV}, we need to show \eqref{4}
via a contradiction argument by assuming \eqref{5} conversely. We need to carrry two main tasks to close the proof of \eqref{5}. 

The first task is to estimate the bound of \(q(t)\).  
Let \(t\in[0,T_1]\) and \(x\in [\bar{x}_1,\bar{x}_2]\), and a priori assume 
\begin{equation*}
\begin{aligned}
A\leq v_0(t,x)\leq B,  
\end{aligned}
\end{equation*}
in which \(A\) and \(B\) are given in Theorem \ref{th:gKdV} satisfying \eqref{con:A,B}. 
We now  define 
\begin{equation*}
\begin{aligned}
\Sigma_{\delta}(t)=\{x\in[\bar{x}_1,\bar{x}_2]:\ v_1(t,x)\leq (A^{p-1}B^{1-p}-\delta)m(t)\},
\end{aligned}
\end{equation*}
and 
\begin{equation*}
\begin{aligned}
v_1(t,x)=:m(0)r^{-1}(t,x).
\end{aligned}
\end{equation*}

With slight modifications, one can follow the proof of Lemma \ref{le:a1}-Lemma \ref{le:a2} to show the following: 
\begin{lemma}\label{le:v1} We have \(\Sigma_{\delta}(t_2)\subset \Sigma_{\delta}(t_1) \) whenever \(0\leq t_1\leq t_2\leq T_1\).
	
\end{lemma}

\begin{lemma}\label{le:v2} \(q(t)\) is decreasing and satisfies
	\begin{equation*}
	\begin{aligned}
	0<q(t)\leq 1,\quad \text{for}\ t\in [0,T_1].
	\end{aligned}
	\end{equation*}
	We also have the integral estimates 
	\begin{equation*}
	\begin{aligned}
	\int_0^tq^{-s}(\tau)\, d \tau
	&\leq (1-s)^{-1}m^{-1}(0)(pA^{p-1}-\delta)^{-1}
	(A^{p-1}B^{1-p}-\delta)^{-s}\\
	&\quad\times\big[q^{1-s}(t)-(A^{p-1}B^{1-p}-\delta)^{s-1}\big],\quad \text{for}\ t\in [0,T_1].
	\end{aligned}
	\end{equation*}	
	where \(s>0, s\neq 1\), and	
	\begin{equation*}
	\begin{aligned}
	\int_0^tq^{-1}(\tau)\, d \tau
	&\leq m^{-1}(0)(pA^{p-1}-\delta)^{-1}(A^{p-1}B^{1-p}-\delta)^{-1}\\
	&\quad\times\big[\log (A^{p-1}B^{1-p}-\delta)+\log q(t)\big],\quad \text{for}\ t\in [0,T_1].
	\end{aligned}
	\end{equation*}

\end{lemma}

The other task is to estimate the nonlocal terms.	
\begin{lemma}\label{le:gfKdV-nonlocal} For all \(t\in[0,T_2]\) and \(x\in\R\),  we have 
	\begin{equation*}
	\begin{aligned}
	|K_0(t,x)|\leq \big[4C_0(1+\alpha)^{-1}+2C_1(-\alpha)^{-1}\big]q(t)^{-(1+\alpha)},
	\end{aligned}
	\end{equation*}
	and
	\begin{equation*}
	\begin{aligned}
	|K_1(t,x)|
	&\leq \big[4C_1(1+\alpha)^{-1}+2C_s(-\alpha)^{-1}\\
	&\quad\times(2^{-1}A^{p-1}B^{1-p})^{-7B^{p-1}/(2pA^{2p-2})}\|\phi^{\prime\prime\prime}\|_{L^2}\big]q^{-2}(t),
	\end{aligned}
	\end{equation*}

\end{lemma}

\begin{proof} 
	Compared with the proof of Lemma \ref{le:gfKdV-nonlocal}, the only difference is to apply Lemma \ref{le:v2} to replace Lemma \ref{le:a2} to estimate \(\|\partial_x^3u\|_{L^2}\) by using \eqref{the third derivative}
	\begin{equation*}
	\begin{aligned}
	\|\partial_x^3u\|_{L^2}
	&\leq \|\phi^{\prime\prime\prime}\|_{L^2}(A^{p-1}B^{1-p}-\delta)^{-\frac{7}{2(pA^{p-1}-\delta)(A^{p-1}B^{1-p}-\delta)}}\\
	&\quad\times q(t)^{-\frac{7}{2(pA^{p-1}-\delta)(A^{p-1}B^{1-p}-\delta)}}\\
	&\leq (2^{-1}A^{p-1}B^{1-p})^{-7B^{p-1}/(2pA^{2p-2})}\|\phi^{\prime\prime\prime}\|_{L^2}
	q(t)^{-7B^{p-1}/(2pA^{2p-2})}.
	\end{aligned}
	\end{equation*}

\end{proof}

We finally note that one can similarly extend the result in Theorem \ref{th:gKdV} to the generalized Burgers-Hilbert and Whitham equations (the cubic nonlinearity \(u^2u_x\) being replaced by the generalized nonlinearity \(u^pu_x\) \((p>1)\) in \eqref{eq:mBH} and \eqref{eq:mWhitham}).

\section{The case $\alpha>0$}
As aforementioned, most assertions in this Section are conjectures that will be motivated by numerical simulations.

\subsection{Solitary waves}

We first state and recall some theoretical results, starting by the solitary wave solutions, that is solutions of the form
$$u_c(x,t)=Q_c(x-ct),$$
where $Q_c\in H^{\alpha/2}(\R).$
They should satisfy the equation
\begin{equation}\label{sol}
D^{\alpha}Q_c + cQ_c\mp\frac{1}{3}Q_c^3=0.
 \end{equation}

The first one is classical and concerns the non-existence of solitary waves.

\begin{proposition}\label{nonex}
There exists no nontrivial solitary waves of \eqref{eq:main}:

1. In the defocusing case for all $\alpha>0$; 

2. In the focusing case when $0<\alpha\leq 1/2$ (energy super critical).

\end{proposition}

\begin{proof}
The proof is similar to that of the quadratic case (see \cite{LPS3}).  By \eqref{sol}, we have the energy identity
 $$\int_\R |D^{\alpha/2}Q_c|^2\, dx+c\int_\R Q_c^2\, dx\mp\frac{1}{3}\int_\R Q_c^4\, dx=0,$$
and the Pohozaev identity
$$\frac{\alpha-1}{2}\int_\R |D^{\alpha/2}Q_c|^2\,  dx-\frac{c}{2}\int_\R Q_c^2\, dx\pm\frac{1}{12}\int_\R Q_c^4 \, dx=0,$$
which in turn is a consequence of the identity (see for instance Lemma 3 in \cite{KMR})
\begin{equation*}
\int_{\mathbb R}(D^{\alpha}\phi)x\phi^\prime\, dx=\frac{\alpha-1}2\int_{\mathbb R}|D^{\alpha/2}\phi|^2\, dx,
\end{equation*}
imply in the focusing case
\begin{equation*}
(1-2\alpha)\int_\R |D^{\alpha/2} Q_c|^2\, dx+c\int_\R Q_c^2\, dx=0
\end{equation*}
proving that no finite energy solitary waves exist in the energy supercritical case $\alpha>1/2$ when $c\leq 0$.

The same conclusion holds for any $\alpha>0$ in the defocusing case just by the energy identity.

\end{proof}

The existence of non trivial solutions for \eqref{sol} in the admissible range is standard and we refer for instance to \cite{FL} from which we extract the 

\begin{proposition}\label{FL-SW}
Let $1/2<\alpha<2$. 

(i) (Existence) There exists a solution $Q\in H^{\alpha/2}(\R)$ of \eqref{sol} such that $Q(|x|)>0$ is even, positive and strictly decreasing in $|x|.$ 

(ii) (Symmetry and monotonicity) If $Q\in H^{\alpha/2}(\R)$ with 
$Q\geq 0$ and $Q\not\equiv 0,$ then there exists $x_0\in \R$ such that $Q(\cdot-x_0)$ is even, positive and strictly decreasing in $|x-x_0|.$

(iii) (Regularity and decay) If $Q\in H^{\alpha/2}(\R)$ solves 
\eqref{sol}, then $Q\in H^{\alpha +1}(\R).$ Moreover we have the decay estimate
$$|Q(x)|+|xQ'(x)|\leq \frac{C}{1+|x|^{\alpha +1}},$$
for all $x\in \R$ and some constant $C>0.$
\end{proposition}

\begin{remark}
In fact $Q$ can be obtained as a minimizer of the Weinstein functional
$$J^\alpha(u)=\left(\int_\R |u|^4\, dx\right)^{-1}\left(\int_\R 
|D|^{\alpha/2}u|^2 \, dx\right)^{1/\alpha}\left(\int_\R u^2\, 
dx\right)^{2-\frac{1}{\alpha}},$$
or optimizers of the Gagliardo-Nirenberg inequality 
$$\int_\R |u|^4\, dx\leq C_\alpha\left(\int_\R |D|^{\alpha/2}u|^2\, dx\right)^{1/\alpha}\left(\int_\R u^2 \, dx\right)^{2-\frac{1}{\alpha}},.$$

\end{remark}

Actually, following \cite{FL}, for $\alpha>1/2,$ a ground state solution of
\begin{equation}\label{GS}
D^\alpha Q+Q-Q^3/3=0
\end{equation}
is a positive and even solution that minimizes the Weinstein functional on $H^{\alpha/2}(\R)\setminus \{0\}.$ 
%As noticed in \cite{LPS3}, Remark 2.3, the ground state solution is (conditionally) orbitally stable when $\alpha>1$ but it is expected to be orbitally unstable via a finite-time blow-up mechanism when $1/2<\alpha<1$, as confirmed by the numerical simulations below.

Concerning stability, one can in a straightforward way extend the 
results of \cite{LPS3} that concerned the (quadratic) fKdV equation to the focusing modified fKdV equation to prove that the ground state is (conditionally) stable in the $L^2$ subcritical case $\alpha>1$. We refer to \cite{pava} for a complete analysis  of stability issues for the focusing modified fKdV equation (and of fKdV with higher nonlinearities) in the $L^2$ subcritical regime.

%\textcolor{red}{ground states, to be continued}

Note that 
\begin{equation*}
	Q_{c}(x)= \sqrt{c}Q(c^{1/\alpha}x)
	\label{Qc},
\end{equation*}
where we have put $Q_{1} = Q$,
so that 
$$\|Q_c\|_{L^2}^2=c^{\frac{\alpha-1}{\alpha}}\|Q\|_{L^2}^2,$$
proving that  the solitary waves may have arbitrary small $L^2$ norm by taking arbitrary large velocities when $1/2<\alpha<1$ (resp. arbitrary small velocities when $\alpha>1$). This excludes scattering of small solutions in the $L^2$ norm.

The case $\alpha =1$ (modified Benjamin-Ono) is $L^2$ critical and the solitary waves have constant $L^2$ norm.

Similarly, 
$$\|D^{\alpha/2}Q_c\|_{L^2}^2=c^{\frac{3\alpha-2}{2\alpha}}\|D^{\alpha/2}Q\|^2_{L^2}$$
proving for instance that no scattering in the energy norm is possible when $1/2<\alpha<1$.

 We show some ground states of the equation (\ref{eq:main}) with the 
$+$ sign for $c=1$ and various values of $\alpha$  in 
Fig.~\ref{mfKdVsol}. The ground states get more and more peaked the 
smaller the dispersion is. Compared to the ground states for the fKdV 
equation in \cite{KS} their maximum is smaller here because of the 
higher nonlinearity. 
\begin{figure}[htb!]
  \includegraphics[width=0.7\textwidth]{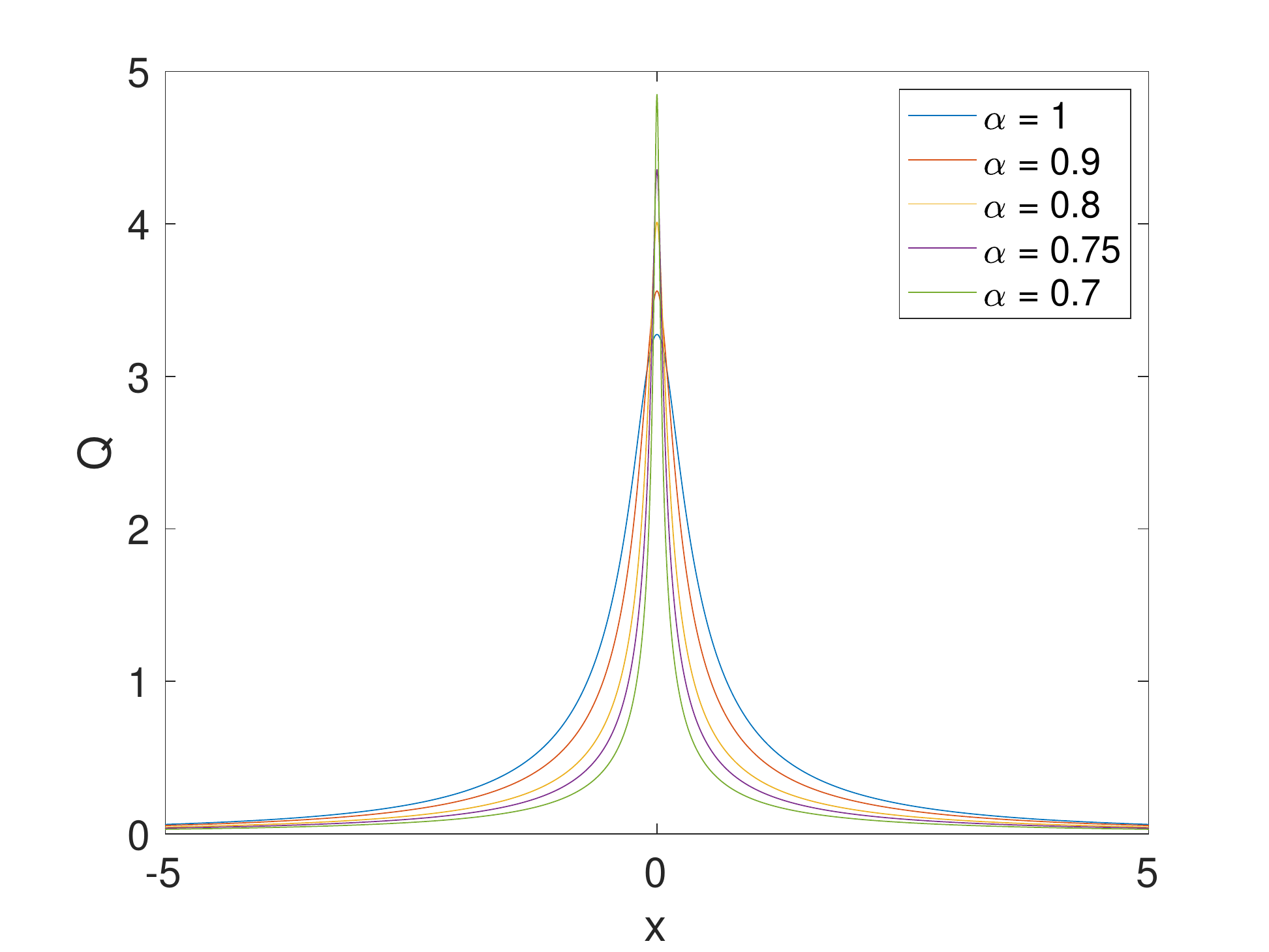}
  \caption{Ground states of the equation (\ref{eq:main}) with the $+$ 
  sign for $c=1$ and various values of $\alpha$.}
 \label{mfKdVsol}
\end{figure}

\subsection{The Cauchy problem: generalities}

It is straightforward to prove that the modified fKdV equation is locally well posed in $H^s(\R), s>3/2$ when $\alpha\geq-1.$ On the other hand the modified Burgers equation is ill-posed in $H^{3/2}(\R)$ as can be checked by extending the proof in \cite{LPS2} for the usual Burgers equation.

Adding  a dispersive perturbation  ($\alpha>0$) has the effect  to 
enlarge the space of resolution for the local Cauchy problem. Actually it was proven in \cite{LPS3} that the Cauchy problem for the fKdV equation is locally well-posed in $H^s(\R), s>\frac{3}{2}-\frac{3\alpha}{8}.$ This was improved in \cite{MPV} to $s>\frac{3}{2}-\frac{5\alpha}{4}$ leading to the global well-posedness of the Cauchy problem in the energy space when $\alpha>6/7.$ The expected global well-posedness for the whole range $1/2<\alpha<1$ is still open.

A similar enlargement of the space of the resolution for the local Cauchy problem is expected for the modified fKdV equation but this is outside  the scope of the present paper that is focused on global issues. We just recall that the Cauchy problem for the modified Benjamin-Ono ($\alpha=1$), in both the focusing and defocusing case,  was proven to be locally well-posed in $H^s(\R), s\geq 1/2$ by Kenig and Takaoka \cite{KT} extending a previous work of Molinet and Ribaud \cite{MR} who proved the result for $s>1/2.$

\subsection{The Cauchy problem: focusing case} 
For the focusing modified Benjamin-Ono equation ($\alpha=1$), Martel and Pilod \cite{MP} proved the finite time blow-up by constructing a minimal mass blow-up solution.

For the other values of $\alpha$ we a priori  formulate the following conjectures which will be reinforced by the numerical simulations below:

(i) Finite time blow-up in the case $1/2<\alpha<1$ similar to the 
gKdV equation

$$u_t+u^pu_x+u_{xxx}=0,$$

when $p>4.$

(ii) Finite time blow-up in the energy supercritical case $0<\alpha<1/2.$
% Possible rigorous proof in the case $\alpha=1/2$ (energy critical).

(iii) Global well-posedness and soliton resolution when $\alpha>1.$

\begin{remark}
It is well known that the focusing modified KdV equation possesses special solutions that are periodic in t and localized in x, the so-called breather solution (see for instance \cite{AlMu} and the references therein). Establishing the existence of such solutions for the focusing modified mKdV when $\alpha>1/2$ is an interesting open question.
\end{remark}

\subsection{The Cauchy problem: defocusing case}

In the energy subcritical case $\alpha >1/2$, one obtains classically by a compactness method the global existence of weak solutions:

\begin{proposition}\label{weak}
Let $\alpha>1/2$ and $u_0\in H^{\alpha/2}(\R).$ Then there exists a solution $u\in L^\infty(\R: H^{\alpha/2}(\R))$ of \eqref{eq:main} with initial data $u_0.$
  \end{proposition}
  
 % \begin{remark} When $0<\alpha<\frac{1}{2}$ the global existence of weak solutions holds in the space $H^{\alpha/2}(\R)\cap L^4(\R).$
 % \end{remark}
  
  We recall that for the modified defocusing Benjamin-Ono equation ($\alpha =1$) Kenig and Takaoka \cite{KT} proved the global well-posedness in $H^s(\R), s\geq 1/2$.

One might a priori conjecture that  global well-posedness holds for $\alpha>1/2$ but the case  $0<\alpha\leq 1/2$ is unclear.

We now present numerical simulations that support the previous conjectures.

Equation (\ref{eq:main}) is invariant under the rescaling 
\begin{equation}\label{resc}
	y = \frac{x}{L},\quad \frac{d\tau}{dt} = 
	\frac{1}{L^{1+\alpha}},\quad U = L^{\alpha/2}u,
\end{equation}
i.e., if $u(x,t)$ is a solution to equation (\ref{eq:main}), so is 
$U(y,\tau)$ for constant $L$. If in (\ref{resc}) one lets $L$ depend 
on $\tau$, one speaks of a \emph{dynamic rescaling}, and 
(\ref{eq:main}) takes the form
\begin{equation}\label{eqresc}
	U_{\tau}-\partial_{\tau} (\ln L)(yU_{y}+\alpha U/2) \pm 
	U^{2}U_{y}-|D_{y}|^{\alpha}U_{y} = 0.
\end{equation}
The mass $m[u]$ (the square of the $L^{2}$ norm) and the energy 
$H[u]$ have the following scaling behavior under 
(\ref{resc}), 
\begin{equation}\label{mH}
	m[U] = L^{1-\alpha}m[u], \quad H[U] = L^{1-2\alpha}H[u].
\end{equation}
This means, as was previously noticed, the equation (\ref{eq:main}) is $L^{2}$ 
critical for 
$\alpha=1$, the modified Benjamin-Ono equation, and energy critical 
for $\alpha=1/2$. Thus a self-similar blow-up is expected in the 
focusing case for $0<\alpha\leq 1$, in the rescaling (\ref{resc}) 
such that $\tau\to\infty$ and $L(\tau)\to 0$ in this case. In this 
limit, one gets the blow-up profile  $U^{\infty}$ from 
(\ref{eqresc}), 
\begin{equation}\label{profile}
	a^{\infty}(yU_{y}^{\infty}+\alpha U^{\infty}/2) + 
	(U^{\infty})^{2}U_{y}^{\infty}-|D_{y}|^{\alpha}U_{y}^{\infty} = 0,
\end{equation}
where $a = -(\ln L)_{\tau}$ and $\lim_{\tau\to\infty}a =a^{\infty}$. 
In the $L^{2}$ critical case one expects $a^{\infty}=0$ in which case 
the blow-up profile will be just the solitary wave of the modified 
Benjamin-Ono equation. In the general case, the blow-up profile 
depends on $a^{\infty}$, and it is not known whether there are 
localized solutions to (\ref{profile}).

In \cite{KP}, we have studied 
blow-up in solutions to generalized KdV equations, and it was shown 
that it is numerically problematic to solve the rescaled equation 
(\ref{eqresc}). Instead we solved the generalized KdV equation 
numerically, and fitted certain norms as the $\|u\|_{L^\infty}$ and 
$\|u_{x}\|_{L^2}$ to certain blow-up models. In the $L^{2}$ critical 
case one expects $L\propto 1/\tau$, and thus 
\begin{equation}\label{critscal}
	\|u\|_{L^\infty}\propto (t^{*}-t)^{-1/2},\quad 
	\|u_{x}\|_{L^2}^{2} \propto (t^{*}-t)^{-(1+1/\alpha)}.
\end{equation} 
In the $L^{2}$ supercritical one expects $L\propto \exp(-\kappa 
\tau)$ with $\kappa>0$. In this case one has 
\begin{equation}\label{supcritscal}
	\|u\|_{L^\infty}\propto (t^{*}-t)^{-\alpha/(1+\alpha)/2},\quad 
	\|u_{x}\|_{L^2}^{2} \propto (t^{*}-t)^{-1}.
\end{equation}
In the case of a blow-up, these norms will be fitted to the above 
models to validate the conjectured blow-up scenario. 

\subsection{Numerical study of the focusing case}

We first consider the focusing $L^{2}$ subcritical case for 
$\alpha=1.5$ with $N=2^{12}$ Fourier modes for $x\in10[-\pi,\pi]$ and 
$N_{t}=10^{4}$ time steps for $t\in[0,1]$.  For small 
enough initial data, there is just dispersion of the data as can be 
seen for the initial data $u(x,0)=\exp(-x^{2})$ on the left of 
Fig.~\ref{mfKdV15gauss} on the left. Dispersive radiation is emitted 
to the left, no stable structure seems to appear. This is also 
confirmed by the $L^{\infty}$ norm of the solution in the same figure 
on the right. It appears to be monotonically decreasing. 
\begin{figure}[htb!]
  \includegraphics[width=0.49\textwidth]{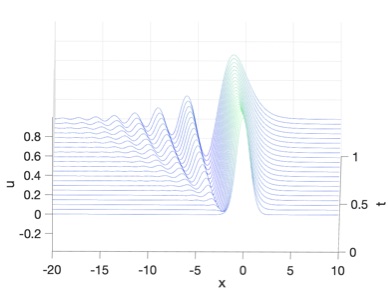}
  \includegraphics[width=0.49\textwidth]{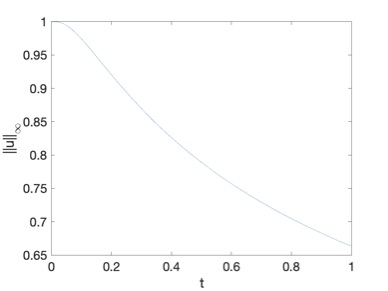}
 \caption{Solution to equation (\ref{eq:main}) with the $+$ sign for 
 $\alpha=1.5$ and 
 the initial data $u(x,0)=\exp(-x^{2})$  on the 
 left, and the $L^{\infty}$ norm of the solution on the right.}
 \label{mfKdV15gauss}
\end{figure}

We are here in a scattering scenario which suggests that the initial data 
$u(x,0)=\exp (-x^2)$ leads to a solution without solitary waves.

The situation changes if initial data of larger norm are considered. 
In Fig.~\ref{mfKdV15_5gauss} we show the solution in the same setting 
as in Fig.~\ref{mfKdV15gauss}, but this time for the initial data 
$u(x,0)=5\exp(-x^{2})$. In addition to the dispersive radiation 
emitted towards $-\infty$, there is now a solitary wave traveling to 
the right. 
\begin{figure}[htb!]
  \includegraphics[width=0.7\textwidth]{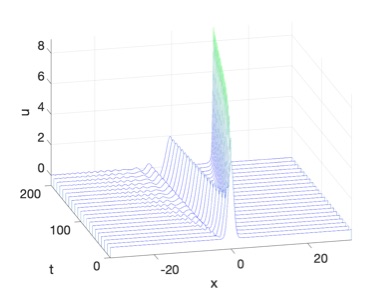}
 \caption{Solution to equation (\ref{eq:main}) with the $+$ sign for 
 $\alpha=1.5$ and 
 the initial data $u(x,0)=5\exp(-x^{2})$.}
 \label{mfKdV15_5gauss}
\end{figure}

This is confirmed by the $L^{\infty}$ norm of the solution on the 
left of Fig.~\ref{mfKdV15_5gaussinf}. To make this even more 
explicit, we show on the right of the same figure the solution of 
Fig.~\ref{mfKdV15_5gauss} at the final time $t=1$ together with the 
solitary wave fitted according to (\ref{Qc}) in green. It can be 
seen that the solitary wave is already fully developed, possibly the 
smaller hump to the left of the large solitary wave will become 
a ground state of smaller velocity. This suggests that the soliton 
resolution conjecture also applies to this equation. 
\begin{figure}[htb!]
  \includegraphics[width=0.49\textwidth]{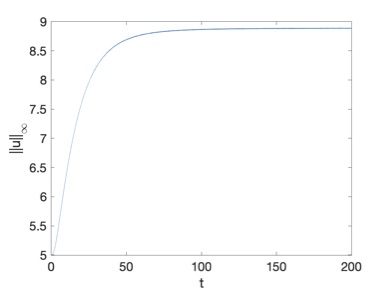}
  \includegraphics[width=0.49\textwidth]{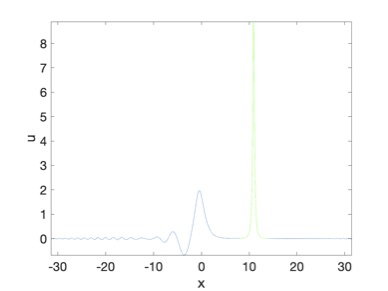}
 \caption{$L^{\infty}$ norm of the solution to equation (\ref{eq:main}) with the $+$ sign for 
 $\alpha=1.5$ and 
 the initial data $u(x,0)=5\exp(-x^{2})$  on the 
 left, and the solution for $t=1$ with a fitted solitary wave  on the 
 right in green.}
 \label{mfKdV15_5gaussinf}
\end{figure}

Next we consider the $L^{2}$ critical case $\alpha=1$, i.e., the 
modified Benjamin-Ono equation. For the initial data 
$u(x,0)=5\exp(-x^{2})$, we use $N=2^{14}$ Fourier modes for 
$x\in5[-\pi,\pi]$ and $N_{t}=10^{4}$ time steps for $t\in[0,0.081]$. 
The code breaks shortly after the final time since relative energy conservation is 
no longer assured to the order of $10^{-3}$. The solution for various 
values of $t$ can be seen in Fig.~\ref{mfKdVa1_5gauss}. A 
solitary wave detaches from the initial hump and
eventually appears to blow up.  
\begin{figure}[htb!]
  \includegraphics[width=0.49\textwidth]{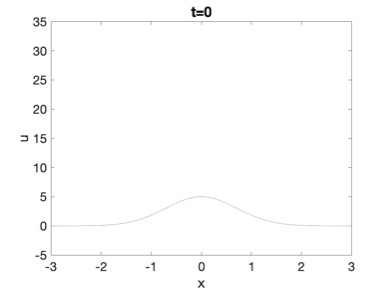}
  \includegraphics[width=0.49\textwidth]{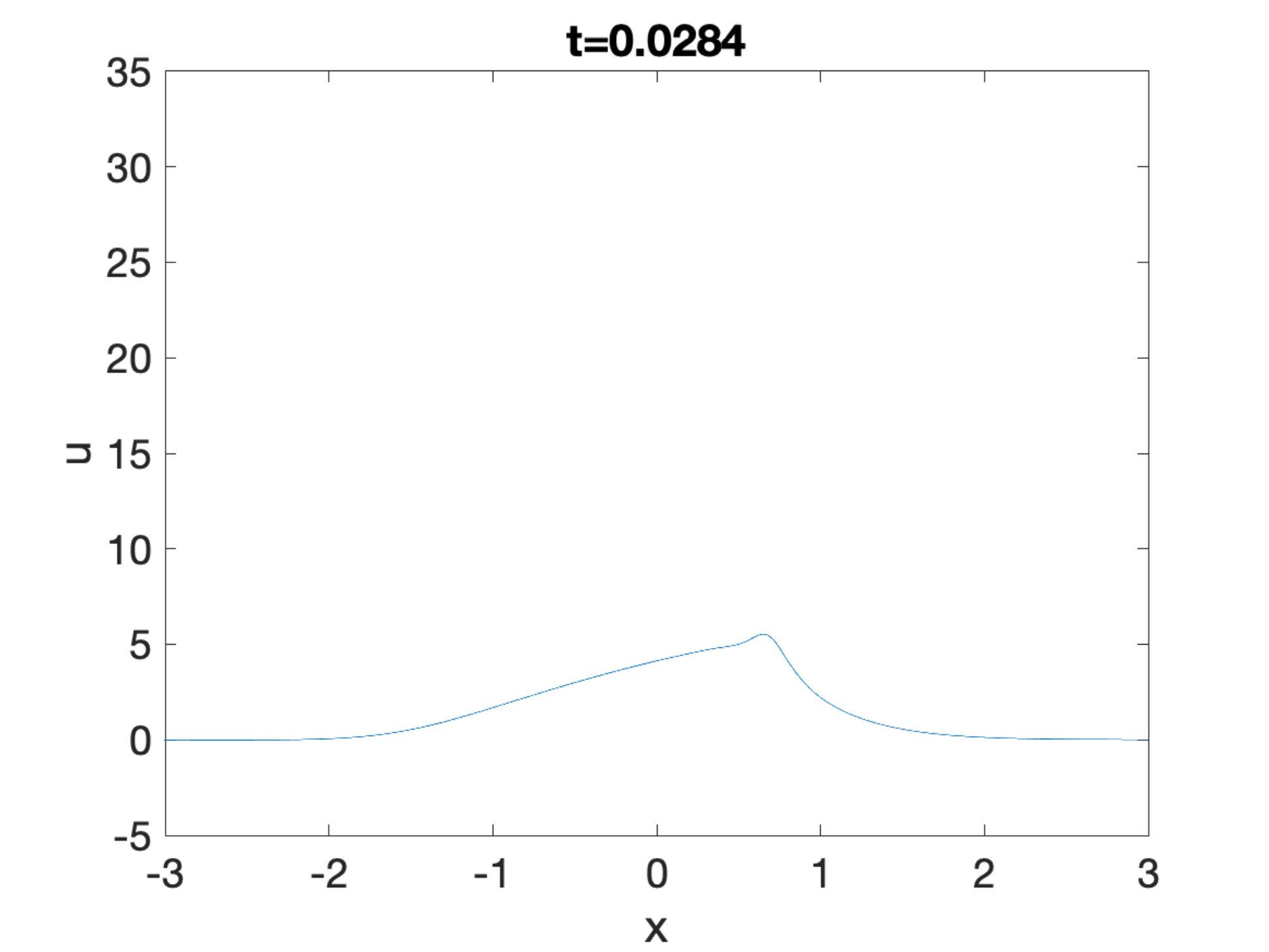}\\
   \includegraphics[width=0.49\textwidth]{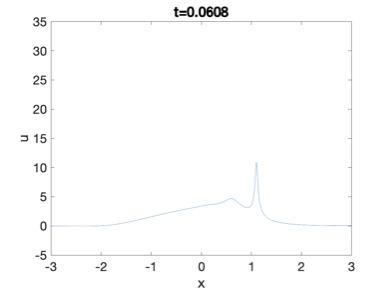}
  \includegraphics[width=0.49\textwidth]{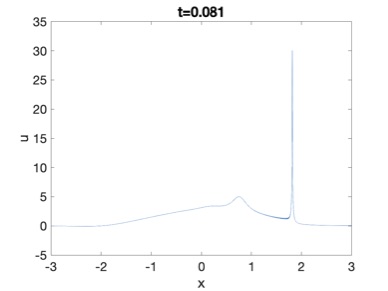}
\caption{Solution to equation (\ref{eq:main}) with the $+$ sign for 
 $\alpha=1$  for 
 the initial data $u(x,0)=5\exp(-x^{2})$ at various times.}
 \label{mfKdVa1_5gauss}
\end{figure}

An $L^{\infty}$ blow-up is also confirmed by the $L^{\infty}$ norm 
and the $L^{2}$ norm of the gradient as shown in 
Fig.~\ref{mfKdVa1_5gaussinf}.
\begin{figure}[htb!]
  \includegraphics[width=0.49\textwidth]{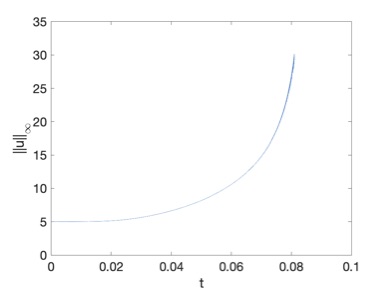}
  \includegraphics[width=0.49\textwidth]{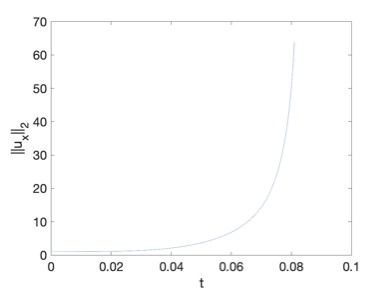}
\caption{$L^{\infty}$ norm on the left and $\|u_{x}\|_{L^2}$ 
(normalized to 1 for $t=0$) on the 
right for the solution to equation (\ref{eq:main}) with the $+$ sign for 
 $\alpha=1$  for 
 the initial data $u(x,0)=5\exp(-x^{2})$.}
 \label{mfKdVa1_5gaussinf}
\end{figure}

If we fit the norms in Fig.~\ref{mfKdVa1_5gaussinf} for the last 500 
time steps according to 
\begin{equation}
	\ln \|u\|_{L^\infty}\sim a \ln (t^{*}-t)+b
	\label{fit}
\end{equation}
in \cite{KS}, then we get a minimum residual for the values $a=-0.54$,     
$b=0.49$, and     $t^{*}=0.0854$. Similarly we get for $\|u_{x}\|_{L^2}$ 
the fitted values $a=-1.04$,    $b=-1.62$ and     $t^{*}=0.085$. The 
good agreement of the blow-up times $t^{*}$ is an indicator of the 
quality of these results which do not change much if a slightly 
smaller or larger number of points are taken into account for the 
fitting. The results for the scaling exponents are in accordance with 
the expectation for the critical case in (\ref{critscal}). In 
Fig.~\ref{mfKdVa1_5gaussmatch} we show the solution at the last 
recorded time together with a rescaled 
soliton (according to (\ref{resc}), 
the ratio of the maximum of $u$ and $Q$ respectively determines the factor 
$L^{\alpha/2}$). It can be seen that the blow-up profile is  given by 
the ground state though the blow-up has not yet been reached. Note 
that we obtain the same behavior for the $-$ sign in equation 
(\ref{fBBM}). 
\begin{figure}[htb!]
  \includegraphics[width=0.7\textwidth]{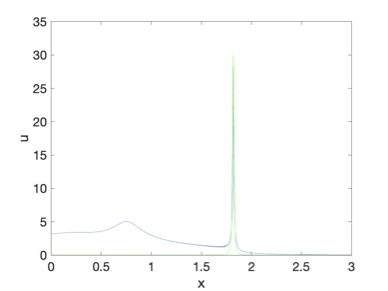}
\caption{Solution to equation (\ref{eq:main}) with the $+$ sign for 
 $\alpha=1$  for 
 the initial data $u(x,0)=5\exp(-x^{2})$ for $t=0.081$ and a rescaled 
 soliton according to (\ref{resc}) in green.}
 \label{mfKdVa1_5gaussmatch}
\end{figure}

As an example for the $L^{2}$ supercritical (but energy subcritical) 
case, we consider $\alpha=0.8$, the initial data 
$u(x,0)=3\exp(-x^{2})$ and $t\in[0,0.22]$. We use $N=2^{14}$ Fourier 
modes for $x\in2[-\pi,\pi]$ and $N_{t}=10^{4}$ time steps. 
The code is stopped for 
$t=0.2123$ since the relative energy conservation drops below 
$10^{-3}$ shortly afterwards. The solution at this time can be seen 
in Fig.~\ref{mfKdVa08_3gauss}. The profile of the blow-up is clearly 
different from the $L^{2}$ critical one in 
Fig.~\ref{mfKdVa1_5gaussmatch} .
\begin{figure}[htb!]
  \includegraphics[width=0.7\textwidth]{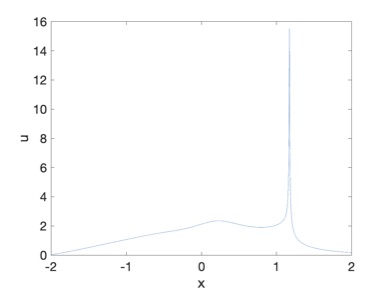}
\caption{Solution to equation (\ref{eq:main}) with the $+$ sign for 
 $\alpha=0.8$  for 
 the initial data $u(x,0)=3\exp(-x^{2})$ for $t=0.2123$.}
 \label{mfKdVa08_3gauss}
\end{figure}

If we fit the $L^{\infty}$ norm of the solution for the last 100 
recorded time steps (the 
results are similar for a slightly larger number of steps) according to
(\ref{fit}), we find $a = -0.27$, $b=1.01$, and $t^{*}=0.2140$. An 
analogous fitting for $\|u_{x}\|_{L^2}$ yields $a = -0.62$,    $b=0.14$, 
and $t^{*}=    0.2140$. Note the excellent agreement for the blow-up 
time. From (\ref{supcritscal}) one would expect in the former case 
for $a$ a value of $-0.222$, and in the latter $-0.5$. This means that 
the agreement with expectation is better for the $L^{\infty}$ norm, 
and that we are not close enough to the blow-up (on the rescaled time 
scales since the blow-up is expected to be exponential in $\tau$) for 
$\|u_{x}\|_{L^2}$. 

In the energy critical case $\alpha=0.5$ and the energy supercritical 
case $\alpha=0.4$, we consider the initial data 
$u(x,0)=2\exp(-x^{2})$ with the same numerical parameters. The codes 
break for $t=0.3281$ and $t=0.2674$ respectively in the sense that 
the relative energy conservation drops below $10^{-3}$. 
\begin{figure}[htb!]
  \includegraphics[width=0.49\textwidth]{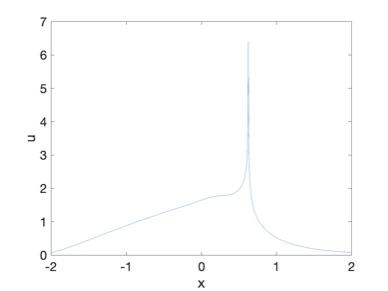}
  \includegraphics[width=0.49\textwidth]{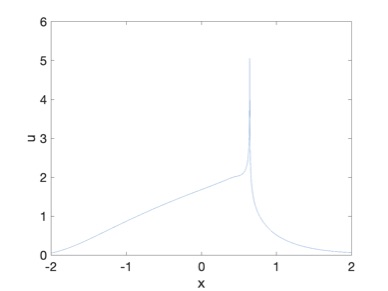}
\caption{Solutions to equation (\ref{eq:main}) with the $+$ sign for 
 the initial data $u(x,0)=2\exp(-x^{2})$ at the times where the code 
 breaks, on the left for $\alpha=0.5$, on the right for $\alpha=0.4$.}
 \label{mfKdVa05_2gauss}
\end{figure}

Fitting the norms for the last 100 recorded time steps to (\ref{fit}), 
we get in the 
case $\alpha=0.5$ for the  $L^{\infty}$ norm $a=-0.168$ (expected 
value $-1/6$) and for the $\|u_{x}\|_{L^2},$ $a=-0.47$ (expected value 
$-1/2$) with blow-up times of $0.3298$ and $0.3296$ respectively. 
Thus an excellent agreement between the blow-up times and the 
expectation for the scalings is observed. For $\alpha=0.4$, we get 
for the $L^{\infty}$ norm $a=-0.14$ (expected value $-0.1429$), 
and $a=-0.66$ (expected value $-0.5$) for $\|u_{x}\|_{L^2}$. In both cases the blow-up time is determined to be $0.2685$.

The above results can be summarized in the following 
\begin{conjecture}
    Consider  smooth initial data $u_{0}\in L^{2}(\mathbb{R})$ 
with a single hump for the focusing equation (\ref{eq:main}). Then for
    \begin{itemize}
\item $\alpha>1$: solutions to the focusing modified fKdV equations with the initial 
data $u_{0}$ stay smooth for 
all $t$. For large $t$ they decompose asymptotically into solitary 
waves and 
radiation.

\item $0<\alpha\leq 1$: solutions to the focusing modified fKdV equations with initial 
data $u_{0}$ of sufficiently small, but non-zero mass stay smooth for 
all $t$.

\item $\alpha=1$: solutions to the focusing modified fKdV equations  with the initial 
data $u_{0}$ with negative energy and mass larger than the solitary 
wave 
mass
blow up at finite time $t^{*}$. The type 
of the blow-up for $t\nearrow t^*$ is characterized by 
\begin{equation*}\label{uL2}
    u(x,t)\sim \frac{1}{\sqrt{L(t)}}Q\left(\frac{x-x_{m}}{L(t)}\right), \quad
    L = c_{0}(t^{*}-t),
\end{equation*}
where $c_{0}$ is a constant, and where $Q$ is the solitary wave 
solution (\ref{Qc}) for $c=1$. 

\item $0<\alpha<1$: solutions to the focusing modified fKdV equations  with the initial 
data $u_{0}$ and sufficiently large $L^{2}$ norm
blow up at finite time $t^{*}$ and finite $x=x^{*}$. 
The nature of blow-up
is self similar as given by (\ref{resc}), but the blow-up profile 
appears to be  
given by (\ref{profile}) with non-vanishing $a^{\infty}$.

 \end{itemize}
\end{conjecture}

\begin{remark}
	The numerical study of the appearance of singularities in 
	solutions to PDEs pushes all numerical methods to the limits, and 
	the results always have to be taken with a grain of salt. The 
	results we present in this paper are always stable with respect 
	to moderate changes of the numerical resolution. In addition we 
	fit in the case of shocks the Fourier coefficients and in the 
	case of $L^{\infty}$ blow-ups various norms. Nonetheless all 
	numerical results are with finite precision, and the analytical 
	phenomena to be studied can always be inaccessible within this 
	precision.
\end{remark}

\subsection{Numerical study of the defocusing case}

The same initial data as in Fig.~\ref{mfKdV15_5gauss} for the defocusing equation (\ref{eq:main}) 
lead only to dispersion as can be seen in Fig.~\ref{mfKdVm15_5gauss}. 
This is once more indication for the absence of stable structures in the 
defocusing case. 
\begin{figure}[htb!]
  \includegraphics[width=0.7\textwidth]{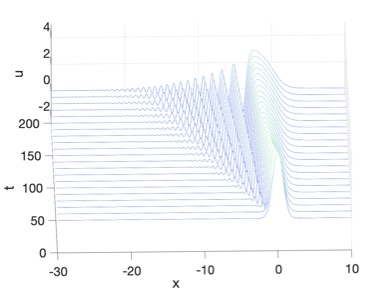}
 \caption{Solution to equation (\ref{eq:main}) with the $-$ sign for 
 $\alpha=1.5$ and 
 the initial data $u(x,0)=5\exp(-x^{2})$.}
 \label{mfKdVm15_5gauss}
\end{figure}

In the energy critical case $\alpha=0.5$, solutions to the defocusing 
equation (\ref{eq:main}) still appear to be dispersed. We use 
$N=2^{14}$ Fourier modes for $x\in 5[-\pi,\pi]$ and $N_{t}=10^{4}$ 
time steps for $t\in[0,0.05]$. As can be seen on the left of 
Fig.~\ref{mfKdVm05_5gauss}, there appears to be a strong gradient to 
the right of the initial hump, but the dispersion is still sufficient 
to create a DSW as is even more clear from the close-up of this zone 
on the right of the same figure. 
\begin{figure}[htb!]
  \includegraphics[width=0.49\textwidth]{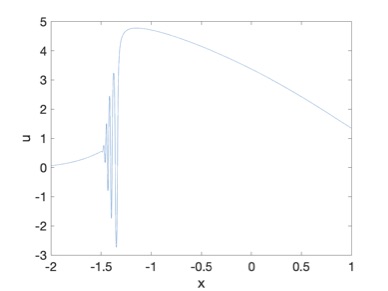}
  \includegraphics[width=0.49\textwidth]{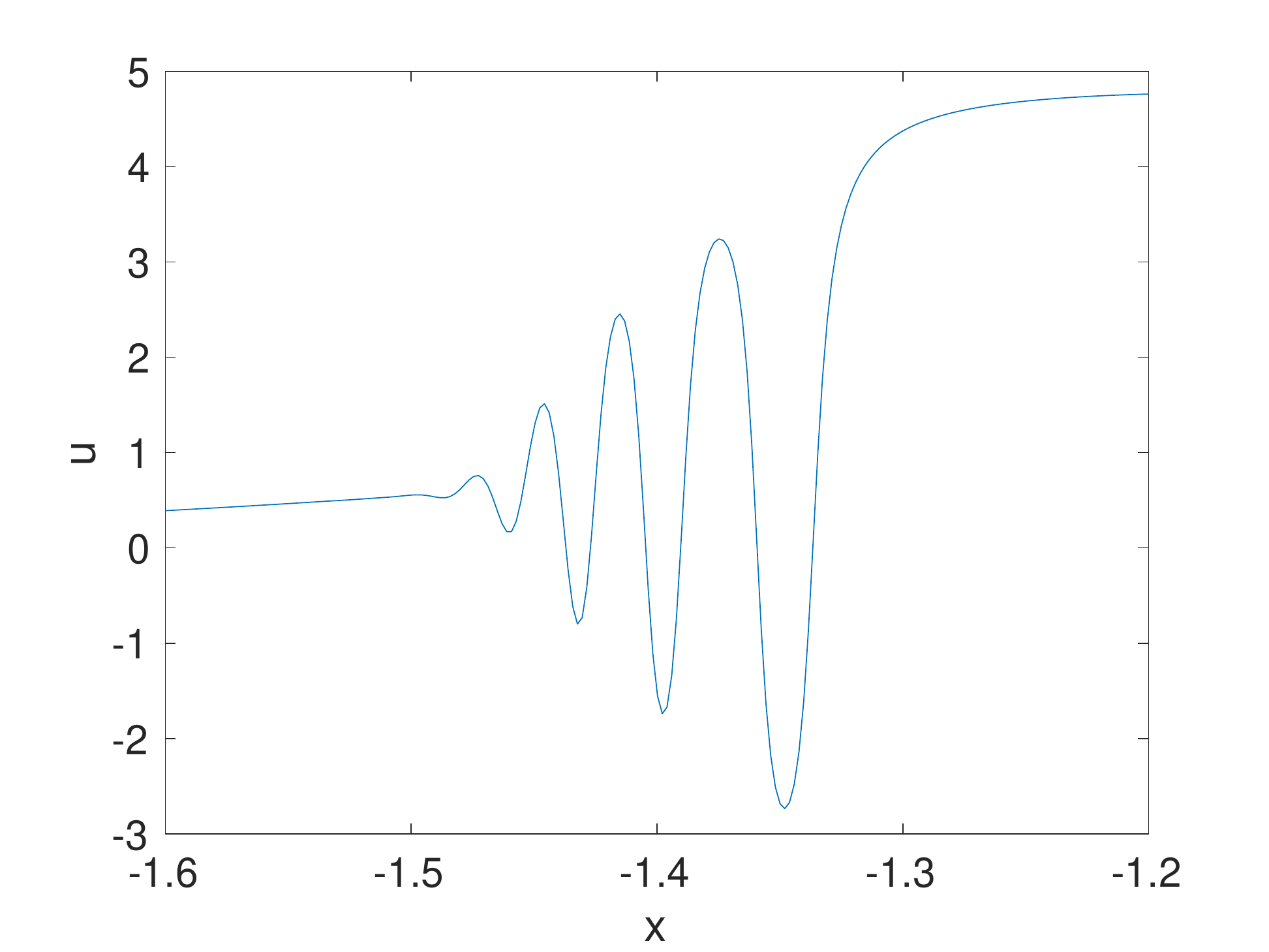}
 \caption{Solution to equation (\ref{eq:main}) with the $-$ sign for 
 $\alpha=0.5$  for 
 the initial data $u(x,0)=5\exp(-x^{2})$ at $t=0.05$ on the 
 left, and close-up of the oscillatory zone on the right.}
 \label{mfKdVm05_5gauss}
\end{figure}

In the energy supercritical case, for instance $\alpha=0.2$, the 
situation appears, however, to be similar to what has been seen for 
negative $\alpha$. If we consider the same initial data as in 
Fig.~\ref{mfKdVm05_5gauss} in this case, the code breaks for 
$t=0.332$, and the fitting of the Fourier coefficients on the right 
of Fig.~\ref{mfKdVm02_5gauss} indicates a 
$\mu\sim 0.33$, i.e., the formation of a cusp as can be seen in 
Fig.~\ref{mfKdVm02_5gauss}. Note that this phenomenon does not change 
if the code is rerun with the same numerical resolution for 
$x\in2[-\pi,\pi]$, i.e., with more than twice the resolution. But the 
situation appears to be similar to what has been observed in 
\cite{KSM} in the case of the defocusing fractional nonlinear 
Schr\"odinger equation. There an almost singular behavior was 
observed, which seemed to disappear when higher resolution had been 
used. 
\begin{figure}[htb!]
  \includegraphics[width=0.49\textwidth]{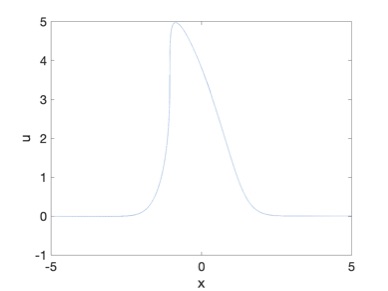}
  \includegraphics[width=0.49\textwidth]{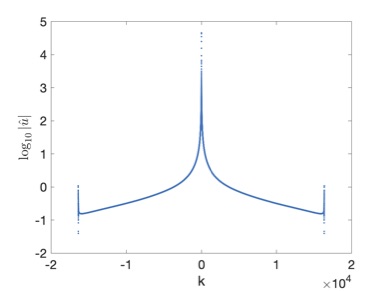}
 \caption{Solution to equation (\ref{eq:main}) with the $-$ sign for 
 $\alpha=0.2$  for 
 the initial data $u(x,0)=5\exp(-x^{2})$ at $t=0.0332$ on the left, 
 and the corresponding Fourier coefficients on the right.}
 \label{mfKdVm02_5gauss}
\end{figure}

The above behavior is confirmed by various norms of the solution. In 
Fig.~\ref{mfKdVm02_5gaussinf} it can be seen on the left that the 
$L^{\infty}$ norm of the solution is essentially constant and even 
decreases slightly. But the $L^{\infty}$ norm of the gradient on the 
right of the same figure appears to explode. This means both norms 
reflect the behavior of a shock of the solution. 
\begin{figure}[htb!]
  \includegraphics[width=0.49\textwidth]{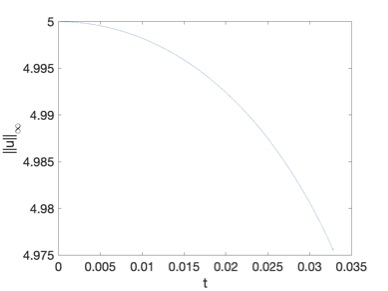}
  \includegraphics[width=0.49\textwidth]{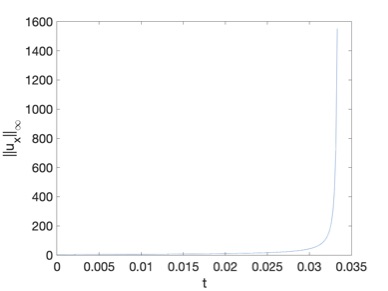}
 \caption{$L^{\infty}$ norm of the solution to equation (\ref{eq:main}) with the $-$ sign for 
 $\alpha=0.2$  for 
 the initial data $u(x,0)=5\exp(-x^{2})$ in dependence of time on the left, 
 and the $L^{\infty}$ norm of the gradient  on the right.}
 \label{mfKdVm02_5gaussinf}
\end{figure}

If we consider the same initial data for the defocusing equation 
(\ref{eq:main}) with $\alpha=0.1$, we get very similar results. The 
code is stopped for $t=0.033$ since the Fourier coefficients appear 
to show a completely algebraic decay for large $|k|$, and a fitting 
of the coefficients gives $\mu\sim 0.35$. The solution at this time 
can be seen in Fig.~\ref{mfKdVm01_5gauss} on the left. The $L^{\infty}$ 
norm of the gradient in dependence of time on the right of the same 
figure again indicates a shock. 
\begin{figure}[htb!]
  \includegraphics[width=0.49\textwidth]{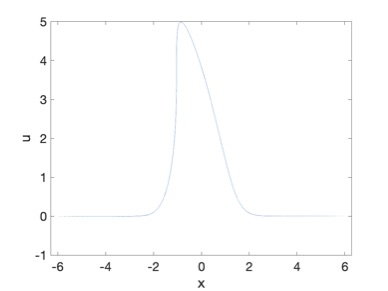}
  \includegraphics[width=0.49\textwidth]{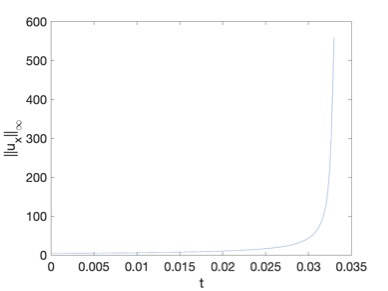}
 \caption{Solution to equation (\ref{eq:main}) with the $-$ sign for 
 $\alpha=0.1$  for 
 the initial data $u(x,0)=5\exp(-x^{2})$ for $t=0.033$ on the left, 
 and the $L^{\infty}$ norm of the gradient  on the right.}
 \label{mfKdVm01_5gauss}
\end{figure}

The above results can be summarized in the following 
\begin{conjecture}
	Consider smooth initial data $u_0\in L^{2}(\mathbb{R})$ with a 
	single hump for the defocusing equation (\ref{eq:main}). Then
	\begin{itemize}
		\item  Initial data with sufficiently small mass will be 
		dispersed, the solution is global in time.
	
		\item  For $\alpha\geq 1/2$, initial data of arbitrary mass lead 
		to solutions which will be dispersed and are 
		global in time. 
	
		\item  For $\alpha<1/2$, initial data of sufficiently large 
		mass will lead to the formation of a shock in finite time, a 
		singularity of the form $|x-x^{*}|^{1/3}$ for $|x|\sim 
		x^{*}$, where $^{*}=const$. 
	\end{itemize}
\end{conjecture}

\begin{remark}
The global existence and scattering for small solutions in the case $0<\alpha<1$ has been recently proven (\cite{SW3}).
\end{remark}

% We postpone  the study of the global existence and modified scattering to  \eqref{eq:main} in the case of $0<\alpha<1$ in the spirit of Theorem \ref{th:previous} to a forthcoming paper, but provide below some simulations on various cases.
 
%\textcolor{red}{Christian : simulations for $0<\alpha<1/3$ and $1/3<\alpha<1$? Is there a difference between the focusing and defocusing case?}.

 \section{The BBM version}
 We comment here briefly on the BBM version of the modified fKdV equation, that is 
 
 \begin{equation}\label{fBBM}
 u_t+u_x+|D|^\alpha u_t\pm u^2u_x=0,
 \end{equation}
 which makes sense for $\alpha>0$, and we will restrict to $0<\alpha\leq 1.$
 
 For any $\alpha>0$ the energy
 $$E(t)=\frac{1}{2}\int_\R (u^2+|D^{\alpha/2}u|^2)\, dx$$
 is formally conserved.  By  a standard compactness method this 
 implies that the Cauchy problem for \eqref{fBBM} admits a global weak 
 solution in $L^{\infty}(\R: H^{\alpha/2}(\R))$ for any initial data in $H^{\alpha/2}(\R)$ when $\alpha>1/3$ (this condition ensures the compactness of the embedding $H^{\alpha/2}(\R)\subset L^3_{\text{loc}}(\R)$).

 One can also use the equivalent form
 
 \begin{equation*}
 \partial_t u+\partial_x(I+D^\alpha)^{-1}\left(u\pm \frac{u^3}{3}\right)=0,
 \end{equation*}
 which gives the Hamiltonian formulation
 $$u_t+J_\alpha\nabla_uH(u)=0,$$
 where the skew-adjoint operator $J_\alpha$ is given by $J_\alpha=\partial_x(I+D^\alpha)^{-1}$
 and 
  $$H(u)=\frac{1}{2}\int_\R\bigg(u^2\pm\frac{1}{6}u^4\bigg)\, dx.$$
  
  Note that the Hamiltonian $H(u)$ makes sense for $u\in H^{\alpha/2}(\R)$ (due to the energy \(E(t)\) and Sobolev embedding) if and only if $\alpha\geq 1/2$.
  
  As noticed in \cite{LPS2} for the quadratic fBBM equation, when $\alpha\geq 1$ \eqref{fBBM} is an ODE in the Sobolev space $H^s(\R), s>1/2$, which by standard arguments yields the local well-posedness of the Cauchy problem in $H^s(\R), s>1/2$.  When $\alpha=1$ (the BBM version of the modified Benjamin-Ono equation), the conservation of energy and an ODE argument (see\cite{JCS}) {\it \`a la Brezis-Gallouet} implies that this solution is in fact global.
  
  The situation is more delicate when $0<\alpha<1.$ Concerning the local 
  Cauchy problem, the local well-posedness in   $H^s(\R), s>3/2$ is 
  trivial. As in the quadratic case (see \cite{LPS2})  a local theory 
  in $H^{3/2-s_\alpha}(\R), s_\alpha>0$ can be carried out  but we 
  will focus here on the global issues. Although the  global 
  existence of small solutions is expected, things are less clear for 
  the  behavior of large solutions (global existence versus finite 
  time blow-up), and we will rely on numerical simulations, for 
  $0<\alpha<1/3$ and $1/3<\alpha<1,$ for both signs in (\ref{fBBM}).

  Equation (\ref{fBBM}) has solitary waves of the form 
  $u(x,t)=\tilde{Q}(\xi)$, $\xi=x-ct$ satisfying 
  \begin{equation}\label{QfBBM}
  	(1-c)\tilde{Q}-c|D|^{\alpha}\tilde{Q}\pm 
	\frac{1}{3}\tilde{Q}^{3}=0.
  \end{equation}
  One has for $c>1$ and the $+$ sign in (\ref{QfBBM})
  \begin{equation}\label{Q+}
  	\tilde{Q}(x) = \sqrt{c-1}Q(\beta^{1/\alpha}x),\quad \beta = \frac{c-1}{c},
  \end{equation}
  where $Q$ is the solution of (\ref{sol}) for $c=1$, and for $c<-1$ 
  and the $-$ sign in (\ref{QfBBM})
  \begin{equation}\label{Q-}
  	\tilde{Q}(x) = \sqrt{1-c}Q(\beta^{1/\alpha}x),\quad \beta = \frac{c-1}{c}.
  \end{equation}

  Thus for equation (\ref{fBBM}) solitary waves are expected for both 
  signs in front of the nonlinearity. They are propagating to the 
  right for the $+$ sign, and to the left for the $-$ sign.

  \begin{remark}
  We refer to \cite{pava} for a rather complete analysis of the stability of solitary waves to the fBBM equation:
  
  $$u_t+u_x+|D|^\alpha u_t+uux=0.$$
   \end{remark}
  
 For the numerical computations below we use 
  $N=2^{16}$ Fourier modes for the shown domains in $x$ and $N_{t}=10^{4}$ 
  time steps for the given time intervals. Note that in all cases 
  discussed below, initial data of small enough mass are simply 
  dispersed towards infinity. We do not show examples for this since 
  they are very similar to what has been shown in this context 
  before. 
  
  As a first example we address the case 
  $\alpha=0.8$ where solitary waves are known to exist. We consider 
  the initial data $u(x,0)=5\exp(-x^{2})$. The solution to 
  (\ref{fBBM}) for the $+$ 
  sign can be seen in Fig.~\ref{mfBBMa08_5gausswater}. 
  A larger solitary wave appears to form and is propagating to the right. 
  \begin{figure}[htb!]
  \includegraphics[width=0.7\textwidth]{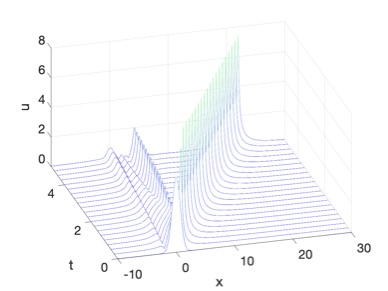}
 \caption{Solution to equation (\ref{fBBM}) with the $+$ sign for 
 $\alpha=0.8$  for 
 the initial data $u(x,0)=5\exp(-x^{2})$.}
 \label{mfBBMa08_5gausswater}
\end{figure}

The $L^{\infty}$ norm of the solution on the left of 
Fig.~\ref{mfBBMa08_5gaussinf} is in accordance with the 
interpretation that at least one large solitary wave appears. On the 
right of the same figure we show the solution at the final time 
together with a fitted solitary wave according to (\ref{Q+}). It can 
be seen that the agreement is already excellent. 
\begin{figure}[htb!]
  \includegraphics[width=0.49\textwidth]{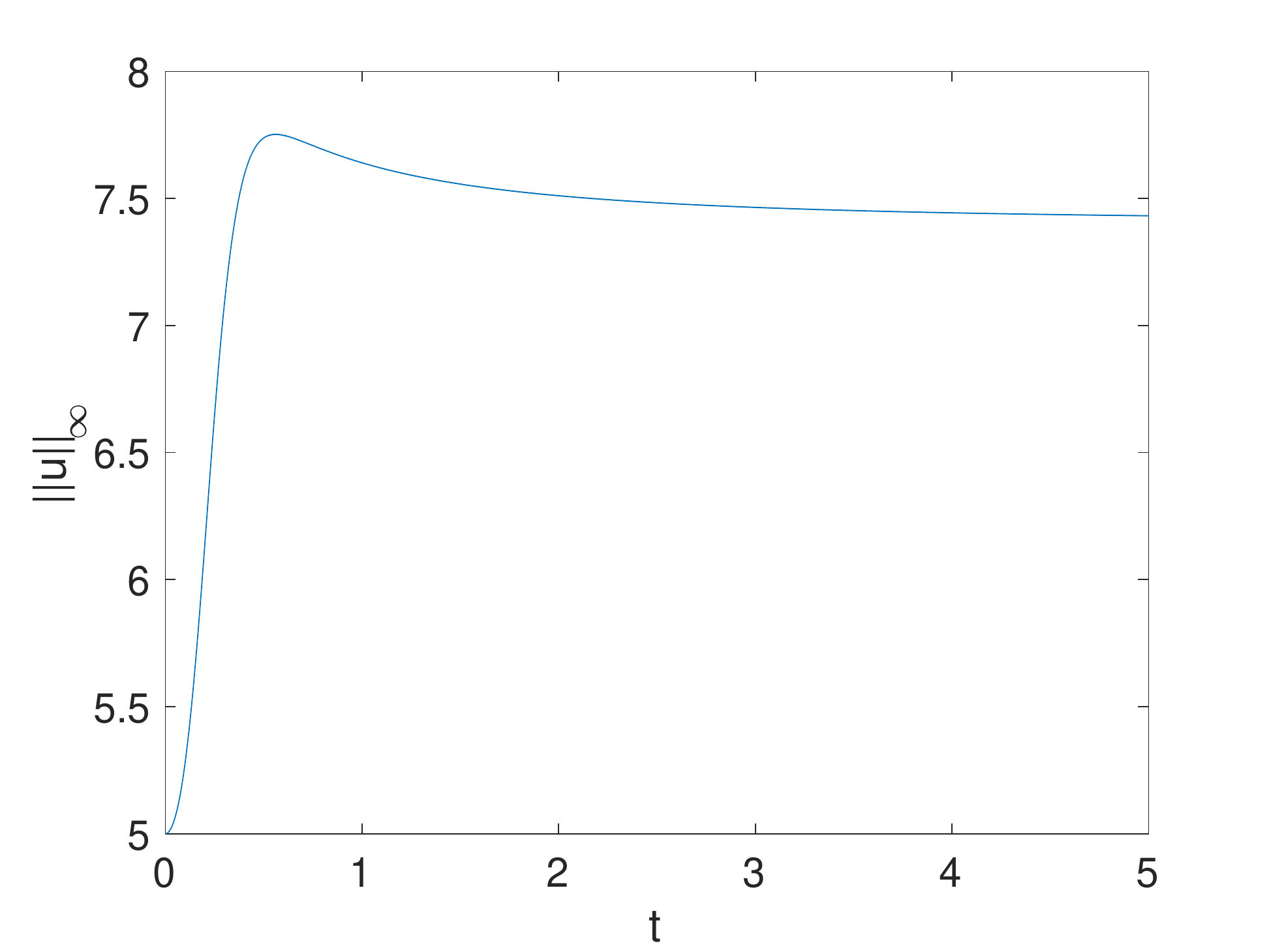}
  \includegraphics[width=0.49\textwidth]{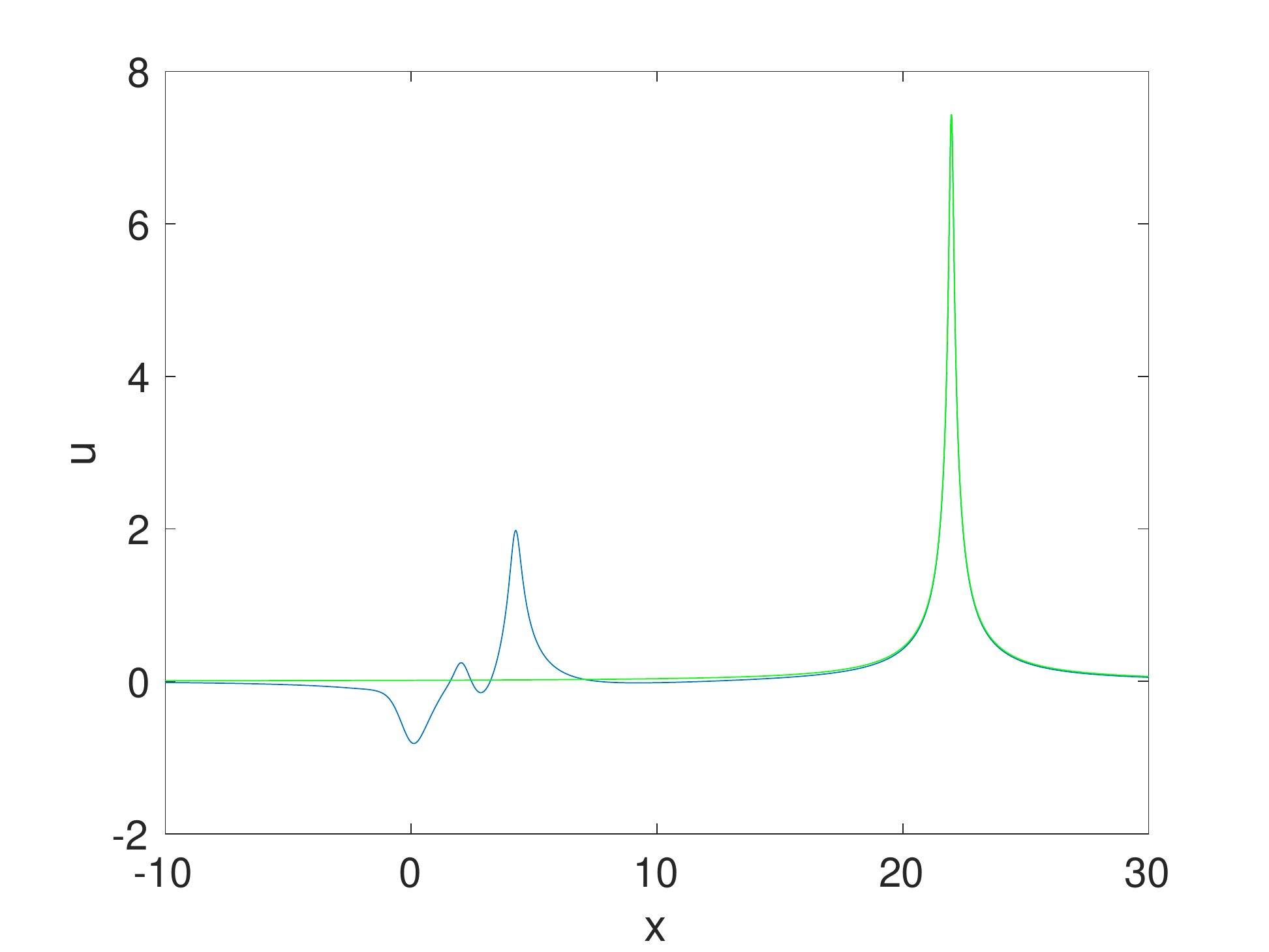}
 \caption{$L^{\infty}$ norm of the solution to equation (\ref{fBBM}) with the $+$ sign for 
 $\alpha=0.8$  for 
 the initial data $u(x,0)=5\exp(-x^{2})$  on the left, and the 
 solution at the final time together with a fitted solitary wave 
 according to (\ref{Q+}) in green on the right.}
 \label{mfBBMa08_5gaussinf}
\end{figure}

The same initial data and $\alpha=0.8$ for the $-$ sign in 
(\ref{fBBM}) lead to the solution in 
Fig.~\ref{mfBBMma08_5gausswater}. Again at least one solitary wave forms 
which travels to the left this time. 
  \begin{figure}[htb!]
  \includegraphics[width=0.7\textwidth]{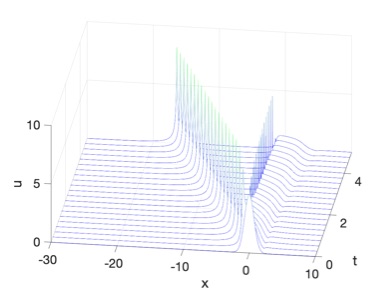}
 \caption{Solution to equation (\ref{fBBM}) with the $-$ sign for 
 $\alpha=0.8$  for 
 the initial data $u(x,0)=5\exp(-x^{2})$.}
 \label{mfBBMma08_5gausswater}
\end{figure}

This is once more confirmed by the $L^{\infty}$ norm of the solution 
on the left of Fig.~\ref{mfBBMma08_5gaussinf}, and even more so by 
the solution at the final time on the right of the same figure 
together with a fitted solitary wave according to (\ref{Q-}). This 
shows that the soliton resolution conjecture applies to both the 
equation (\ref{fBBM}) with the plus and the minus sign. In particular 
this implies that the solitons are stable for $\alpha>0.5$. 
\begin{figure}[htb!]
  \includegraphics[width=0.49\textwidth]{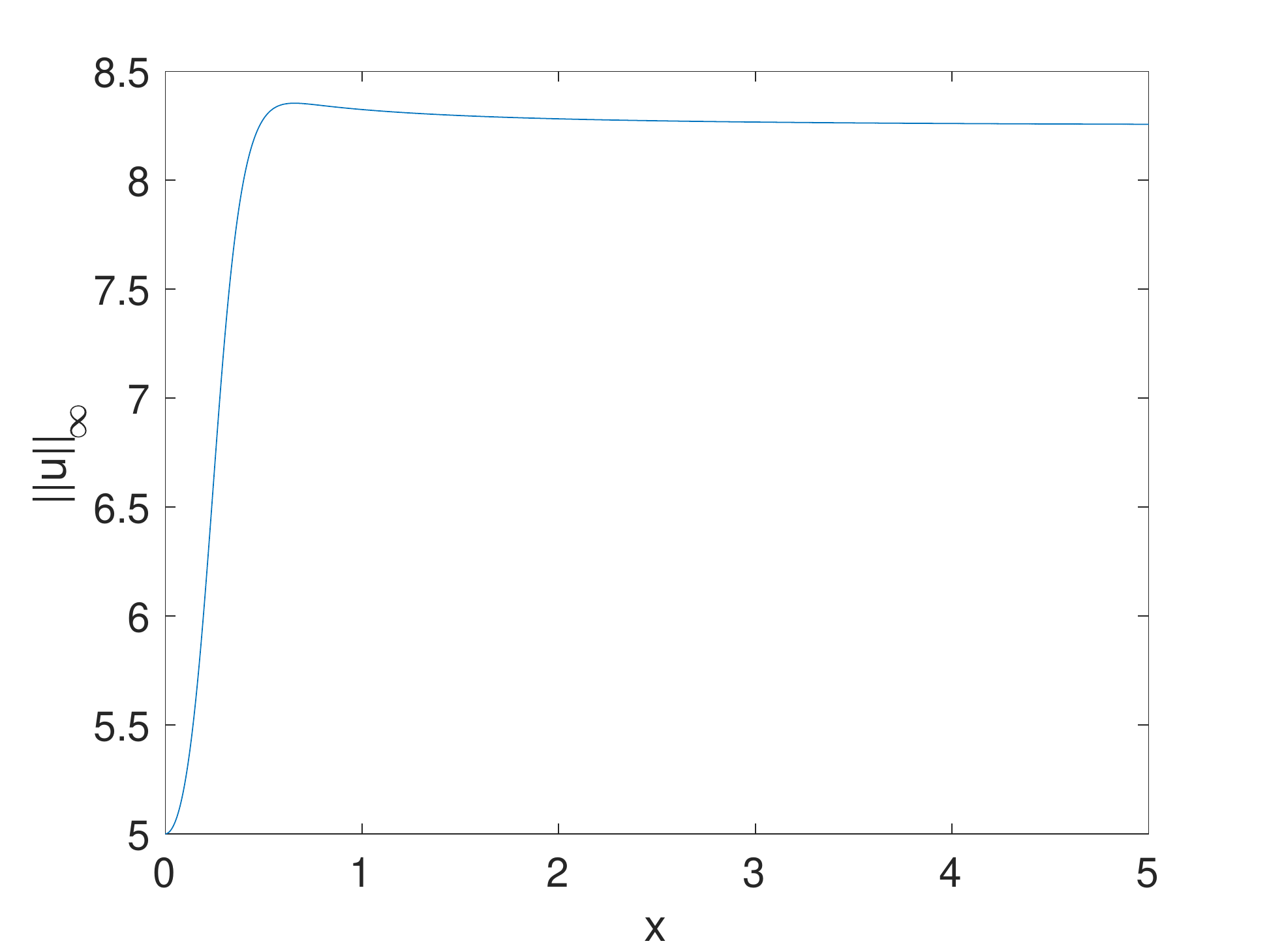}
  \includegraphics[width=0.49\textwidth]{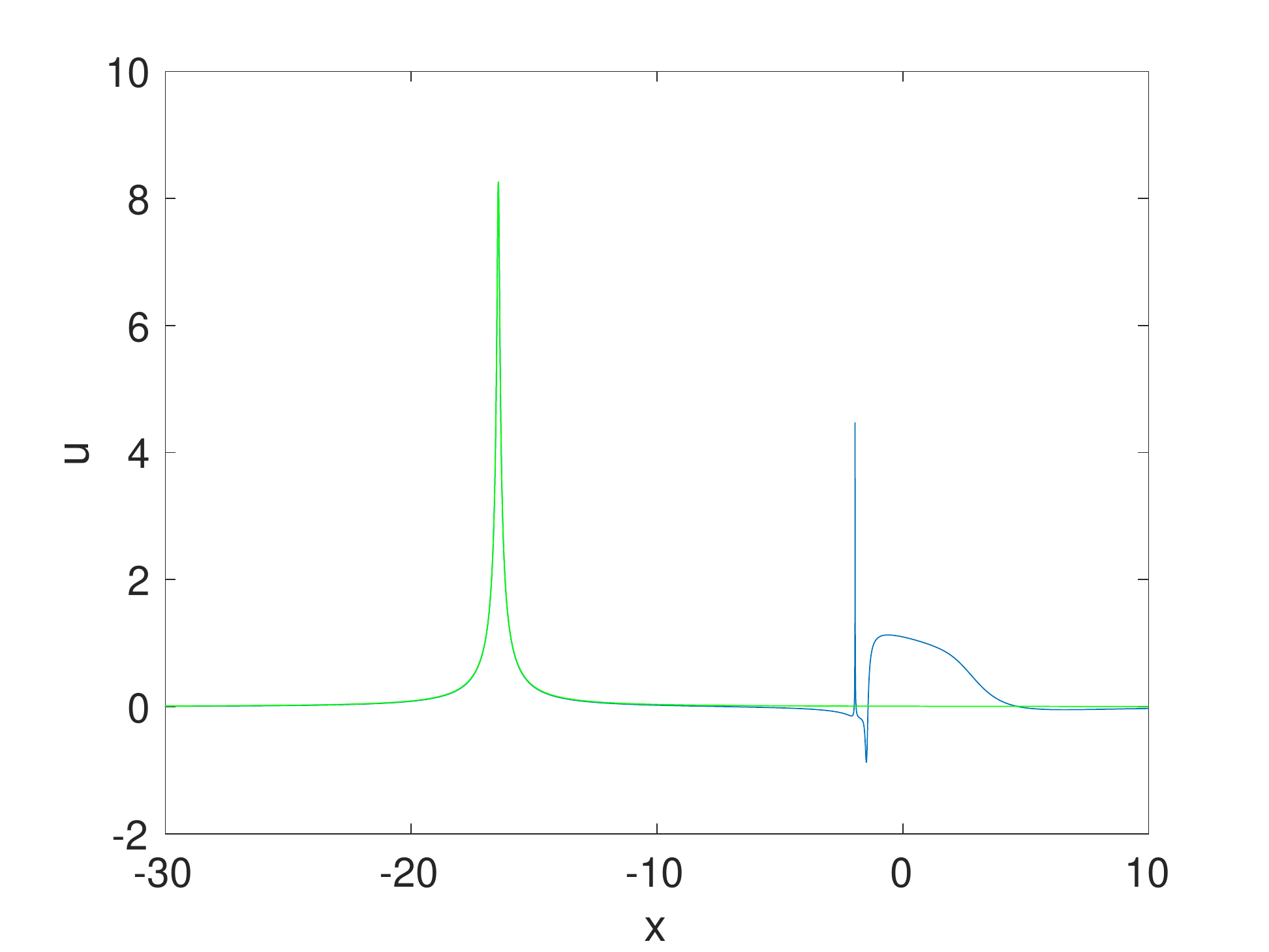}
 \caption{$L^{\infty}$ norm of the solution to equation (\ref{fBBM}) 
 with the $-$ sign for  $\alpha=0.8$  for 
 the initial data $u(x,0)=5\exp(-x^{2})$  on the left, and the 
 solution at the final time together with a fitted solitary wave 
 according to (\ref{Q-}) in green on the right.}
 \label{mfBBMma08_5gaussinf}
\end{figure}

  For $0<\alpha<1/3$, we consider as an example $\alpha=0.2$ and the 
  initial data $u(x,0)=\exp(-x^{2})$.  We find that the code breaks for $t = 
  2.5713$ since a fit of the Fourier coefficients according to 
  (\ref{mufit}) indicates the formation of a singularity with 
  $\mu\sim0.19$. The solution at this time can be seen in 
  Fig.~\ref{mfBBMalpha02gauss} on the left. The $L^{\infty}$ norm of the 
  solution in dependence of time is shown on the right of the same 
  figure. It grows until the code breaks, but as mentioned the 
  Fourier coefficients indicate that a cusp appears at this time with 
  finite $L^{\infty}$ norm. This is similar to what was observed for 
  the fBBM equation in \cite{KS}. 
 \begin{figure}[htb!]
  \includegraphics[width=0.49\textwidth]{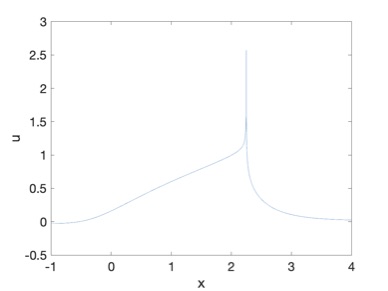}
  \includegraphics[width=0.49\textwidth]{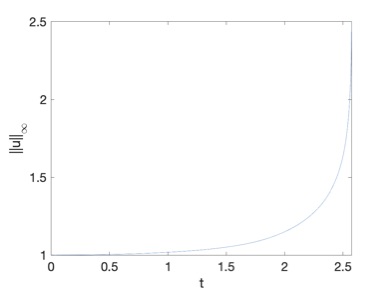}
 \caption{Solution to equation (\ref{fBBM}) with the $+$ sign for 
 $\alpha=0.2$  for 
 the initial data $u(x,0)=\exp(-x^{2})$ for $t=2.5713$ on the left, 
 and the $L^{\infty}$ norm in dependence of time   on the right.}
 \label{mfBBMalpha02gauss}
\end{figure}

The situation for the same initial data is different in the case of 
(\ref{fBBM}) with the $-$ sign. Here we get a 
dispersive shock wave as can be seen in 
Fig.~\ref{mfBBMalpham02gauss}. There is no indication of the 
formation of a singularity in this case. 
\begin{figure}[htb!]
  \includegraphics[width=0.49\textwidth]{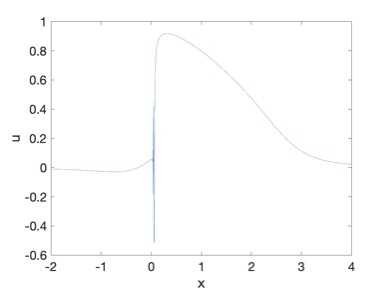}
  \includegraphics[width=0.49\textwidth]{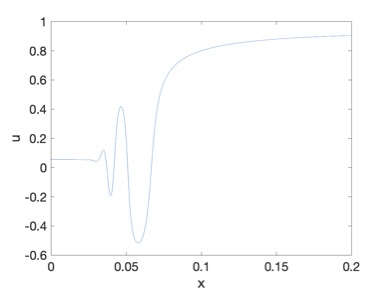}
 \caption{Solution to equation (\ref{fBBM}) with the $-$ sign for 
 $\alpha=0.2$  for 
 the initial data $u(x,0)=\exp(-x^{2})$ for $t=3$ on the left, and a 
 close-up of the DSW on the right.}
 \label{mfBBMalpham02gauss}
\end{figure}

However, for initial data with slightly larger norm, the situation is 
as in the focusing case in Fig.~\ref{mfBBMalpha02gauss}. If we 
consider the initial data $u(x,0)=1.5\exp(-x^{2)}$, then the code breaks 
for $t=1.0174$ since a fitting of the Fourier coefficients according 
to (\ref{mufit}) indicates a cusp ($\mu\sim0.17$). The solution for 
$t=1.014$ can be seen on the left of Fig.~\ref{mfBBMalpha0215gauss}. 
The $L^{\infty}$ norm of the gradient on the right of the same figure 
also indicates the formation of a cusp.
\begin{figure}[htb!]
  \includegraphics[width=0.49\textwidth]{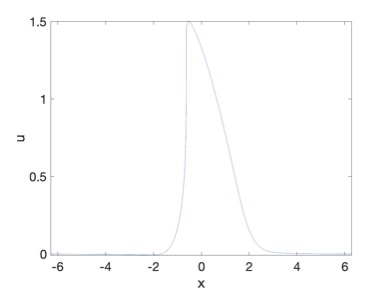}
  \includegraphics[width=0.49\textwidth]{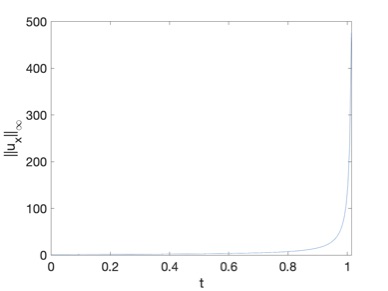}
 \caption{Solution to equation (\ref{fBBM}) with the $-$ sign for 
 $\alpha=0.2$  for 
 the initial data $u(x,0)=1.5\exp(-x^{2})$ for $t=1.014$ on the left, 
 and the $L^{\infty}$ norm of $u_{x}$ in dependence of time   on the right.}
 \label{mfBBMalpha0215gauss}
\end{figure}

 The above results can be summarized in the following 
\begin{conjecture}
	Consider smooth initial data $u_{0}\in L^{2}(\mathbb{R})$ with a single
hump for the equation (\ref{fBBM}). Then
\begin{itemize}
	\item  Initial data of sufficiently small mass will be dispersed, 
	the solutions are global in time for $0<\alpha<1$.

	\item  For $\alpha\geq1/3$, the solutions are global in time. 
	\item For $\alpha>1/2$ the long time behavior of initial data of  sufficiently large  mass is 
	characterized by solitary waves and radiation. 
	\item  For $\alpha<1/3$, initial data of sufficiently large mass 
	can lead to the formation of a cusp in finite time. 

\end{itemize}
\end{conjecture}

 \vspace{0.5cm}
\noindent {\bf Acknowledgments.} This work   was partially  supported 
by the ANR project ANuI (ANR-17-CE40-0035-02).  CK's work is partially supported by 
the isite BFC project 
NAANoD, the EIPHI Graduate School (contract ANR-17-EURE-0002), by the 
European Union Horizon 2020 research and innovation program under the 
Marie Sklodowska-Curie RISE 2017 grant agreement no. 778010 IPaDEGAN, 
and the EITAG project funded by the FEDER de Bourgogne, the region 
Bourgogne-Franche-Comt\'e and the EUR EIPHI.


\begin{thebibliography}{000}

\bibitem{AlMu}{\sc M.A. Alejo and C. Munoz}, {\it Nonlinear stability of MKdV breathers}, Commun. Math. Phys. {\bf 324} (2013), 233-262.
	
	\bibitem{BH}
	{\sc J. Biello and J. Hunter}, {\it Nonlinear Hamiltonian waves with constant frequency and surface waves on vorticity discontinuity}, Comm. Pure Appl. Math., {\bf 63} (2009), 303-336.
	
	
	\bibitem{CCG}
	{\sc A. Castro, D. C\'ordoba and F. Gancedo},
	{\it Singularity formation in a surface wave model}, 
	Nonlinearity, {\bf 23} (2010), 2835-2847.
	
	
	\bibitem{CE}
	{\sc A. Constantin and J. Escher}, {\it Wave breaking for nonlinear nonlocal shallow water equations}, Acta Math., {\bf 181} (1998), 229-245. 

	
	\bibitem{EWah}	
	{\sc M.Ehrnstr\"{o}m and E. Wahl\'en}, {\it On Whitham's conjecture of a highest cusped wave for a nonlocal shallow water wave equation}, Ann. Inst. H. Poincar\'e Anal. Non Lin\'eaire, {\bf 36} (2019), 1603-1637.	
	
	\bibitem{EW1}
	{\sc M.Ehrnstr\"{o}m and Y. Wang}, {\it Enhanced existence time of solutions to the fractional Korteweg-de Vries equation},  SIAM J. Math. Anal., {\bf 51} (2019), 3298-3323.
	
	
	\bibitem{EW2}
	{\sc M.Ehrnstr\"{o}m and Y. Wang}, {\it Enhanced existence time of solutions to evolution equations of Whitham type}, arXiv:2008.12722, (2020).
	
	
	
	\bibitem{FL}
	{\sc R.L. Frank and E. Lenzmann}, {\it On the uniqueness and nondegeneracy of ground states of $(-\Delta)^s Q+Q-Q^{\alpha +1}=0\; \text{in}\; \R$}, Acta Math., {\bf 210} (2) (2013), 261-318.
	
	
	
	\bibitem{HT}
	{\sc V. Hur and L. Tao}, {\it Wave breaking for the Whitham equation with fractional dispersion}, Nonlinearity, {\bf 27} (2014), 2937-2949.
	
	
	\bibitem{Hur}
	{\sc V. Hur}, {\it Wave breaking in the Whitham equation}, Adv. Math., {\bf 317} (2017), 410-437.
	
	
	\bibitem{KLPS}
	{\sc C. Klein, F. Linares, D. Pilod and J.-C. Saut}, {\it On Whitham and related equations},  Studies in Appl. Math.,  {\bf 140} (2018), 133-177.
	
	\bibitem{KP}
	{\sc C. Klein and R. Peter}, {\it Numerical study of blow-up in solutions to generalized Korteweg-de Vries equations}, Phys. D, {\bf 304-305} (2015), 52-78.
	
	
	\bibitem{KK}
	{\sc C.E. Kenig and K.D. Koenig}, {\it On the local well-posedness of the Benjamin-Ono and modified Benjamin-Ono equations}, Math. Res. Letters, {\bf 10} (2003), 879-895.
		
		
	\bibitem{KMR} 
	{\sc C.E.  Kenig, Y. Martel and L. Robbiano}, 
	{\it Local well-posedness and blow-up in the energy space for a class of
	$L^2$ critical dispersion generalized Benjamin-Ono equations},
	Ann. Inst. H. Poincar\'e Anal. Non Lin\'eaire, {\bf 28} (2011), 853-887.	
		
		
	\bibitem{KT}
	{\sc C.E. Kenig and T. Takaoka}, {\it Global well-posedness of the modified Benjamin-Ono equation with initial data in $H^{1/2}$}, Int. Math. Res. Not., {\bf 2006} (2006), 1-44.	
		
		
	\bibitem{KR} 
	{\sc C. Klein, K. Roidot}, {\it Numerical study of shock formation in the dispersionless Kadomtsev- Petviashvili equation and dispersive regularizations}, Phys. D, {\bf 265} (2013), https://doi.org/10.1016/j.physd.2013.09.005.
	
	
	\bibitem{KSM}
	{\sc C. Klein, C. Sparber and P. Markowich}, {\it Numerical study of fractional Nonlinear Schr\"odinger equations},
	Proc. R. Soc. A  {\bf 470}: 20140364, http://doi.org/10.1098/rspa.2014.0364.
	
	
	\bibitem{KS}
	{\sc C. Klein, and J.-C. Saut}, {\it A numerical approach to blow-up issues for dispersive perturbations of Burgers' equation}, Phys. D, {\bf 295/296} (2015), 46-65.
	
	
	
	\bibitem{LPS2}
	{\sc F. Linares, D. Pilod, and J.-C. Saut},
	{\it Dispersive perturbations of Burgers and hyperbolic equations I: local
		theory}, SIAM J. Math. Anal., {\bf 46} (2014), 1505-1537.
	
	
	\bibitem{LPS3}
	{\sc F. Linares, D. Pilod and J.-C. Saut}, {\it Remarks on the orbital stability of ground state solutions of fKdV and related equations},  Adv. Differ. Equ., {\bf 20} (9/10), (2015), 835-858.
	
	
	\bibitem{MP}
	{\sc Y. Martel and D. Pilod}, {\it Construction of a minimal mass blow-up solution to the modified Benjamin-Ono equation}, Math. Annal., {\bf 369} (1-2) (2017), 153-245.	
		
		
	\bibitem{Men}
	{\sc A.J. Mendez}, {\it On the propagation of regularity for solutions of fractional Korteweg-de Vries equations}, J. Differ. Equ., {\bf 269} (2020), 9051-9O89.
	
	\bibitem{MPV} 
	{\sc L. Molinet, D. Pilod and S. Vento}, {\it On well-posedness for some dispersive perturbations of the Burgers equation}, Ann. Inst. H. Poincar\'e Anal. Non Lin\'eaire, {\bf 35} (2018), 1719-1756. 
	
	\bibitem{MR}
	{\sc L. Molinet and F. Ribaud}, {\it Well-posedness for the generalized Benjamin-Ono equation with arbitrary large initial data}, Int. Math. Res. Not., {\bf 70} (204), 3757-3795.
	
	%\bibitem{JCS}\textsc{J.-C.Saut}, {\it Sur quelques g\' en\' eralisations de l'\' equation de KdV I}, J. Math. Pures Appl. {\bf 58} (1979), 21-61.
	
	\bibitem{NS}
	{\sc P. I. Naumkin and I. A. Shishmar\"{e}v}, {\it Nonlinear nonlocal equations in the theory of waves}, Translations of Mathematical Monographs {\bf 133} (1994), American Mathematical Society, Providence, RI, Translated from the Russian manuscript by Boris Gommerstadt.
	
	\bibitem{pava} \textsc{J. A. Pava}, {\it Stability properties of solitary waves for fractional KdV and BBM equation}, Nonlinearity {\bf 31} (2018), 920-956.
	
	\bibitem{JCS}\textsc{J.-C.Saut}, {\it Sur quelques g\' en\' eralisations de l'\' equation de KdV I}, J. Math. Pures Appl. {\bf 58} (1979), 21-61.
	

	
	\bibitem{SW1}
	{\sc J.-C. Saut and Y. Wang}, {\it Long time behavior of the fractional Korteweg-de Vries equation with cubic nonlinearity}, Discrete Contin. Dyn. Syst.,  https://doi.org/10.3934/dcds.2020312, arXiv:2003.05910, (2020).
	
	\bibitem{SW2}
	{\sc J.-C. Saut and Y. Wang}, 
	{\it The wave breaking for Whitham-type equations revisited},  arXiv:2006.03803, (2020). 
	
	\bibitem{SW3}
	{\sc J.-C. Saut and Y. Wang}, {\it In preparation}.	
		
		
	\bibitem{SSF}
	{\sc C. Sulem, P. Sulem, and H. Frisch}, {\it Tracing complex 
	singularities with spectral methods}, J. Comp. Phys., {\bf 50} (1983), 
	138-161.  
	
	\bibitem{W}
	{\sc G.B. Whitham}, {\it Linear and nonlinear waves}, Wiley, New York 1974.
	
	\bibitem{Whi}
	{\sc G. B. Whitham}, {\it Variational methods and applications to water waves}, Proc.R. Soc. Lond. Ser. A., {\bf 299} (1967), 6-25.	
		
		
\end{thebibliography}
\end{document}